\documentclass [10pt, final]{article}
\usepackage {a4}
\usepackage{authblk}
\usepackage {amsfonts}
\usepackage {amssymb,amsmath,amsthm}
\usepackage [utf8]{inputenc}
\usepackage{subcaption}
\usepackage{hyperref}
\usepackage[inline]{showlabels}
\usepackage [english]{babel}
\usepackage{enumerate}
\usepackage{graphicx}
\usepackage{fullpage}
\usepackage{bbm}
\usepackage{hyperref}
\usepackage{todonotes}
\usepackage{ifdraft}
\usepackage{verbatim}
\usepackage{afterpage}
\usepackage{url}

\newcommand{\id}{\mathbbm{1}}

\newcommand{\R}{\mathbb{R}}

\newcommand{\Pp}{\mathbb{P}}
\newcommand{\DD}{\mathcal{D}}
\newcommand{\RR}{\mathcal{R}}

\newcommand{\PP}{\mathcal{P}}
\newcommand{\EE}{\mathcal{E}}

\newcommand{\E}{\mathbb{E}}

\newcommand{\N}{\mathbb{N}}
\newcommand{\Ff}{\mathcal{F}}

\newcommand{\Kk}{\mathcal{K}}
\newcommand{\X}{\mathcal{X}}
\newcommand{\Zz}{\mathcal{Z}}
\newcommand{\core}{G^{\copyright}}
\newcommand{\pr}[1]{\mathbb{P}\left( #1 \right)}
\newcommand{\p}{\mathbb{P}}

\newcommand{\diam}{\mathrm{diam}}
\newcommand*{\ind}[1]{\mathbbm{1}_{\{#1\}}}
\newcommand*{\Ind}[1]{\mathbbm{1}_{#1}}

\newcommand{\rbr}[1]{\left(#1\right)}
\newcommand{\orb}{\mathrm{orb}}
\newcommand{\cbr}[1]{\{#1\}}

\newcommand{\edges}[1]{#1^{\uparrow}}

\newcommand{\crw}{\mathcal{X}}
\newcommand{\filG}{\mathcal{G}}
\newcommand{\filF}{\mathcal{F}}
\newcommand{\badSet}{\Upsilon}
\newcommand{\whp}{with high probability}

\newcommand{\wsp}{C e^{-c\log^2 n}}

\newcommand{\vx}[2]{\mathcal{C}_{#1}(#2)}
\newcommand{\vg}[2]{\mathcal{G}_{#1}(#2)}
\newcommand{\cE}{\mathbf{E}}
\newcommand{\cp}{\mathbf{P}}
\newcommand{\fil}{\mathcal{F}}
\newcommand{\interval}{I}

\def\fff#1{{\hyperlink{#1}{\pageref*{#1}}}}
\def\hfff#1{\label{#1}\hypertarget{#1}}

\newtheorem{theorem}{Theorem}[section]

\newtheorem{lemma}[theorem]{Lemma}
\newtheorem{assumption}[theorem]{Assumption}
\newtheorem{corollary}[theorem]{Corollary}
\newtheorem{cor}[theorem]{Corollary}
\newtheorem{proposition}[theorem]{Proposition}

\newtheorem*{mainresult}{Main result}
\theoremstyle{remark}

\title{Phase transition for the interchange and quantum Heisenberg models on the Hamming graph}
\author[$\star$]{Radosław Adamczak}
\author[$\dagger$]{Michał Kotowski}
\author[$\dagger$]{Piotr Miłoś}

\affil[$\dagger$]{Institute of Mathematics of the Polish Academy of Sciences}
\affil[$\star$]{Institute of Mathematics, University of Warsaw}

\begin{document}
	
\maketitle

\abstract{
We study a family of random permutation models on the Hamming graph $H(2,n)$ (i.e., the $2$-fold Cartesian product of complete graphs), containing the interchange process and the cycle-weighted interchange process with parameter $\theta > 0$. This family contains the random walk representation of the quantum Heisenberg ferromagnet. We show that in these models the cycle structure of permutations undergoes a \textit{phase transition} -- when the number of transpositions defining the permutation is $\leq c n^2$, for small enough $c > 0$, all cycles are microscopic, while for more than $\geq C n^2$ transpositions, for large enough $C > 0$, macroscopic cycles emerge with high probability.

We provide bounds on values $C,c$ depending on the parameter $\theta$ of the model, in particular for the interchange process we pinpoint exactly the critical time of the phase transition. Our results imply also the existence of a phase transition in the quantum Heisenberg ferromagnet on $H(2,n)$, namely for low enough temperatures spontaneous magnetization occurs, while it is not the case for high temperatures.

At the core of our approach is a novel application of the cyclic random walk, which might be of independent interest. By analyzing explorations of the cyclic random walk, we show that sufficiently long cycles of a random permutation are uniformly spread on the graph, which makes it possible to compare our models to the mean-field case, i.e., the interchange process on the complete graph, extending the approach used earlier by Schramm.

\tableofcontents

\begin{section}{Introduction}

In this paper we investigate the cycle structure of random permutations in the interchange process (sometimes called the random stirring process) and its generalizations. For a finite graph $G = (V,E)$ the interchange process $\sigma = (\sigma_{t})_{t \geq 0}$ on $G$ is defined as follows: put particles numbered from $1$ to $|V|$ on vertices of the graph and Poisson clocks of rate $1 / |E|$ on each edge. Whenever the clock on an edge $e \in E$ rings, the particles at the endpoints of $e$ are swapped. In this way for each $t \geq 0$ we obtain a permutation $\sigma_t \colon V \to V$, which is determined by the sequence of transpositions corresponding to swaps occurring up to time $t$.

The model has attracted considerable attention, in particular one is interested in how the cycle structure of $\sigma_t$ changes with $t$, especially in the asymptotic case where $G = G_n$ belongs to a family of graphs with $|V| = n \to \infty$. The starting point of our work is a remarkable result due to Schramm (\cite{Schramm:2011fk}), which shows that in the case of the complete graph $G = K_n$ the model exhibits a phase transition. Suppose the interchange process is run for time $c n$, then if $c > 1/2$, the resulting permutation will, with high probability, contain a macroscopic cycle (i.e., of size comparable to $n$), while for $c < 1/2$ all cycles will have size $o(n)$. Furthermore, for $c > 1 /2$ after proper rescaling the joint distribution of macroscopic cycle sizes converges to the Poisson-Dirichlet distribution with parameter $1$ (which is also the limiting distribution of macroscopic cycles for permutations chosen from the uniform measure on $S_n$). This should be contrasted with the classical result of Diaconis and Shahshahani (\cite{Diaconis:1981aa}) that the mixing time of the random transposition process on $K_n$ is $\frac{1}{2} n \log n$, in particular Schramm's result shows that long cycles equilibrate long before the distribution of the whole permutation.

Our main interest is twofold -- first, to move beyond the complete graph and extend these results to graphs with non-trivial geometry, and second, to obtain similar results for a certain generalization of the interchange process, motivated by studies of models in statistical physics.

Namely, we will be interested in the \emph{cycle-weighted interchange process}, depending on an additional parameter $\theta > 0$, in which the probability of a sequence of transpositions is weighted depending on the number of cycles in the resulting permutation (a more precise definition will be given shortly). The physical importance of this model is that for $\theta = 2$ it corresponds to the random walk representation of the quantum Heisenberg ferromagnet.

We will first state our main result informally and just for the case of the interchange process, with more precise statements given afterwards.

\begin{mainresult}
Let $G$ be the Hamming graph $H(2,n) = K_n \times K_n$, where $K_n$ is the complete graph on $n$ vertices. Consider the interchange process on $G$ run up to time $t = \beta n^2$. Then the permutation $\sigma_t$ obtained at time $t$ exhibits a phase transition: for $\beta > 1/2$ with high probability $\sigma_t$ contains macroscopic cycles, while for $\beta < 1/2$ all cycles of $\sigma_t$ are with high probability of size $o(n)$.
\end{mainresult}

We will now introduce the more general setup, which will allow us to formulate our results rigorously.

We consider the \hfff{def:hammingGraph}{Hamming graph $H = H_n = H(2,n) = (V,E)$}. The vertices $V = \{0,\dots,n-1\}^2$ are given by a subset of the square lattice and an edge is present between a pair of vertices if they are either in the same row or the same column, where for $i\in\{0,\ldots,n-1\}$ the sets $L_i = \{0,\dots,n-1\}\times \{i\}$ are called rows and $D_i = \{i\}\times \{0,\ldots, n-1\}$ columns. One can check that $|V| = n^2$ and $|E| = n^2(n-1)$. In the whole paper, we assume implicitly that $n\geq 2$.

Let \hfff{def:configurations}{$\mathfrak{X}$ be the space of finite subsets of $E\times [0,1)$}, which we will call \emph{configurations}. Given $X \in \mathfrak{X}$ we denote by \hfff{def:edges}{$\edges{X}$ the sequence $(e_1, \ldots, e_{|X|})$, where $(e_1, t_1), \ldots, (e_{|X|}, t_{|X|})$} are all points of $X$ ordered with respect to the second coordinate (and an arbitrary fixed order on $E$ if $t_i = t_j$). We define $\sigma(X)$, a permutation associated with the configuration $X$, by
\begin{equation}\label{eq:permutation_associated}
	\sigma(X) := {e_{|X|}} \circ {e_{|X|-1}} \circ \ldots \circ {e_1},
\end{equation}
where any edge $e_i \in E$ is identified with the transposition of its endpoints.

We call a function $\mathcal{C} \colon \mathfrak{X}\mapsto \R_{+}$ \emph{admissible} if for any $X,Y\in \mathfrak{X}$ we have $\mathcal{C}(X) = \mathcal{C}(Y)$ whenever $\edges{X} = \edges{Y}$ and the following Lipschitz condition holds
\begin{equation}\label{eq:Lipschitz}
	|\mathcal{C}(X) - \mathcal{C}(Y) | \leq |X \triangle Y|,
\end{equation}
where $\triangle$ is the symmetric difference of sets.

Fix $\beta > 0$. \hfff{def:PPPbridges}{Let $\mathcal{B}$ be the law of a Poisson point process} on $E\times [0,1)$ with intensity measure given by $\frac{\beta}{n-1}\#(\cdot)\otimes{\rm Leb}$, where $\#(\cdot)$ is the counting measure and ${\rm Leb}$ is the Lebesgue measure. Here we treat the Poisson process as a random element of $\mathfrak{X}$ (we endow this space with a $\sigma$-field $\mathcal{S}$, being a completion with respect to the measure $\mathcal{B}$ of the $\sigma$-field generated by all functions of the form $X \mapsto X \cap A$, where $A \in 2^E \otimes Bor([0,1))$).

By rescaling time and changing intensity on the edges to $1 / |E|$ one easily sees that if a configuration $X$ is drawn at random according to $\mathcal{B}$, the resulting sequence of transpositions has the same distribution as the interchange process on $H$ after time $\beta n^2$. In particular the intensity is chosen so that $X$ has size $|E| \cdot \frac{\beta}{n-1} = \beta n^2$ on average.

Fix $\theta > 0$ and an admissible function $\mathcal{C}$. We define a probability distribution \hfff{def:mu}{$\mu_{\beta, \theta, \mathcal{C}}$} on $\mathfrak{X}$ which will be the main object of our study
\begin{equation}\label{eq:measure}
	  \mu_{\beta, \theta, \mathcal{C}}(U ) := Z_{\beta, \theta, \mathcal{C}}^{-1} \int_{\mathfrak{X}}\Ind{U}(X) \theta^{\mathcal{C}(X)} \mathcal{B}(dX),
\end{equation}
where $Z_{\beta, \theta, \mathcal{C}} = \int_{\mathfrak{X}} \theta^{\mathcal{C}(X)} \mathcal{B}(dX)$ is the partition function normalizing the measure to $1$. Note that the condition \eqref{eq:Lipschitz} ensures that the measure is well defined. Throughout the paper we set $\Theta := \max(\theta^{-1}, \theta)$.

The process defined by $\mu_{\beta, \theta, \mathcal{C}}$ will be called the \emph{weighted interchange process} in general and we will use the name \emph{cycle-weighted interchange process} if $\mathcal{C}(X)$ is the number of cycles in $\sigma(X)$ (which is easily seen to be an admissible function). Note that $\theta = 1$ corresponds simply to the interchange process.

The main result of our paper states that for $\beta$ large enough the random permutation induced by the measure $\mu_{\beta, \theta, \mathcal{C}}$ has macroscopic cycles with high probability

\begin{theorem}\label{thm:main_theorem_upperbound}
	Let $\beta,\theta>0$ be such that $\beta>\Theta/2$ and let $\mathcal{C}$ be an admissible function. Let $X$ be randomly sampled from $\mu_{\beta, \theta, \mathcal{C}}$. Then
		\begin{equation*}
			\lim_{\varepsilon\to 0} \liminf\limits_{n\to +\infty} \p(\text{ there exists a cycle of } \sigma(X) \text{ of length at least $\varepsilon n^2$}) = 1.
		\end{equation*}
\end{theorem}
This contrasts with the situation when $\beta$ is small.
\begin{theorem}\label{thm:main_theorem_lowerbound}
	Let $\beta,\theta>0$ be such that $\beta< \Theta^{-1}/2$ and let $\mathcal{C}$ be an admissible function. Let $X$ be randomly sampled from $\mu_{\beta, \theta, \mathcal{C}}$. Then for some $C>0$
	\begin{equation*}
		\lim_{n \to +\infty} \p(\text{ all cycles of } \sigma(X) \text{ are shorter than $C\log n$}) = 1.
	\end{equation*}
\end{theorem}

Together our results imply the existence of a phase transition from microscopic cycles when $\beta <\Theta^{-1}/2$ to macroscopic ones when $\beta >\Theta/2$. We expect that the point of the phase transition is unique, possibly under some mild assumptions on $\mathcal{C}$.

The two most important cases in which our results apply are $\theta = 1$ and $\theta = 2$ (with $\mathcal{C}$ being the number of cycles).
\paragraph{Interchange process.}

In the special case of the interchange process, corresponding to $\theta = 1$, we have $\Theta^{-1}/2 = \Theta/2 = 1/2$, so the above theorems determine precisely the transition point for the occurrence of large cycles: for $\beta < 1/2$ (and large $n$) with high probability all the cycles are of logarithmic size, while for $\beta > 1/2$ we get cycles of length comparable to the size of the graph with probability arbitrarily close to one.

\paragraph{Quantum Heisenberg model.}
The case $\theta = 2$ is particularly interesting from the point of view of statistical physics, since it corresponds to the random walk representation of the quantum Heisenberg ferromagnet. In this representation, introduced by T\'{o}th in \cite{Toth1993}, the existence of macroscopic cycles translates to nonvanishing spontaneous magnetization in the model. We describe the connection very briefly here, referring the reader to the survey \cite{Goldschmidt:2011aa} for more details on this model.

The quantum Heisenberg ferromagnet is a model of a spin system whose physical properties depend on a parameter $\beta > 0$ called the \emph{inverse temperature}. One of the major questions about the model is whether a spontaneous ordering of spins occurs at low enough temperatures (corresponding to high $\beta$). \emph{T\'{o}th's  random walk representation} of the Heisenberg model is given by the measure $\mu_{\beta, \theta, \mathcal{C}}$ with $\beta$ equal to the inverse temperature, $\theta = 2$ and $\mathcal{C}(X)$ being the number of cycles in the permutation $\sigma(X)$. One can then express all physical quantities of interest in terms of the cycle-weighted interchange process given by $\mu_{\beta, \theta, \mathcal{C}}$. For example, the correlation between spins at sites $u$ and $v$ corresponds to the probability that $u$ and $v$ are in the same cycle of $\sigma(X)$ when $X$ is sampled from $\mu_{\beta, \theta, \mathcal{C}}$.

Crucially, our results in Theorem \ref{thm:main_theorem_upperbound} and Theorem \ref{thm:main_theorem_lowerbound} imply the existence of a phase transition for the model on the Hamming graph -- for $\beta < 1/4$ there is no magnetic ordering as $n \to \infty$, while it emerges for $\beta > 1$. Our methods do not imply sharpness of the phase transition, although we conjecture that it is indeed sharp, with the critical value being $\beta = 1$. It is a major open problem to determine whether a similar phase transition occurs for $G_n =  [-n, n]^d \cap \mathbb{Z}^{d}$, $d \geq 2$, as $n \to \infty$.

For a more precise relation between the existence of macroscopic cycles and the phase transition for spontaneous magnetization see Section 4 of \cite{Goldschmidt:2011aa}.

\paragraph{Outline of proof strategy.}

We now outline the main ideas behind the proof of Theorem \ref{thm:main_theorem_upperbound} (the proof of Theorem \ref{thm:main_theorem_lowerbound} is much less involved). First, however, we would like to stress that the novelty of our paper comes not only from the result itself but also from the methods developed for this purpose. Our new techniques establish a precise geometric picture of cycles which is believed to hold for a much broader family of graphs. We also note that the tools developed in this paper proved useful in the analysis of loop models related to the XXZ model on the complete graph (\cite{Bjornberg:2018aa}).

Broadly speaking, we would like to follow the approach used by Schramm for the complete graph, which consists of showing that after a long enough time cycles of mesoscopic size appear and then quickly merge into macroscopic ones. This in turn relies on analyzing a split-merge process of cycles in the complete graph, with each new transposition either causing two cycles to split or to merge. In the case of the complete graph it is easy to give an upper bound on the rate at which cycles split and a lower bound on the rate at which (long enough) cycles merge.

The key difficulty which appears on any graph with non-trivial geometry, in particular in the case of the Hamming graph $H_n$, is that, unlike on the complete graph, the split-merge probabilities depend not only on sizes of the cycles, but also on their spatial structure, more precisely on their isoperimetric properties. We are able to prove that long enough fragments of cycles on $H_n$ are typically ``uniformly spread'' on the graph, resembling an exploration of a simple random walk and thus making their isoperimetric properties easy to analyze. In particular, the split-merge probabilities (and thus the behavior of the interchange process) can be approximated by the mean-field (complete graph) case.

This intermediate result is at the core of our arguments and we believe it might be of independent interest when analyzing random transposition processes on other graphs. The crucial tool that we employ is the so-called \emph{cyclic random walk} (abbreviated by CRW), introduced in \cite{Toth1993} and later used by Angel in \cite{Angel:2003aa} under the name cyclic time random walk. This is an exploration process which, given a configuration $X$ of transpositions and a starting vertex $v \in V$, visits subsequent vertices of the cycle of $\sigma(X)$ containing $v$. The fundamental difficulty in the analysis of the CRW is that it is a non-Markovian process, involving interactions of the random walk with vertices visited in the past. The bulk of our effort is devoted to the analysis of these interactions.

We analyze the behavior of the CRW at a mesoscopic timescale which is:
\begin{itemize}
\item short enough so that the interactions with the history are tractable and it is possible to exploit methods similar to excursion theory for random walks,
\item long enough so that the trace of the CRW occupies the vertices of the graph in a uniform way, and probabilistic bounds we obtain are strong enough to extend results (by union bound) to longer timescales, including macroscopic (i.e., of the order of $n^2$).
\end{itemize}
A more detailed outline of this part of the proof is given in Section \ref{sec:isoperimetry-setting}. Once we know that the trace of the CRW with high probability occupies the graph in a uniform way, we can extend this property to cycles of $\sigma(X)$ and carry out the analysis of the corresponding split-merge process described above. We note that the assumption $\beta > \Theta / 2$ is crucial in our approach, as it enables us to show that the explorations of the CRW are sufficiently long.

Schramm's approach requires as a prerequisite the existence of mesoscopic cycles. To prove that indeed they exist with high probability, we employ a natural coupling between the random transposition process and a percolation process, with the quality of the coupling on a fixed timescale depending on the isoperimetric bound.

We use this coupling twice to obtain cycles of mesoscopic length, and then employ the argument by Schramm relying on the mean-field behavior of the split-merge probabilities. An abstract version of Schramm's argument that we use is presented in Section \ref{sec:schramm-argument} and we believe that this part of the paper might be of independent interest, as the results are formulated in a way convenient for application to general transposition processes (e.g., the interchange process on more general graphs).

It is worth noting that an additional difficulty is present in both parts of the proof in the case of models with $\theta \neq 1$, as subsequent transpositions appear there in a non-i.i.d. fashion. Our methods are based on the observation that on small timescales having $\theta\neq 1$ tilts the measure in a controllable way. Roughly speaking, adding or removing a transposition from a configuration changes the number of cycles only by one, which can change the relative probability of the configuration by at most a factor of $\Theta^2$. Thus we can compare this process with an i.i.d process, which makes the analysis of the cyclic random walk and emergence of macroscopic cycles still possible. For $\theta \neq 1$ this approach allows to analyze the size of cycles for small and large $\beta$, however it does not give the critical value of the phase transition.

We end this part with a plan of the rest of the paper. The whole Section \ref{sec:isoperimetry} is devoted to the analysis of the cyclic random walk. The main results of this part are encapsulated in Proposition \ref{prop:isoperimetry-log2n} and Proposition \ref{prop:isoperimetry-log2n-lower-bound}, providing an upper and a lower bound on the typical isoperimetry of the trace of the CRW. In Section \ref{sec:transposition_process} we flesh out the connection between the CRW and the cycles of the transposition process. The crucial property that split-merge probabilities are comparable to the mean-field case is stated in Proposition \ref{cor:splitting}. In Section \ref{sec:schramm-argument} we provide an abstract formulation of Schramm's argument regarding macroscopic cycles. It is given in Lemma \ref{le:Schramm_key_lemma} and we believe it might of independent interest. We then use it together with results from Section \ref{sec:transposition_process} to prove Theorem \ref{thm:main_theorem_upperbound}. The (much simpler) proof of Theorem \ref{thm:main_theorem_lowerbound} is given in Section \ref{sec:proof-of-subcritical}.

\paragraph{Related works.}

By now the interchange process and its generalizations have attracted considerable attention, both from the point of view of probability theory and mathematical physics. Here we mention some of the work that is most closely related to the topic of this paper.

\begin{itemize}

\item The existence of a phase transition for the appearance of macroscopic cycles in the interchange process on the complete graph, together with convergence of the law of macroscopic cycles to the Poisson-Dirichlet distribution, is due to Schramm (\cite{Schramm:2011fk}). An alternative, simpler proof of the statement that large cycles appear after $cn$ transpositions, for $c > 1/2$, was given by Berestycki (\cite{Berestycki:2011aa}).

To the best of our knowledge, the only rigorous results concerning finite graphs other than the complete graph, all in the case $\theta = 1$, are due to Koteck\'{y}, Mi\l o\'{s} and Ueltschi (\cite{Kotecky:2016aa}) and Mi\l o\'{s} and \c{S}eng\"{u}l (\cite{milos-sengul}). In \cite{Kotecky:2016aa} it is proved that on the hypercube $\{0,1\}^n$ with $N = 2^n$ vertices for any $\varepsilon > 0$ and large enough times a positive fraction of vertices is contained in cycles of length at least $N^{\frac{1}{2} - \varepsilon}$.

In \cite{milos-sengul} it is proved that on the Hamming graph $H(2,n)$ for $\beta > 1/2$ and any $\varepsilon > 0$ asymptotically almost surely a constant fraction of vertices is contained in cycles of length at least $n^{2 - \varepsilon}$. In this work we build on the approach of \cite{milos-sengul} and obtain significantly stronger results, proving the existence of truly macroscopic cycles and also considering the case $\theta \neq 1$, which, as mentioned above, poses an additional difficulty. The techniques of \cite{milos-sengul} enabled the authors to analyze only relatively short explorations of the cyclic random walk (on a timescale of the order of $n$) and thus obtain rather crude bounds on the isoperimetry of the cycles. The main improvement in our work is a detailed analysis of the structure of explorations of the cyclic random walk (in particular its interactions with its history) on a long timescale, enabling us to give tight isoperimetric bounds on the cycles' spatial structure.

\item Another approach to the analysis of the cycle structure of random permutations, based on representation theory of the symmetric group, was developed in \cite{Alon:2013aa} and \cite{Berestycki:2015aa}, with the second paper providing yet another proof of the existence of macroscopic cycles in the interchange process on the complete graph. This approach was recently extended in \cite{alon-kozma-heisenberg} to prove a sharp phase transition in the $\theta = 2$ case on the complete graph (mean-field Heisenberg ferromagnet). Another recent result by the same authors (\cite{alon-kozma-octopi}) implies that in the case $\theta = 1$ macroscopic cycles emerge with positive probability on the Hamming graph for large enough $\beta$ (note, however, that their result does not identify the critical value of $\beta$ and concerns only cycles of length at least $n^2 / 2$, instead of $\varepsilon n^2$ for any $\varepsilon > 0$).

\item In \cite{Bjornberg:2016aa} Bj\"{o}rnberg, also using representation theory, computed the free energy and the critical temperature in a family of quantum spin models on the complete graph corresponding to the cycle-weighted interchange process with $\theta = 2, 3, 4, \ldots$. This extends previous results obtained by Penrose (\cite{Penrose1991}) and T\'{o}th (\cite{Toth1990}) for $\theta = 2$.

The existence of large cycles was established for arbitary $\theta > 1$ by a different method in \cite{Bjornberg:2015aa}, where it is proved that macroscopic cycles appear on the complete graph for $\theta > 1$ as soon as $\beta > \theta$ (note, however, that apart from the case $\theta =2$ this is strictly larger than the critical value $\beta_c$ determined in \cite{Bjornberg:2016aa} for $\theta = 2, 3, 4, \ldots$).

\item The interchange process can be defined in a natural way also on infinite bounded-degree graphs. Here one asks whether infinite orbits appear almost surely when time exceeds certain critical value. It is conjectured that a phase transition occurs if the underlying infinite graph is transient. The case of a $d$-regular infinite tree was first considered by Angel (\cite{Angel:2003aa}), who proved that infinite orbits exist in an appropriate bounded time interval, and then in subsequent work by Hammond (\cite{Hammond:2013aa}, \cite{Hammond:2015aa}), where an actual phase transition was established for large enough $d$. These results were recently extended to more general random loop models (see below).

\item Another generalization of the interchange process are the so-called random loop models, corresponding to a family of quantum spin models containing, among others, the quantum Heisenberg antiferromagnet (\cite{Goldschmidt:2011aa}). Recently it has been proved by Hammond and Hegde (\cite{Hammond:2018aa}), building upon earlier work by Bj\"{o}rnberg and Ueltschi (\cite{Ueltschi:2016aa}), that there exists a phase transition for the appearance of infinite loops on the infinite $d$-regular tree for $d$ large enough. In the case of finite graphs, it is proved \cite{Bjornberg:2018aa} that on the complete graph the distribution of macroscopic loops for $\beta > 1$ converges to the Poisson-Dirichlet distribution with parameter $1/2$.

\end{itemize}

\paragraph{Further research and open questions.}
There are a number of open questions closely related to our paper. We believe that techniques we have developed here could be useful in approaching some of them.
\begin{itemize}
\item We expect that the same techniques as for the Hamming graph $H(2,n) $could be used to analyze other families of Hamming graphs $\{H(d, n)\}_{n\geq 2}$, for fixed $d\in \mathbb{N}$. Here $H(d,n)$ has vertex set $\{0, \ldots, n-1\}^d$ and an edge is present between any two vertices which differ in exactly one coordinate. In this paper we decided to focus only on the case $d=2$, so as not to obfuscate already long proofs.

On the other hand, it would be of interest to extend our results to Hamming graphs $H(d,n)$ which satisfy $d \to \infty$ (as well as possibly $n \to \infty$). An extreme example is the hypercube $\{H(d, 2)\}_{d\geq 2}$, which is interesting as the degree of each of its vertices diverges as $d \to \infty$, but only slowly (as it is logarithmic in the number of vertices of the graph). We believe that many ideas from our paper should be applicable to this case, although some new insights will also be required, as the geometry of the hypercube is more complicated that of $\{H(2, n)\}_{n\geq 2}$. We expect that understanding the hypercube would essentially enable one to analyze any Hamming graph.

\item The results of this paper do not establish the critical value of $\beta$ at which the phase transition occurs (apart from the case $\theta = 1$). It is conjectured (and partially proved, see results and discussion in \cite{Bjornberg:2016aa}) that on the complete graph the critical value is given by
\[
\beta_{c}(\theta) =
\begin{cases}
\theta &\quad \mbox{if } 0 < \theta \leq 2, \\ 2 \left(\frac{\theta - 1}{\theta - 2}\right) \log(\theta - 1) &\quad \mbox{if } \theta > 2,
\end{cases}
\]
which coincides with the critical parameter of the random-cluster model on the complete graph with $q=\theta$ (\cite{Bollobas1996}). It would be an interesting question to explore the possible connection between the two models further and determine the critical value of $\beta$ for the Hamming graph.
\item We conjecture that the properly normalized list of macroscopic cycle lengths obtained in the weighted interchange process with parameter $\theta$ should converge to the Poisson-Dirichlet distribution $PD(\theta)$. This would extend the convergence to $PD(1)$ in the case of $\theta = 1$ on the complete graph proved in \cite{Schramm:2011fk}.

\end{itemize}

\end{section}

\subsection*{Acknowledgments} We would like to thank Roman Kotecký and Daniel Ueltschi for useful discussions. We thank Wojtek Samotij for providing us with the idea of the proof of Lemma \ref{lm:core-bad-ball-is-small}. Research partially supported by the National Science Centre, Poland, grants no. 2015/18/E/ST1/00214 (RA), no. 2014/15/B/ST1/02165 (PM) and no. 2019/32/C/ST1/00525 (MK).

\newpage

\subsection*{Glossary}
To help the reader we include the glossary of notation used in the paper.
\begin{center}
	\begin{tabular}{ l l r }
	$H = H_n$ & the Hamming graph on $n^2$ vertices & \fff{def:hammingGraph} \\
	$L_i, D_i$ & rows and columns of the Hamming graph $H$ & \fff{def:hammingGraph} \\
	$\mathfrak{X}$ & the space of configurations (finite subsets of $E \times [0,1)$) & \fff{def:configurations} \\
	$\edges{X}$ & sequence of edges corresponding to a configuration $X$ & \fff{def:edges} \\
	$\mu_{\beta, \theta, \mathcal{C}}$ & measure defining the weighted interchange process with parameters $\beta, \theta, \mathcal{C}$ & \fff{def:mu} \\
	$\crw_s$ & the cyclic random walk at time $s$ & \fff{def:crw} \\
	$\crw_I, \Zz_I$ & the path and trace of the cyclic random walk on interval $I$ & \fff{def:crw_notions_interval} \\
	$\mathcal{Z}_s, Z_k$ & trace of the cyclic random walk (up to times $s$, resp. first $k$ vertices)& \fff{def:crw_notions_interval} \\
	$T_k$ & time of entering the $k$-th new vertex by the CRW  & \fff{def:crw_notions} \\
	$\mathcal{O}, \mathcal{O}_{k}$ &  the orbit of $v$ under the permutation $\sigma(X)$ and its first $k$ vertices & \fff{eq:very-big-ooo} \\
	$\filF_s, \filG_k$ & filtrations of the cyclic random walk  & \fff{def:crw_filtrations} \\
	$\iota, \chi$ & isoperimetry upper and lower bound of a given set & \fff{eq:iotaDefinition} \\
	$T$ & $T \leq n\log^2n$, timescale on which we study the trace of the CRW  & \fff{def:Ttime} \\
	$G_t$ & graph induced by the cyclic random walk  & \fff{def:gtcrw} \\
	$\badSet_{t}$ & the bad set at time $t$  & \fff{def:badset} \\
	$\tau_{iso}^\delta $ & time until which the trace of the cyclic random walk has small $\iota$  & \fff{def:tauiso} \\
	$\tau_c $ & time when the cyclic random walk closes into a cycle  & \fff{def:tauc} \\
	$I, I^b, I^d$ & processes counting internal, bad and direct jumps  & \fff{l:jumps} \\
	$\lambda, \lambda^b, \lambda^d$ & intensities of internal, bad and direct jumps  & \fff{l:intensities} \\
	$\PP_t$ & potential of the path $\crw_{[0,t)}$  & \fff{def:potential} \\
	$\mathcal{E}_t(k)$ &  event that the CRW makes an excursion of length $k$ at time $t$ & \fff{def:excursion} \\
	$\core_t$ & core of the graph $G_t$  & \fff{def:core} \\
	$\mathcal{D}^\delta_t$ & event that the graph $G_t$ does not contain too many high degree vertices & \fff{lab:event-d-t} \\
	$\mathcal{Q}_t$ & event that all straight paths of length $\log^2 n$ in $\core_t$ have good potential & \fff{eq:straight-paths} \\
	$\tau_k$ & time of the visit to the $k$-th new vertex in $L_{0}$ & \fff{eq:tau_k} \\
	$\orb_{s}(v), \orb_{s}^{\ell}(v)$ & orbit of vertex $v$ in the permutation $\sigma_{s}$ and its $\ell$ first elements & \fff{def:orbits} \\
	$\mathcal{I}$ & event that that orbits of $\sigma_t$ have good isoperimetric properties & \fff{def:assumptionI} \\
	$\vx{k}{\ell}$ & set of vertices belonging to components of $\sigma_k$ of size at least $\ell$ & \fff{def:cycleSize} \\
	$\{ G^{s}_{u} \}_{u = 0, \ldots, |X| - s}$ & random graph process coupled to $\sigma_{s+u}$ & \fff{def:graphProcess} \\
	$\vg{s, u}{\ell}$ & vertices belonging to components of size at least $\ell$ in $G^{s}_{u}$ & \fff{def:graphComponents}
\end{tabular}

\end{center}

\begin{section}{The cyclic random walk -- preliminaries}\label{sec:setup}

Let us now introduce the \emph{cyclic random walk} (abbreviated as CRW), which will be the crucial tool in our analysis of permutations arising from the distribution $\mu_{\beta, \theta, \mathcal{C}}$.

Recall that $\mathfrak{X}$ consists of finite subsets of $E \times [0,1)$. For a configuration $X = \{(e_1, t_1), \ldots, (e_k, t_k)\} \in \mathfrak{X}$ the pairs $(e_i, t_i)$ will be called \emph{bridges}. For a vertex $v \in V$ the set $\{v\}\times [0,1)$ will be called the \emph{bar} of vertex $v$. If $e_{i} = \{v,w\}$, we think of a bridge $(e_i, t_i)$ as joining the bars of a vertex $v$ and a vertex $w$ at time $t_i \in [0,1)$.

We note that in the sequel $X$ will always be sampled from a distribution for which almost surely all $t_{i}$ are pairwise different and there are no bridges at prescribed deterministic times, hence there is no ambiguity in how the process is defined.

\begin{figure}[h]
	\centering
	\begin{subfigure}[b]{0.35\textwidth}
		\includegraphics[width=\linewidth]{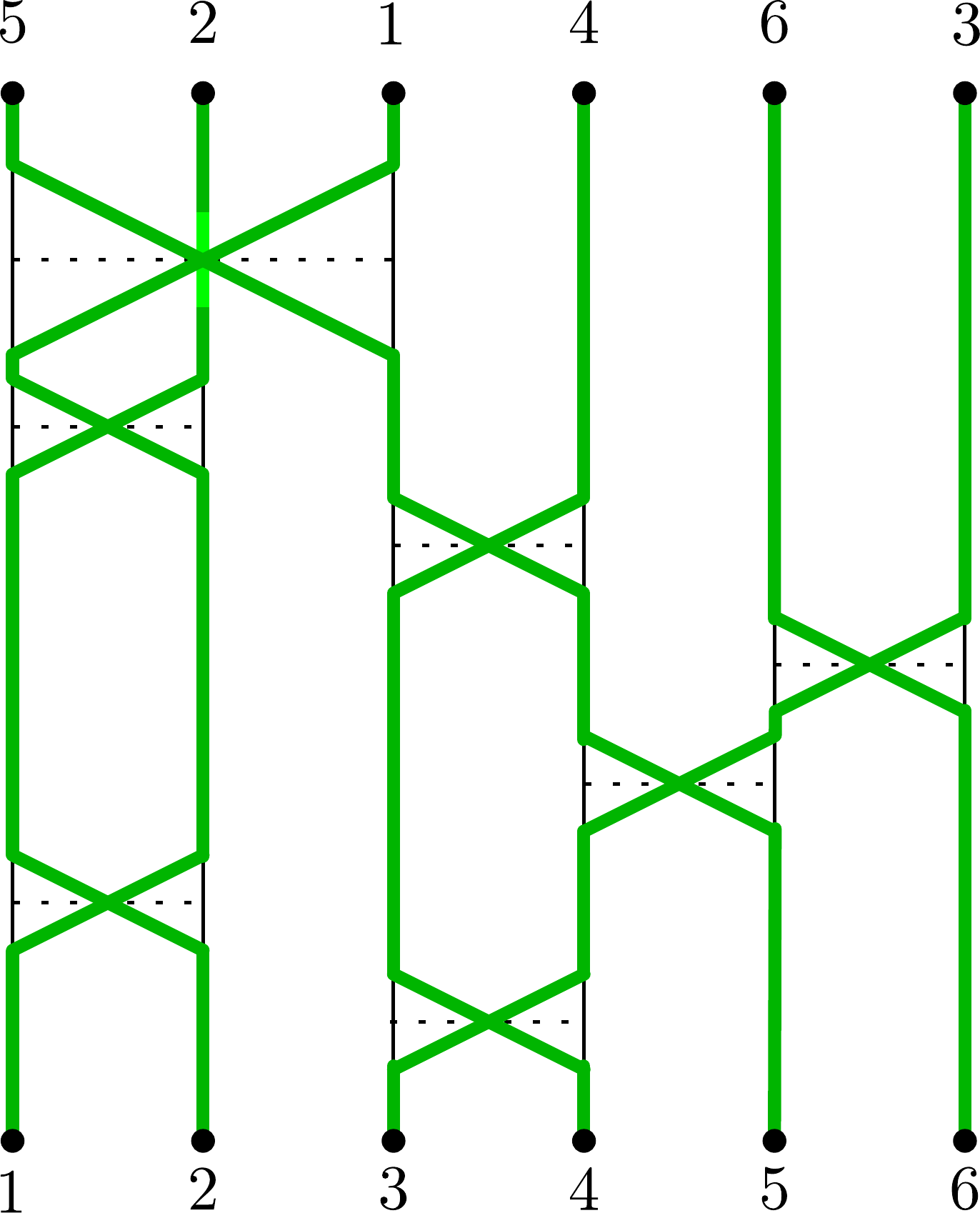}
    \subcaption{A configuration $X$ of bridges}
	\end{subfigure}
	\hfill
	\hspace{3pc}
	\begin{subfigure}[b]{0.35\linewidth}
		\includegraphics[width=\textwidth]{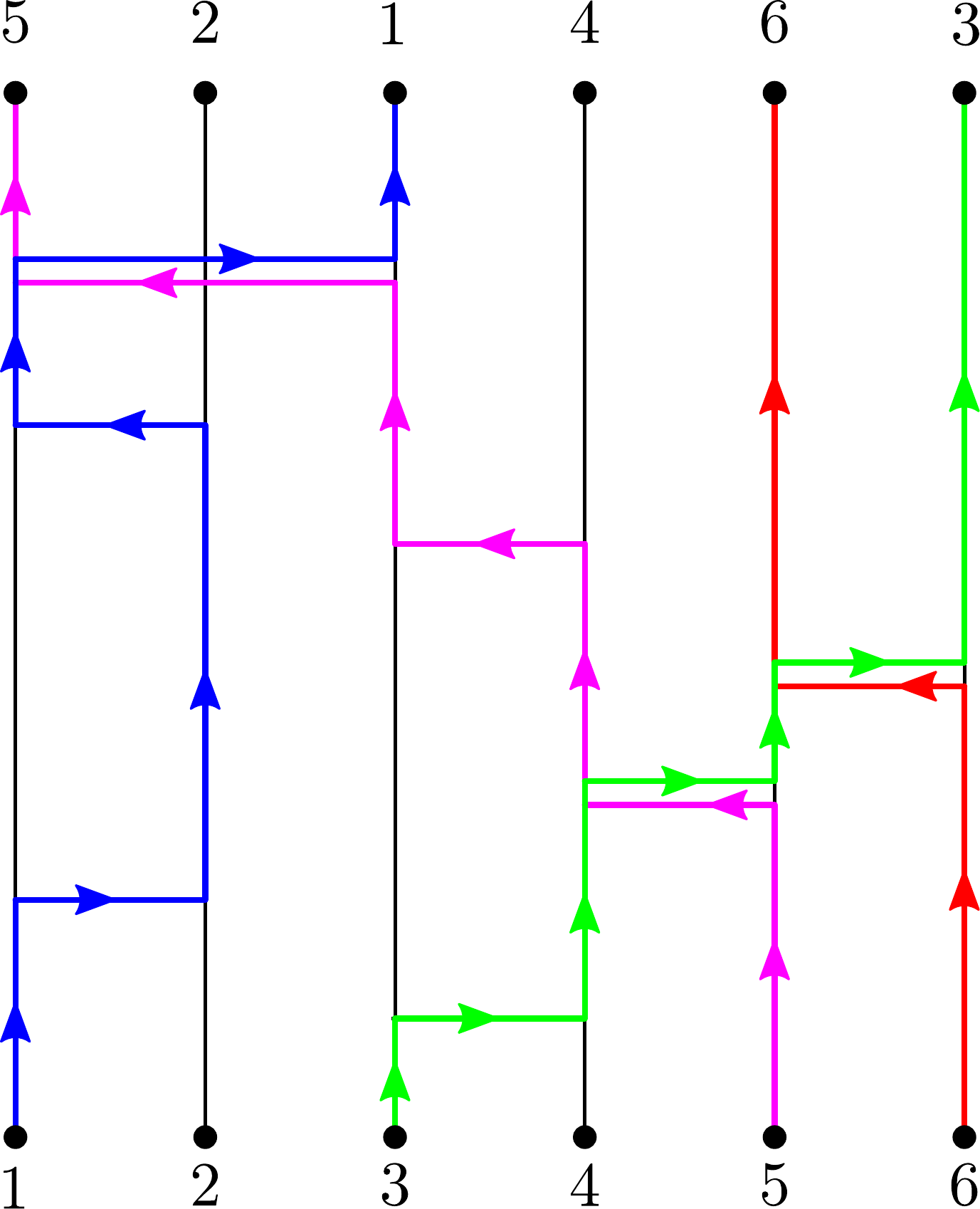}
    \subcaption{The corresponding cyclic random walk}
	\end{subfigure}
	
	\caption{A configuration of bridges and the corresponding CRW. (left) The dotted lines represent the bridges of $X$. The permutation $\sigma(X)$ is determined in the following way -- the labels at the bottom are the labels of the vertices and at the top we have put where they map to under $\sigma(X)$. In this example we have $\sigma(X) = (1 3 6 5) (2) (4)$. (right) The path of the CRW using the bridges of $X$. The direction in which the CRW travels is indicated by the arrows. Note that some bridges are traversed twice, which corresponds to backtracks.
}
\label{fig:crw}
\end{figure}

Consider now a (possibly random) configuration $X \in \mathfrak{X}$. \hfff{def:crw}{The associated \emph{cyclic random walk} $\crw = (\mathcal X_{s} \colon s \geq 0)$} is a continuous time process with values in $V\times [0,1)$, exploring the bridges given by $X$. It starts at a point $\mathcal X_{0} = (v,t) \in V\times [0,1)$, then moves upwards on the bar of the vertex $v$ at unit speed, starting at height $t$, until it encounters a bridge $(\{v,w\}, s) \in X$. Upon encountering a bridge, the CRW jumps to its other end and continues moving on the new bar. Once it gets to height $1$, the CRW moves to the bottom of the bar, at height $0$. Note that the CRW can encounter an already traversed bridge, in which case we say that it makes a backtrack.

Notice that the CRW is periodic. Once $\crw$ reaches its starting point again (which will happen in a finite time), then it will repeat itself. See Figure \ref{fig:crw} for an example of a configuration $X$ and the trajectory of the corresponding CRW.

The CRW as well as other jump processes we will consider in the paper will be always c\`adl\`ag.

\hfff{def:crw_notions_interval}{For a time interval $I$ by $\X_{I}$ we denote the path of the CRW during $I$, likewise $\Zz_I$ is the set of vertices visited by $\crw_I$}, i.e.,
\[
	\Zz_I := \{w\in V: \mathcal X_{s}=(w, z)\text{ for some }s \in I \text{ and }z\in [0,1)\}.
\]
\hfff{def:crw_notions}{We will use the abbreviation $\Zz_s := \Zz_{[0,s]}$ and simply write $\Zz$ for $\Zz_{[0, \infty)}$. For $k\in \N$ we denote by $T_k:=\inf\{s \geq 0: |\mathcal Z_s|\geq k\}$ the time at which the CRW discovers a previously unvisited vertex for the $k$-th time (where we use the convention that $\inf \emptyset = \infty$ and note that $T_1 = 0$).}
For $k\in \N$ we set $Z_k := \mathcal Z_{T_k}$.

The cyclic random walk started at $(v, 0)$ will be denoted by $\crw(v)$ (and likewise for $\Zz(v)$, $T_{k}(v)$, $Z_{k}(v)$ etc.). We will often write simply $\crw = \crw(v)$, $\Zz = \Zz(v)$ etc. if $v$ is fixed. We will also often abuse notation and write $\X_{s} = v$ for $\X_{s} = (v,t)$. Note that since the bars are of height one, the second coordinate $t$ can be read from the time $s$, i.e., $t = s \mod 1$.

The reason for introducing the cyclic random walk is the following relation between the CRW $\crw$ using the bridges of $X$ and the permutation $\sigma(X)$. Consider the sets
\begin{align}\label{eq:very-big-ooo}
	\mathcal{O}(v) & := \{w \in V \colon \exists_{t \ge 0} \crw_t(v) = (w,0)\}, \\
	\mathcal{O}_k(v) & := \{w \in V \colon \exists_{t \in [0, k)} \crw_t(v) = (w,0)\}.
\end{align}

It is readily seen that $\mathcal{O}(v)$ is equal to the orbit of the vertex $v$ under the permutation $\sigma(X)$. Moreover if the orbit has size $\ell$, then $\mathcal{O}_{k}(v)$ consists of the first $k\wedge \ell$ elements of the orbit. In other words, vertices visited by $\crw(v)$ at integer times enter the orbit of $v$ under the permutation $\sigma(X)$. Note that we have $\mathcal{O}(v) \subset \Zz(v)$ but not necessarily  $\mathcal{O}_{k}(v) \subset Z_{k}(v)$. For example, the CRW shown in Figure \ref{fig:crw}, started at $1$, visits vertices $1$, $2$ and $3$ up to time $1$ (blue path), but $\mathcal{O}_{2}(1) = \{1,3\} \not\subset \{1,2\} = Z_{2}(1)$.

Suppose now that the set of bridges $X$ is random. Then for $v \in V$ the cyclic random walk $\crw(v) = (\crw_{s}(v) \colon s\geq 0)$ is itself a stochastic process. \hfff{def:crw_filtrations}{Let $\mathcal F=(\mathcal F_s:s \geq 0)$ denote its natural filtration.}
Set also $\mathcal G_k:=\mathcal F_{T_k}$.

In what follows we will be interested in the situation where $X$ is drawn from the distribution $\mu_{\beta, \theta, \mathcal{C}}$ defined in \eqref{eq:measure}. Fix $\beta, \theta > 0$, an admissible function $\mathcal{C}$, and let $X \in \mathfrak{X}$ be distributed according to $\mu_{\beta, \theta, \mathcal{C}}$. For a fixed vertex $v \in V$, by $\crw^{\beta, \theta, \mathcal{C}}(v)$ we will denote the corresponding cyclic random walk and call it the cyclic random walk associated to $\mu_{\beta, \theta, \mathcal{C}}$, started at $v$.

As a final note, we remark that on several occasions we will work with events defined in terms of uncountable intersections over a set of times. Since the processes we are considering almost surely make only countably many jumps, all such events will be in fact measurable.

\end{section}

\begin{section}{Isoperimetry of the cyclic random walk}\label{sec:isoperimetry}

\begin{subsection}{The setting and main results}\label{sec:isoperimetry-setting}

We will now define a notion of isoperimetry for subsets of $H_{n}$. For a set $A\subset V$ let
\begin{equation}\label{eq:iotaDefinition}
		\iota(A) :=\max\left\{\max_{i\in \{0,\dots,n-1\}} |L_i\cap A|\, ,\, \max_{i\in \{0,\dots,n-1\}} |D_i\cap A| \right\}
\end{equation}
and
\begin{equation}\label{eq:chiDefinition}
		\chi(A) := \min\left\{\min_{i\in \{0,\dots,n-1\}} |L_i\cap A|\, ,\, \min_{i\in \{0,\dots,n-1\}} |D_i\cap A| \right\}.
\end{equation}
Given $A \subset V$ by $E(A)$ we will denote the set of edges $\{v,w\}\in E$ such that $v, w \in A$. As each vertex $v \in A$ has at least $2 (\chi(A) - 1)$ and at most $2 (\iota(A) - 1)$ neighbors in $A$, we have the inequalities
\[
\chi(A) - 1 \leq \frac{|E(A)|}{|A|} \leq \iota(A) - 1,
\]
which justifies the name ``isoperimetry''. Note also the following subadditivity property of $\iota$: for any two subsets $A, B \subset V$ we have $\iota(A \cup B) \leq \iota(A) + \iota(B)$. We also have $\iota(A) \leq \iota(B)$, $\chi(A) \leq \chi(B)$ whenever $A \subset B$, in particular both $\iota(\Zz_{t})$ and $\chi(\Zz_{t})$ are nondecreasing in $t$.

Our main technical result is the following upper bound on $\iota$ of $\Zz_t$
\begin{proposition}\label{prop:isoperimetry-log2n}
Fix $\theta > 0$, an admissible function $\mathcal{C}$ and let $\beta_{0},\beta_{1}$ be such that $\beta_1>\beta_{0} > \Theta / 2$. Consider $\beta \in [\beta_{0}, \beta_{1}]$ and let $\Zz(v) := \Zz^{\beta, \theta, \mathcal{C}}(v)$ be the trace of the cyclic random walk associated to $\mu_{\beta, \theta, \mathcal{C}}$, started at $v\in V$. Then there exist $C, c >0$ (depending only on $\beta_{0}$, $\beta_{1}$, $\theta$, in particular independent of $\mathcal{C}$) such that
\[
	\Pp \left( \forall v \in V \, \iota\left(\Zz_{n
	\log^2 n}(v) \right) \leq C \log^2 n  \right) \geq 1 - C e^{- c \log^2 n}.
\]
\end{proposition}

The above proposition will be a key tool in the proof of existence of long cycles. Together with a corresponding (easier) lower bound on the quantity $\chi$, given in Proposition \ref{prop:isoperimetry-log2n-lower-bound} below, it gives good control of the isoperimetric properties of \emph{mesoscopic} segments of cycles (i.e., of length roughly $n\log^2 n$), which can be then lifted to isoperimetry of full cycles of length at least $n\log^2 n$. This will enable us to prove that with high probability the conditional split and merge probabilities in the evolution of cycles behave similarly as in the mean-field case, allowing for an application of a modification of an argument due to Schramm.

The parameters $\beta_0,\beta_1$ are introduced for technical reasons. When applying the above proposition in the proof of Theorem \ref{thm:main_theorem_upperbound} we will need uniformity of constants for $\beta$ belonging to the interval $[\beta_0, \beta_1]$, with appropriately chosen $\beta_0, \beta_1$.

The crucial part of the proof of Proposition \ref{prop:isoperimetry-log2n} is a bootstrap argument regarding isoperimetry. Informally speaking, we will show that if at some time $T \leq n \log^2n$ we have with high probability ``good'' isoperimetry (of the order of $n^{\alpha}$ for some small enough $\alpha$), then actually we have with high probability ``very good'' isoperimetry (of the order of $\log^2 n$). This is formalized in the following

\begin{lemma}[Bootstrap]\label{le:bootstrap-isoperimetry}
Let $\theta, \mathcal{C}, \beta_{0}, \beta_{1}, \beta, \Zz(v)$ be as in Proposition \ref{prop:isoperimetry-log2n}. Let $\alpha \in (0, 1/100)$ and $T \leq n\log^2 n$. Suppose that for some $C_1, c_1 >0$ we have
\begin{equation*}
	\Pp\left( \forall v \in V \, \iota(\Zz_{T}(v)) \leq n^{\alpha} \right) \geq 1 - C_1 e^{-c_1 \log^2 n}.
\end{equation*}
Then there exist $C, C_2, c_2 >0$ (depending only on $\theta, \beta_{0}, \beta_{1}, \alpha, c_1, C_1$) such that we have
\[
\Pp \left( \forall v \in V \, \iota(\Zz_{T}(v) ) \leq C \log^2 n  \right) \geq 1 - C_2 e^{- c_2 \log^2 n}.
\]
\end{lemma}

With this lemma the proof of Proposition \ref{prop:isoperimetry-log2n} is rather straightforward and is given in Section \ref{sec:proof-of-isoperimetry}. The whole next section will be devoted to the proof of Lemma \ref{le:bootstrap-isoperimetry}.

Later on we will also need a lower bound on $\chi$, which is analogous to Proposition \ref{prop:isoperimetry-log2n} provided the orbit of the permutation defined by the CRW contains at least $n \log^2 n$ vertices. Recall the definition \eqref{eq:very-big-ooo} of the sets $\mathcal{O}_k(v)$ and $\mathcal{O}(v)$.

\begin{proposition}\label{prop:isoperimetry-log2n-lower-bound}
Fix $\theta > 0$, an admissible function $\mathcal{C}$ and let $\beta_{0}, \beta_{1} > 0$. Consider $\beta \in [\beta_{0}, \beta_{1}]$ and
let $\crw(v) := \crw^{\beta, \theta, \mathcal{C}}(v)$ be the cyclic random walk associated to $\mu_{\beta, \theta, \mathcal{C}}$, started at $v$. There exist $C, c >0$ (depending only on $\beta_{0}$, $\beta_{1}$, $\theta$, in particular independent of $\mathcal{C}$) such that
\[
	\Pp \left( \forall v \in V \,  \chi(\mathcal{O}_{n
	\log^2 n}(v) ) \geq c \log^2 n  \;\mathrm{or}\; |\mathcal{O}(v)| < n \log^2 n  \right) \geq 1 - C e^{- c \log^2 n}.
\]
\end{proposition}

As the proof of the above proposition is much less involved than for Proposition \ref{prop:isoperimetry-log2n}, it is given separately in Section \ref{sec:isoperimetry-lower-bound}.

\paragraph{Outline of the proof of Lemma \ref{le:bootstrap-isoperimetry}.}

Here we outline the proof strategy of the main technical result of this section -- Lemma \ref{le:bootstrap-isoperimetry}. Most of our effort is  devoted to the analysis of interactions of the CRW with its history. When the CRW enters a previously visited vertex it may reuse already explored bridges, which can generate a complex behavior depending on the graph $G_t$ of already visited vertices.

One should keep in mind the following intuitive picture. By the assumption $\iota(\Zz_{T}) \leq n^{\alpha}\ll n$ the trace of the CRW is not too concentrated in any row or column. Thus, while moving on the bar of a vertex $v$, if the CRW discovers an unexplored bridge, it will typically jump to a yet unexplored vertex. If it fails to discover a new bridge, it backtracks to the vertex visited before $v$. Such backtracks are common and may cascade creating some interactions of the CRW with its history. Due to the assumption $\beta > \Theta / 2$, the rate of discovery of new vertices is fast enough so that the CRW tends to escape its history, making the above mentioned interactions short-ranged and thus fairly easy to analyze. For a very similar reason, unless the CRW  closes into a cycle quickly, it makes a fairly long cycle. When $T$ is at least of the order of $n$, the CRW occasionally jumps to a vertex visited a long time before. Analysis of such long-range interactions is the main technical difficulty of the proof.

The main task is to show that the time between two subsequent visits in any fixed row or column, say $L_0$, is $c n$ (for some $c > 0$) with uniformly positive probability. Having done that, by a comparison with a sum of independent random variables it is straightforward to conclude that $|\Zz_{n \log^2 n}\cap L_0| \leq C \log^2 n$ with very high probability and thus also $\iota(\Zz_{n \log^2n}) \leq C \log^2 n$.

The CRW can hit $L_{0}$ either by a direct jump using a previously unexplored bridge or by entering through its history. As $T \leq n \log^2 n$ and at each step the CRW has chance roughly $c / n$ of a direct jump, typically it will make $c \log^2 n$ visits of the first type.

To analyze entering $L_{0}$ through the history we distinguish two cases. The first is when the CRW jumps using a new bridge to a vertex which is close in $G_t$ to $L_{0}$. The second is when the vertex is far from $L_{0}$ and the CRW makes a long backtrack employing already used edges.

To rule out the first possibility we show, using a rather delicate argument, that the dangerous zone (``bad set'') consisting of small balls around $L_{0}$ in $G_{t}$ is small enough so that the CRW is quite unlikely to jump to it. Thus it is very unlikely to observe $\log^2 n$ of such jumps.

In the second case it is enough to show that any sufficiently long path in $G_t$ has what we call ``large potential''. Intuitively speaking, the CRW traversing such a path has many chances to escape it by jumping to a new vertex and performing a long excursion avoiding its history. Thus the CRW is unlikely to ever make a long backtrack. This part of the argument is rather technical and again uses crucially the assumption $\beta > \Theta / 2$.

The roadmap to the proof is as follows. In the rest of this section and Section \ref{sec:intensities} we set up the framework for analyzing the excursions made by the CRW, in particular providing bounds for the intensity of discovering new vertices, number of jumps to the history (Lemma \ref{lm:internal-hits}) and the rate at which the ``potential'' corresponding to bars is exhausted (Lemma \ref{lm:good-potential}). This is then used in Section \ref{sec:excursions} to show that the CRW is likely to make short excursions not intersecting its history (Lemma \ref{lm:excursions}). In Section \ref{sec:bad-set} we show that the bad set described above is typically small (Proposition \ref{prop:bad-ball-is-small}) and thus is unlikely to be hit by the CRW (Lemma \ref{lm:bad-hits-process}). In Section \ref{sec:backtracks} we show that long paths typically have large potential (Lemma \ref{lm:all-paths-are-good}) and thus are unlikely to be backtracked (Proposition \ref{prop:backtrack}). All these pieces are then used in Section \ref{sec:proof-of-isoperimetry} to show that visits to $L_{0}$ are infrequent (Lemma \ref{lm:main-lemma}), which easily implies (Corollary \ref{cor:intersections-with-l0}) good isoperimetry claimed in the conclusion of Lemma \ref{le:bootstrap-isoperimetry}.

\paragraph{Basic notation and assumptions.}\label{sec:basic}

Our goal for the rest of this section is to prove Lemma \ref{le:bootstrap-isoperimetry}. Therefore, from now on we fix
\begin{itemize}
\item $\theta > 0$ and an admissible function $\mathcal{C}$,
\item $\beta_{0}$, $\beta_{1}$ such that $\beta_{0} > \Theta / 2$, and $\beta \in [\beta_{0}, \beta_{1}]$,
\item $\alpha \in (0, 1/100)$ and \hfff{def:Ttime}{$T \leq n\log^2 n $},
\item $\varepsilon \in (0, \frac{1}{20})$ such that $\alpha < \varepsilon / 4$
\end{itemize}
(the parameter $\varepsilon$ will play a technical role in intermediate calculations).

We will be considering the cyclic random walk $\crw^{\beta, \theta, \mathcal{C}}(v)$ associated to $\mu_{\beta, \theta, \mathcal{C}}$, started at $(v,0)$ for a fixed vertex $v$. From now on for brevity we write simply $\crw_{t} = \crw^{\beta, \theta, \mathcal{C}}_{t}(v)$, $\Zz_{t} = \Zz^{\beta, \theta, \mathcal{C}}_{t}(v)$ etc.

In what follows whenever we write about global constants like $C,c$ we allow them to depend on $\beta_{0}, \beta_{1},\varepsilon$ (in addition to dependence on $\theta$ and $\alpha$), so that for $\beta \in [\beta_{0}, \beta_{1}]$ all the statements hold with constants depending uniformly on $\beta$. The dependence on $\theta,\alpha,\beta_0,\beta_1,\varepsilon$ will be suppressed in the notation. Occasionally, if a constant depends additionally on some other parameter $\delta$, we will stress it by writing e.g., $c = c(\delta)$. All constants will be independent of the function $\mathcal{C}$ as long as the Lipschitz condition \eqref{eq:Lipschitz} is satisfied.

In proofs we will often take $n$ large enough, depending on the other parameters, but unless stated otherwise the propositions being proved will be such that this can be absorbed into constants $C,c$ appearing in the statements, so the propositions in fact hold for all $n$. 

We will now introduce several notions which will be useful in analyzing the explorations of the cyclic random walk.

\paragraph{The graph $G_t$.}

\hfff{def:gtcrw}{The vertices $\Zz_t$ and bridges explored up to time $t$ by the CRW} induce a graph denoted by $G_t$ (we allow multiple edges if there is more than one bridge between two vertices). Let $d_{G_t}(\cdot, \cdot)$ be its natural graph metric. By $B_{G_{t}}(v,r)$ we denote the ball of radius $r$ in this metric around the vertex $v$ in $G_{t}$.

\paragraph{The bad set.}

Recall that $\varepsilon$ is a parameter fixed at the beginning of this section. Let
\begin{equation}\label{eq:defBadSet}
\badSet_t := \bigcup\limits_{v \in \Zz_{t} \cap L_{0}} B_{G_{t}}(v, n^{\varepsilon})
\end{equation}
\hfff{def:badset}{be the \emph{bad set} at time $t$}. Note that $\badSet_t$ is nondecreasing in $t$.

The reason for calling this set 'bad' is that once the CRW ends up in $\badSet_t$, it might quickly backtrack its way to $L_{0}$ and then make multiple jumps inside the same row, thus ruining good isoperimetric properties of the trajectory.

\paragraph{Dead vertices.}

As the CRW explores the set of vertices, it might happen that at time $t$ the bar $\{v\}\times [0,1)$ corresponding to a vertex $v$ is completely exhausted, i.e.,  $\{v\}\times [0,1)\subset \crw_{[0,t]}$. We will call such a vertex $v$ \emph{dead}.

\paragraph{Stopping times.}
Later on it will be convenient to use the following stopping times related to the CRW.  \hfff{def:tauiso}{For $\delta>0$ we define $\tau_{iso}^\delta$ as the time when the CRW loses good isoperimetric properties, i.e.}
\[
\tau_{iso}^\delta := \inf \{ t \geq 0 \colon \iota(\Zz_{t}) > n^{\delta} \},
\]
and \hfff{def:tauc}{let $\tau_c$ be the time that the CRW closes into a cycle}, i.e.
\[
\tau_c := \inf\{t> 0:  \X_t = (v,0)\},
\]
where $v$ is the starting vertex, meaning $\X_0 = (v,0)$.
\paragraph{Jumps.}

It will be important to distinguish several types of jumps that the CRW can make.  We will call a jump \emph{fresh} when a previously unexplored bridge is used, otherwise we call it a \emph{backtrack}. Suppose that $\X$ makes a fresh jump at time $t$. We will call it an \emph{internal jump} if $\X_{t} \in \Zz_{t-}$, a \emph{bad jump} if $\X_{t} \in \badSet_{t-}$ and a \emph{direct jump to $L_{0}$} if $\X_{t-} \notin L_{0}$ and $\X_{t} \in L_{0}$.

\newpage
\afterpage{
\begin{figure}[h]\label{fig:graph}
	\centering
	\begin{subfigure}[t]{0.45\textwidth}
		\includegraphics[width=\linewidth]{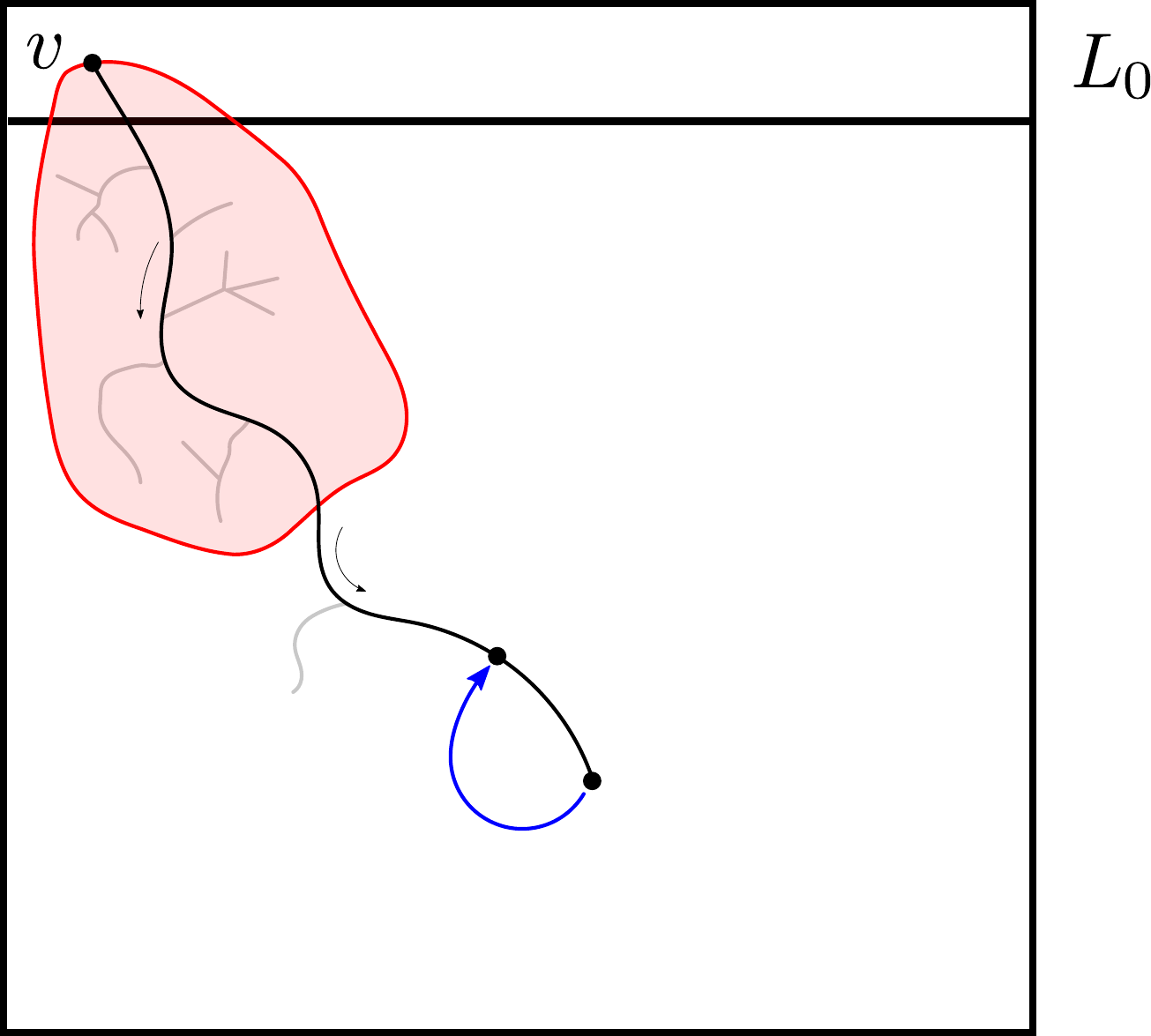}
    \subcaption{The CRW starts at $v$ and makes an exploration leaving the bad set. The exploration ends with an \emph{internal jump} (blue edge).}
	\end{subfigure}
	\hfill
	\vspace{2pc}
	\begin{subfigure}[t]{0.45\textwidth}
		\includegraphics[width=\linewidth]{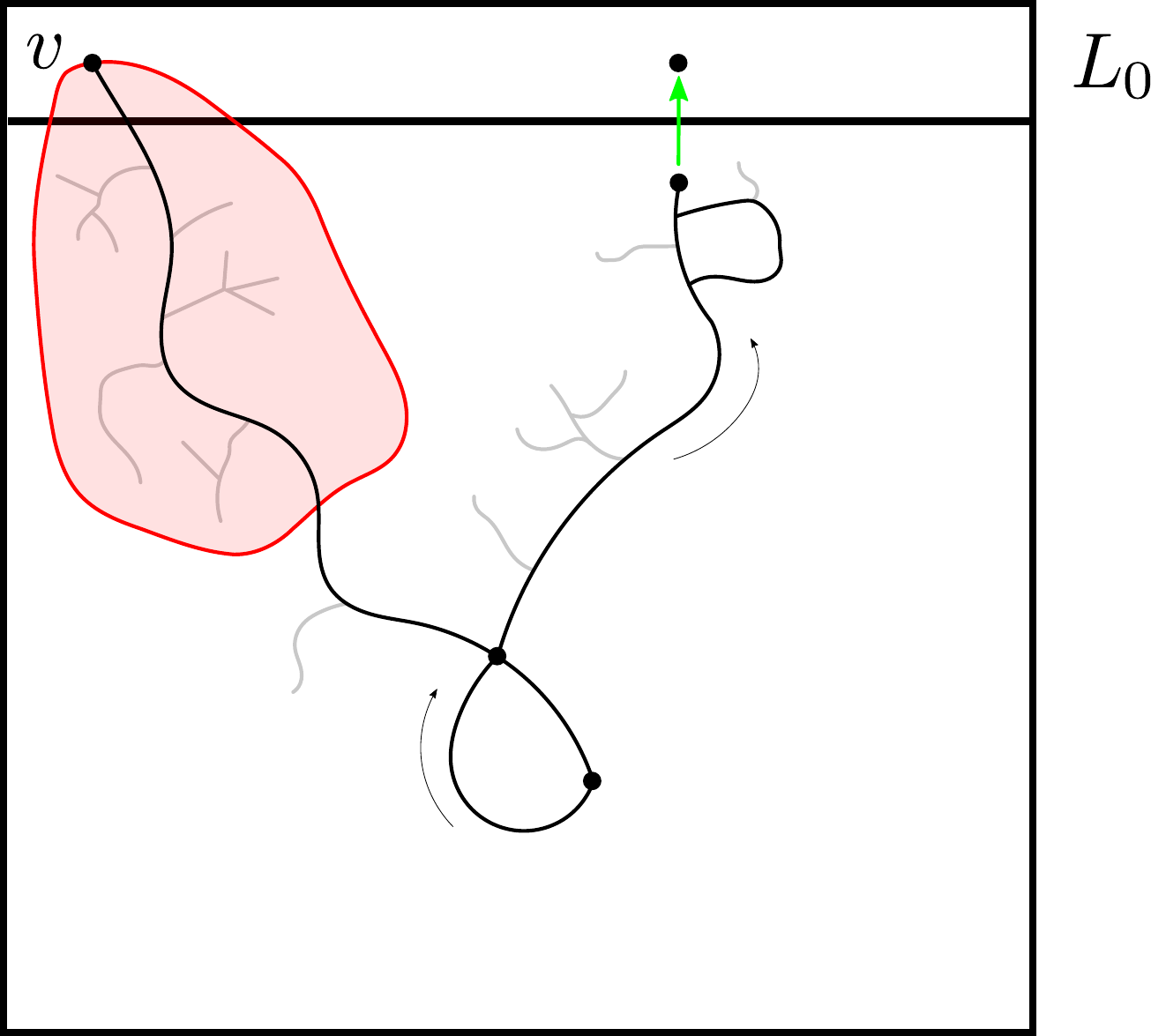}
    \subcaption{The CRW continues the exploration outside the bad set, making a \emph{direct jump} to $L_{0}$ at the end (green edge).}
	\end{subfigure}
		\hfill
	\begin{subfigure}[t]{0.45\textwidth}
		\includegraphics[width=\linewidth]{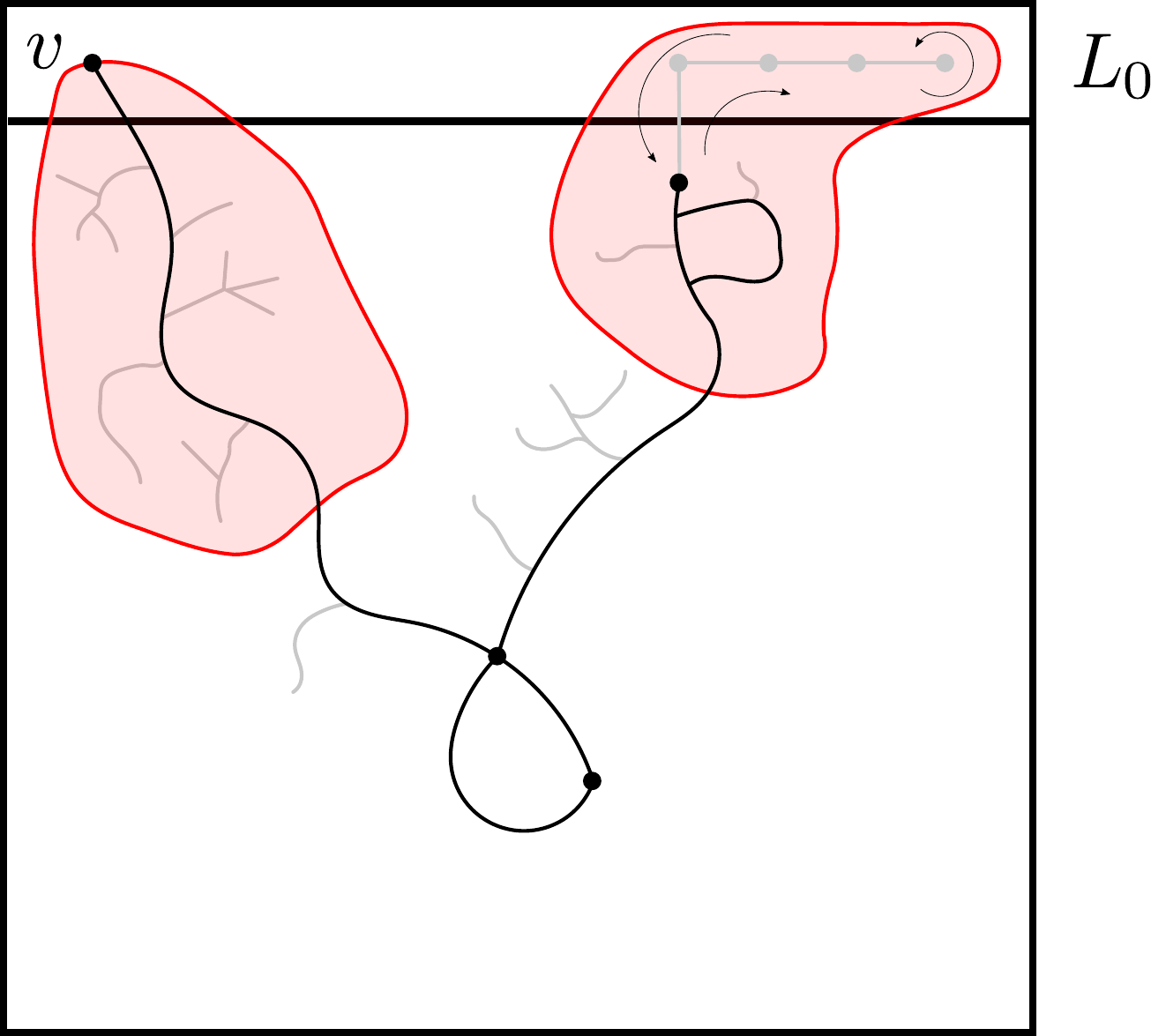}
    \subcaption{After a jump to $L_{0}$ some of the previously visited vertices become a part of the bad set. The CRW then makes an excursion inside $L_{0}$ which is completely backtracked (grey edges), ending up outside $L_{0}$.}
	\end{subfigure}
			\hfill
	\begin{subfigure}[t]{0.45\textwidth}
		\includegraphics[width=\linewidth]{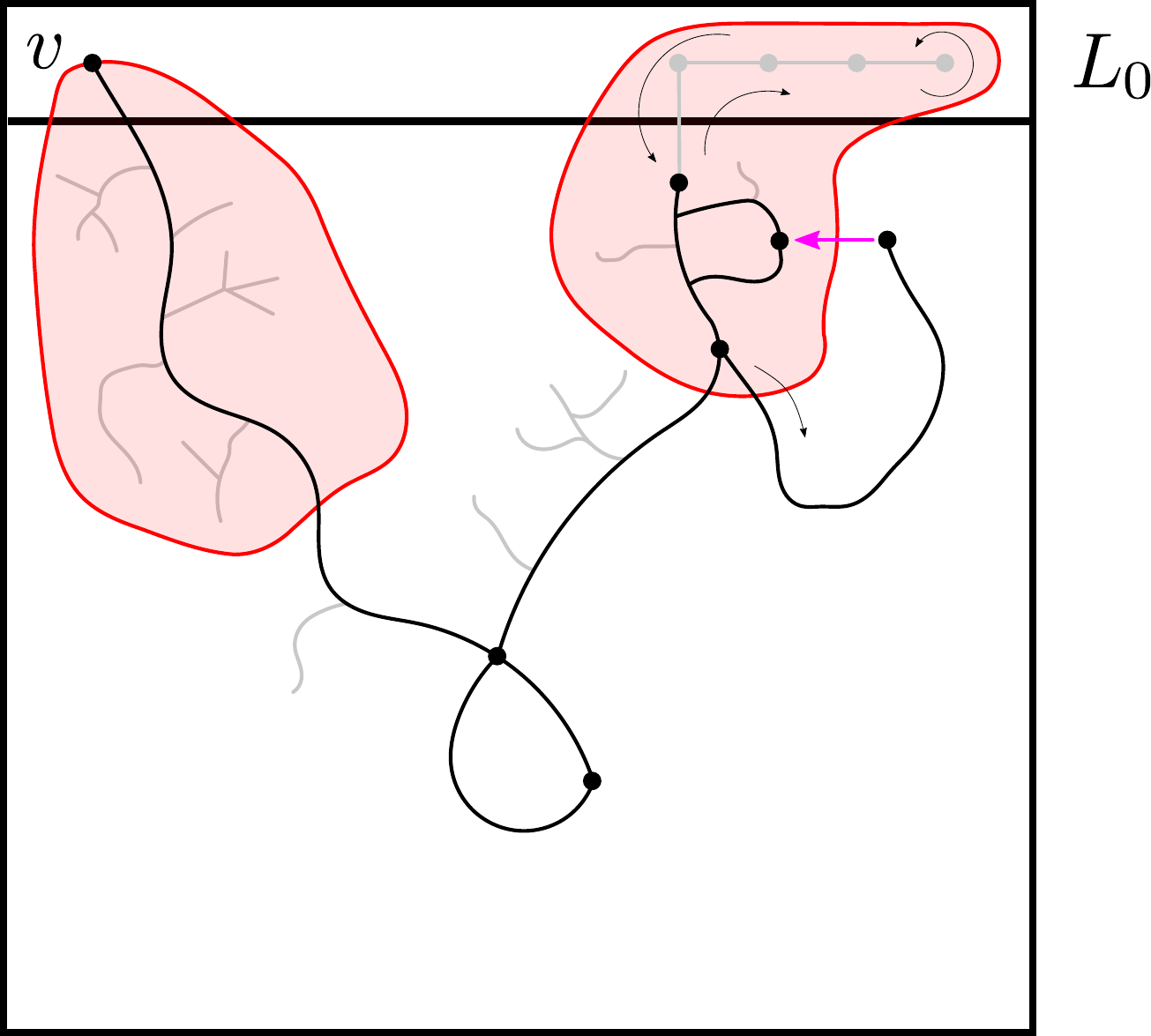}
    \subcaption{The CRW makes an excursion leaving the bad set, ending in a \emph{bad jump} (purple edge).}
	\end{subfigure}
	
	\caption{Sketch of the evolution of the graph $G_{t}$, with the CRW started at $v$. Dead vertices are shown in grey. The bad set $\badSet_{t}$ is marked in red.}
\end{figure}
\clearpage
}
\newpage
\end{subsection}

\begin{subsection}{Intensities of jumps and the potential}\label{sec:intensities}

Recall that a counting process is a nondecreasing, integer valued c\`adl\`ag stochastic process starting at zero and with jumps equal to one.

\paragraph{Intensity.}

Let $Y$ be an $\Ff_t$-adapted counting process. We will say that a nonnegative process $\lambda$ is an \emph{intensity} of $Y$ if $\lambda$ is $\mathcal{F}_t$-progressively measurable, $\int_0^t \lambda_u \, du < \infty$ for all $t$,  and the process $Y_t - \int_0^t \lambda_s \, ds$ is an $\Ff_t$-martingale. \\

In what follows we will often need (conditional) concentration inequalities for counting processes. Two of them, used most frequently, are stated below for convenience and the other two, which will be used only once, are stated in Appendix \ref{sec:appendix_concentation}. Proofs of all of them can be found in Appendix \ref{sec:appendix_concentation} as well.

\begin{lemma}\label{lm:concentration-upper-bound}
Let $Y_t$ be a counting process with bounded intensity $\lambda$ and compensator $\Lambda_t = \int_0^t \lambda_s \, ds$. Assume that $\sigma, \tau$ are bounded stopping times such that $\sigma \le \tau$. Consider  $\ell > 0$ and let $X$ be a Poisson variable with parameter $\ell$.
Then for any $r \geq 0$ we have almost surely
\begin{displaymath}
\p(\{Y_{\tau} - Y_\sigma \ge r\}\cap\{\Lambda_\tau - \Lambda_\sigma \le \ell\}|
\mathcal{F}_\sigma) \le \p(X \ge r).
\end{displaymath}
If $r \geq \ell$, we have in particular
\[
\p(\{Y_{\tau} - Y_\sigma \ge r\}\cap\{\Lambda_\tau - \Lambda_\sigma \le \ell\}|
\mathcal{F}_\sigma) \le \exp\left(-r \log \left( \frac{r}{e\ell}\right) - \ell \right).
\]
\end{lemma}

We will also need a corresponding lower bound.

\begin{lemma}\label{lm:concentration-lower-bound-super}
Let $Y_t$ be as in the previous lemma. Let $\sigma, \tau$ be bounded stopping times such that $\sigma \leq \tau$ and let $\delta,\ell > 0$. Let also $X$ be a Poisson random variable with parameter $\ell$. Then with probability one,
	\[
	\p\rbr{\left\{ Y_{\tau}-Y_{\sigma}\leq\ell(1-\delta)\right\} \cap\left\{ \Lambda_{\tau} - \Lambda_{\sigma}\geq\ell\right\}|\mathcal{F}_\sigma }\leq \p(X \le (1-\delta) \ell) \le \exp\left(-\delta^{2}\ell/2\right),
	\]
	where $\Lambda_{t}=\int_{0}^{t}\lambda_{s} \, ds$.
\end{lemma}

\paragraph{Set of accessible vertices.}
By $A_{t}$ we will denote the set of vertices which at time $t$ are available to the CRW by a fresh jump. Formally, $A_t = \emptyset$ if the CRW has closed into a cycle before time $t$, otherwise let $(w,z) = \crw_t$ and
\[
A_t := \{v \in V\colon \{v,w\} \in E \;\text{and}\; (v,z) \notin \crw_{[0,t)} \}.
\]

\begin{lemma}[Intensity of jumps]\label{le:intensity}
Let $Q_t$ be an $\mathcal{F}_{t}$-adapted c\`adl\`ag process of subsets of $V$ such that $Q_{t}$ can jump only at times when $\crw_{t}$ jumps. Let $J_t = |\{s\le t\colon  \textrm{$\crw$ makes a fresh jump at time $s$ and}\;\mathcal{X}_s \in Q_{s-}\}|$. Then the counting process $J$ has intensity $\lambda$ which satisfies
\begin{displaymath}
  \lambda_t \Theta^{-1} \le \frac{\beta}{n-1}|A_t\cap Q_t| \le \lambda_t \Theta.
\end{displaymath}
In particular for $\theta = 1$ we have $\lambda_t = \frac{\beta}{n-1} |A_t\cap Q_t|$.
\end{lemma}

The proof of this lemma is presented in Appendix \ref{ap:intensity}. We stress that the lemma is not specific to the Hamming graph and its statement holds for any weighted transposition process on a finite graph (with $\frac{\beta}{n-1}$ replaced by the appropriate edge intensity of the underlying point process).

\medskip

\hfff{l:jumps}{Let $I_{t}$ (resp. $I_t^b,I_t^d$) denote the total number of internal} (resp. bad, direct to $L_{0}$) jumps up to time $t$. They are counting processes. \hfff{l:intensities}{We will denote their intensities by $\lambda$ (resp. $\lambda^b$, $\lambda^d$)}. The intensity of the process $|\Zz_t|$ will be denoted by $\mu$.

\begin{cor}\label{cor:intensity}
For $t< \tau_c$,
\begin{equation}\label{eq:intensity-mu}
	\mu_t \geq 2 \Theta^{-1} \beta \frac{n - \iota(\Zz_t) }{n-1}
\end{equation}
and
\begin{equation}\label{eq:intensity-lambda}
	\lambda_t \leq \frac{4 \Theta \beta}{n} \iota(\Zz_t).
\end{equation}
\end{cor}

\begin{proof}
Let us start with \eqref{eq:intensity-mu}. For any $t \geq 0$ let the current vertex $w = \crw_t$ belong to the column $D$ and row $L$. Using Lemma \ref{le:intensity} with $Q_{t} = V \setminus \Zz_t$, together with the definition of the parameter $\iota$ and the obvious containment $( D \cup L) \setminus \Zz_{t} \subset A_{t}$, we have
\begin{align*}
\mu_{t}\geq& \Theta^{-1} \frac{\beta}{n-1} |(D\cup L) \setminus \Zz_t| = \Theta^{-1}  \frac{\beta}{n-1} (|D \setminus \Zz_t| + |L \setminus \Zz_t|) \geq \Theta^{-1} \frac{2\beta }{n-1} (n-\iota(\Zz_t) ),
\end{align*}
as desired. For the proof of \eqref{eq:intensity-lambda}, we write again by the definition of $\iota$ and Lemma \ref{le:intensity} applied with $Q_{t} = \Zz_{t}$
\[
\lambda_t \leq \Theta \frac{\beta}{n-1} |A_t \cap \Zz_{t}| \leq \Theta \frac{\beta}{n-1} (|D \cap \Zz_{t}| + |L \cap \Zz_{t}|) \leq \Theta \frac{2\beta}{n-1} \iota(\Zz_{t}) \leq \Theta \frac{4 \beta}{n} \iota(\Zz_{t}).
\]
\end{proof}

It will be useful to have an estimate of how quickly previously unexplored vertices are discovered by the CRW. We have the following upper bound

\begin{lemma}[New vertices are not very frequent]\label{lm:fresh-vertices-are-not-often}
There exists an increasing function $g : (1, \infty) \to (0, \infty)$ such that for any $\delta > 2 \Theta \beta$,  $k > 0$ and $l\in \N$ we have
\[
\p(|\Zz_{T_l +k}\setminus \Zz_{T_l}|\leq \lceil \delta k \rceil |\mathcal{G}_{l}) = \Pp\left( T_{l+ \lceil \delta k \rceil}\geq T_{l} + k | \mathcal{G}_{l} \right) \geq 1 - e^{-ck},
\]
where $c = 2\Theta \beta g\left( \frac{\delta}{2\Theta \beta} \right) $.
\end{lemma}

\begin{proof}
By Lemma \ref{le:intensity} the intensity $\mu_t$ of $|\Zz_t|$ satisfies
\[
\mu_{t} \leq \Theta \frac{\beta}{n-1} |A_{t}\setminus \Zz_{t}| \leq \Theta \frac{\beta}{n-1} |A_{t}| \leq 2 \Theta \beta.
\]
By Lemma \ref{lm:concentration-upper-bound} (applied with $Y_t = |\Zz_t|$, $\sigma=T_l, \tau = T_l+k$, $\lambda = \mu$, $\ell = 2\Theta \beta k$, $r = \lceil{ \delta k\rceil} \ge \ell$) we obtain
(recall that $\Lambda_t = \int_0^t \lambda_s ds$),
\begin{multline*}
\p(|\Zz_{T_l + k}\setminus \Zz_{T_l}| > \lceil \delta k \rceil|\mathcal{G}_{l}) =
\p(\{|\Zz_{T_l + k}\setminus \Zz_{T_l}| > \lceil \delta k \rceil\}\cap \{\Lambda_{T_l+k} - \Lambda_{T_l} \le 2\Theta \beta k\}|\mathcal{G}_{l})\\
\le \exp\left(- \delta k \log \left( \frac{\delta}{2 e \Theta \beta}\right) - 2 \Theta \beta k \right).
\end{multline*}
Using again the inequality $\delta > 2 \Theta \beta$, it is easy to see that the right hand side is bounded from above by $e^{-ck}$, with $c = 2\Theta \beta g\left( \frac{\delta}{2\Theta \beta} \right)$ for $g(x) = x \log x + 1 - x $ as desired.
\end{proof}


In the remaining part of this section we will work under the following assumption which will not be explicitly stated in the hypotheses of theorems (we recall that we consider the CRW started at $ \crw_0 =(v,0)$ for some fixed vertex $v$).

\begin{assumption}[Main assumption]\label{lm:isoperimetry}
There exist $C,c > 0$ such that
\[
\Pp\left( \iota(\Zz_{T}) > n^{\alpha} \right) \leq C e^{-c \log^2 n}.
\]
\end{assumption}

We will often need the following lemma, which is a consequence of Markov's inequality.
\begin{lemma}\label{le:easy-lemma}
For any event $A$, any $\sigma$-field $\mathcal{G}$, and any $C,r \ge 0$, if $\p(A) \ge 1 - Ce^{-r}$, then
with probability at least $1- Ce^{-r/2}$,
\begin{displaymath}
  \p(A|\mathcal{G}) \ge 1 - e^{-r/2}.
\end{displaymath}
\end{lemma}

\begin{corollary}\label{cor:isoperimetry-conditional} There exist $C,c>0$ such that for any stopping time $\eta < T$ we have
	\begin{equation*}
		\Pp(\iota(\Zz_T) \leq n^{\alpha}|\filF_\eta) \geq 1-e^{-c \log^2 n},
	\end{equation*}
	with probability at least $1 - \wsp$.
\end{corollary}

In the next lemma we show that $I_{T}$, the total number of internal jumps up to time $T$, is small with high probability.

\begin{lemma}[Internal jumps are rare]\label{lm:internal-hits}
There exist $C, c > 0$ such that
\[
\Pp \left(I_{T} \geq \frac{1}{2} n^{2\alpha} \right) \leq C e^{-c \log^2 n}.
\]
\end{lemma}

\begin{proof}We have
\[
\Pp \left(I_{T} \geq \frac{1}{2} n^{2\alpha} \right) \leq \Pp \left(I_{T \wedge \tau_{iso}^{\alpha}} \geq \frac{1}{2} n^{2\alpha} \right) + \Pp\left( \tau_{iso}^{\alpha} \leq T \right).
\]
By Assumption \ref{lm:isoperimetry} the second term is small enough. Now we bound, for $t < \tau_{iso}^{\alpha}$, the intensity $\lambda_t$ of $I_t$ by Lemma \ref{le:intensity}
\[
\lambda_{t}\leq  \frac{4 \Theta \beta }{n} \iota(\Zz_t) \leq \frac{4 \Theta \beta}{n} n^{\alpha} = 4 \Theta \beta n^{\alpha - 1}.
\]
Since $T \leq n \log^2 n$ and $\frac{1}{2} n^{2\alpha} > 4 \Theta \beta n^{\alpha-1}T$ for $n$ large enough, we can apply Lemma \ref{lm:concentration-upper-bound} to the stopped process $Y_{t} = I_{t \wedge \tau_{iso}^{\alpha}}$, obtaining for $n$ large enough the bound
\begin{align*}
	\Pp \left(I_{T \wedge \tau_{iso}^{\alpha}} \geq \frac{1}{2} n^{2\alpha} \right) &\leq \exp\left( - \frac{1}{2} n^{2\alpha} \log \left( \frac{\frac{1}{2} n^{2\alpha}}{4 \Theta\beta e n^{\alpha - 1} T} \right)  \right)  \leq \exp\left( - \frac{1}{2} n^{2\alpha} \log \left( \frac{n^{\alpha}}{8 \Theta \beta e \log^2 n} \right)  \right),
\end{align*}
and the right hand side is at most $e^{-c \log^2 n}$ for some $c > 0$.
\end{proof}

\paragraph{The potential.}\label{sec:potential}
While visiting a vertex the CRW does not necessarily exhaust its whole bar. \hfff{def:potential}{For $t>0$ by $\PP_t$ we denote the Lebesgue measure of unused parts of visited bars and call it the \emph{potential}}. Formally,
\begin{equation*}
	\PP_t := |\Zz_t| - U_t,
\end{equation*}
where the number of visited vertices $|\Zz_t|$ is equal to the total measure of visited bars and $U_t$ is the measure of their used parts, which is equal to the Lebesgue measure of the path $\crw_{[0,t)}$.

Notice that  until time $\tau_{c}$ the potential $\PP_{t}$ increases by $1$ each time the CRW visits a previously unexplored vertex and otherwise decreases linearly with $t$. This means that for $t \leq \tau_{c}$ the potential $\PP_{t}$ follows the equation (with $\PP_{0} = 1$)
\begin{equation}\label{eq:potential}
\PP_{t} = |\Zz_{t}| - t.
\end{equation}

In the following technical lemma we show that the potential of a path $\X_{[t,t+s]}$ is typically proportional to $s$ and cannot drop significantly before time $T$. Furthermore, with probability bounded away from $0$ it stays strictly positive. These properties will be useful in the forthcoming analysis of excursions and backtracks of the CRW.

\begin{lemma}[Controlling the change in potential]\label{lm:good-potential}
Fix $a < \Theta^{-1} \beta_{0} - 1/2$. There exist positive constants $C=C(a), c=c(a), q$ such that for $n$ large enough, any stopping time $\eta$ and any $s\geq 0$ the following hold with probability at least $ 1 - Ce^{-c\log^2 n}$:

\begin{equation}\label{eq:good_potential_eq1}
\id_{\{T \geq \eta + s\}}\cdot\Pp\left( \cbr{\PP_{\eta+s} - \PP_{\eta} \le as} \cap\cbr{ \tau_c \geq \eta + s} | \filF_{\eta} \right) \le  C (e^{-c s} +e^{-c\log^2 n}),
\end{equation}
\begin{equation}\label{eq:good_potential_eq2}
\Pp\left(\forall_{u\leq T-\eta} {\PP_{\eta+u} - \PP_{\eta} + 1 >0} | \filF_{\eta} \right) \ge q,
\end{equation}
\begin{equation}\label{eq:good_potential_eq3}
\Pp\left(\cbr{\exists_{s\leq u\leq T-\eta} {\PP_{\eta+u} - \PP_{\eta} + 1 \leq a s}} \cap\cbr{ \tau_c \geq \eta + s} | \filF_{\eta} \right) \le  C (e^{-c s} +e^{-c\log^2 n}).
\end{equation}
\end{lemma}

\begin{proof}
Let $\sigma = \inf\cbr{t > 0 \colon \iota(\Zz_t) > n^{\alpha}} = \tau_{iso}^{\alpha}$. Set $\tau = \tau_c \wedge \sigma$. Let
\begin{displaymath}
\mathcal{E} = \cbr{\PP_{\eta+s} - \PP_{\eta} \le as} \cap\cbr{ \tau_c \geq \eta + s}\cap\{T \ge \eta+s\}.
\end{displaymath}
By Corollary \ref{cor:isoperimetry-conditional} and \eqref{eq:potential}, with probability at least $1 - Ce^{-c\log^2 n}$,
\begin{align*}
\p(\mathcal E|\mathcal F_\eta) &\le \p(\mathcal{E} \cap  \cbr{\sigma >\eta + s} |\filF_\eta) + C e^{-c\log^2 n} \\
&=  \p(\{|\Zz_{\eta+s}| - |\Zz_\eta| \le (a+1)s\} \cap \cbr{\tau \geq \eta + s} |\mathcal{F}_\eta)
+ C e^{-c\log^2 n}.
\end{align*}

Recall that $\mu_t$ denotes the intensity of $|\Zz_t|$. By \eqref{eq:intensity-mu} of Corollary \ref{cor:intensity} and the definition of $\sigma$, for $t < \tau$
we have
\begin{align}\label{eq:estimate-on-intensity-mu}
\mu_{t} \geq \frac{2 \Theta^{-1} \beta }{n-1} (n-\iota(\Zz_t) ) \geq \frac{2 \Theta^{-1} \beta}{n-1} (n-n^{\alpha})\geq 1+2a =: \beta'>1,
\end{align}
for $n$ large enough. Thus, setting $\Lambda_t = \int_0^t \mu_u \, du$, we obtain
\begin{multline*}
\p(\{|\Zz_{\eta+s}| - |\Zz_\eta| \le (a+1)s\} \cap \cbr{\tau \geq \eta + s} |\mathcal{F}_\eta)\\
\le \p(\{|\Zz_{\eta+s}| - |\Zz_\eta| \le (a+1)s\} \cap \cbr{\Lambda_{\eta+s} - \Lambda_\eta \ge \beta's} |\mathcal{F}_\eta).
\end{multline*}

Recalling that $a = \frac{1}{2}(\beta' - 1) > 0$ and using Lemma \ref{lm:concentration-lower-bound-super} with $\ell = \beta's$, $\delta = \frac{\beta' - 1 - a}{\beta'}$ we obtain that the right-hand side above is almost surely bounded by
\begin{displaymath}
e^{- s \beta' \frac{1}{2} \left( \frac{\beta' - 1}{2 \beta'} \right)^2 } = e^{- c s}.
\end{displaymath}
for some $c > 0$. This proves \eqref{eq:good_potential_eq1}.

Let us now prove \eqref{eq:good_potential_eq2}. Denoting the event there by $\mathcal{E}$ and noticing that once the CRW closes into a cycle the potential stays constant, we estimate
\begin{align*}
	 \p(\mathcal E| \filF_\eta) &= \p(\{\forall_{u \in [0,T\wedge \tau_c - \eta]}  \PP_{\eta+u} - \PP_\eta+1\ge  0 \}|\mathcal{F}_\eta)\\
 & \ge \p(\{\forall_{u \in [0,T\wedge \tau_c - \eta]}  |\Zz_{\eta+u}| - |\Zz_\eta| \ge  u-1 \}  \cap  \{\sigma >T\} |\filF_\eta) \\
& \ge \p(\{\forall_{u \in [0,\tau - \eta]}  |\Zz_{\eta+u}| - |\Zz_\eta| \ge  u-1 \}  \cap  \{\sigma >T\} |\filF_\eta)\\
&\ge \p(\{\forall_{u \in [0,\tau - \eta]}  |\Zz_{\eta+u}| - |\Zz_\eta| \ge  u-1 \}|\filF_\eta) - \p(\sigma \le T|\mathcal{F}_\eta).
\end{align*}
By Corollary \ref{cor:isoperimetry-conditional} the second term is with probability at least $1 - \wsp$ bounded by $ C e^{-c \log^2 n}$ and thus negligible. As for the first term, note that thanks to \eqref{eq:estimate-on-intensity-mu} the intensity $\mu_t$ is bounded away from $1$ for $n$ large enough and $t \in [\eta,\tau]$, so Lemma \ref{le:counting-process-half-line} guarantees that the increase of the associated counting process, i.e., $|\Zz_{\eta+u}|$, is always at least $u-1$, with probability bounded from below by some $q > 0$ (observe also that Lemma \ref{le:counting-process-half-line} does not require any of the involved stopping times to dominate the other one). This concludes the proof of \eqref{eq:good_potential_eq2}.

Now we pass to the proof of \eqref{eq:good_potential_eq3}. Again denoting the event there by $\mathcal{E}$ and noticing that once the CRW closes into a cycle the potential stays constant, we estimate
\begin{align*}
\p(\mathcal E|\mathcal F_\eta) &\le \Pp\left(\cbr{\exists_{s\leq u\leq T\wedge \tau_c-\eta} {\PP_{\eta+u} - \PP_{\eta}  \leq a s-1}}  \cap  \cbr{\sigma >T} |\filF_\eta\right) + \p(\sigma \leq T|\filF_\eta) \\
&\le \Pp\left(\exists_{u \in [s,\tau-\eta] } |\Zz_{\eta+u}| - |\Zz_{\eta}| \le u + as |\filF_\eta\right) +  C e^{-c\log^2 n},
\end{align*}
where the second inequality holds with probability at least $1 - \wsp$ (for some $C,c > 0$), by Corollary \ref{cor:isoperimetry-conditional} applied to the second term. Similarly as in the proof of \eqref{eq:good_potential_eq2}, the intensity of $|\Zz_{\eta+u}|$ is bounded away from $1$, so
by Lemma \ref{le:lower-bound-on-interval} (applied with $\beta'$ instead of $\beta$), the first term is almost surely bounded by $e^{-cs}$ for some $c > 0$ depending only on $\beta'$.
\end{proof}

\end{subsection}

\begin{subsection}{Excursions}\label{sec:excursions}

In this section we introduce lemmas which in the final proof will help us show that the CRW with probability bounded away from zero may leave $L_0$ (or more generally the bad set) and move far away from it (in the metric of the graph $G_t$), thus making a quick return difficult.

The first lemma  is of technical nature and asserts that there is a non-negligible probability that the CRW will make a move within its current column. The second lemma will be crucial in proving that the CRW with high probability will not backtrack to $L_0$.

\begin{lemma}[Jumps within columns are quite likely]\label{lm:l-jumps}
Let $k \geq 0$ and let $D(v)$ denote the column containing vertex $v$. Let $X_{k} := \crw_{T_{k}}$. There exist $C, p > 0$ such that
\begin{align}\label{eq:l-jumps}
\Pp( \{ X_{k+1} \in D(X_{k}) \} \cap \{ T_{k+1} < T_k+1 \}| \mathcal{G}_{k}) \geq p \id_{\{ T_k \leq T \}}
\end{align}
with probability at least $1 - C e^{- c \log^2 n}$ for some $C, c > 0$.
\end{lemma}

\begin{proof}
Let $D_t$ (resp. $L_t$) denote the number of jumps of the CRW to a previously unexplored vertex in the same column (resp. row). Then up to time $\tau_c$ by Lemma \ref{le:intensity} their intensities (denote them by $\delta_t,\nu_t$ resp.) satisfy $\delta_t \Theta^{-1}  \le \frac{\beta}{n-1}|A_t\cap D(\crw_t)| \le \delta_t\Theta $ and $\nu_t\Theta^{-1} \leq \frac{\beta}{n-1} |A_t\cap L(\crw_t)|\le \nu_t\Theta$. In particular, on the event $\{t < \tau_{iso}^\alpha\wedge \tau_c\}$ we have
\begin{align}\label{eq:same-column-intensity-1}
  m := \Theta^{-1} \frac{\beta}{n-1} (n - n^\alpha) \le \delta_t,\nu_t.
\end{align}
We also have trivially
\begin{align}\label{eq:same-column-intensity-2}
  \delta_t,\nu_t \le \Theta\beta.
\end{align}

Consider the event $A = \{\tau_{iso}^\alpha > T_k\}$ and two stopping times $\rho = \inf\{t>T_k\colon D_t > D_{T_k}\}$ and $\gamma = \inf\{t > T_k\colon L_t > L_{T_k}\}$. Define also $\mathcal{E} = \{\rho < \gamma \wedge (T_k+1)\}$. The lemma will follow once we prove that almost surely
\begin{align}\label{eq:same-column}
\p(\mathcal{E}|\mathcal{G}_k) \ge p\Ind{A}.
\end{align}

Indeed, the event $\{\Ind{A} < \Ind{\{T_k \le T\}}\}$ is contained in $\{\tau_{iso}^\alpha \le  T\}$ which by Assumption \ref{lm:isoperimetry} has probability at most $C e^{- c\log^2 n}$.

Let $B$ be any element of $\mathcal{G}_k$ and denote $P := \p(\mathcal{E}\cap A\cap B)$. Note that on $A$ we have $T_k < \infty$. Observe also that with probability one $\rho \neq \gamma\wedge (T_k+1)$ and so
\begin{displaymath}
P = \E (D_{\rho\wedge \gamma\wedge (T_k+1)} - D_{T_k})\Ind{A\cap B} = \E \int_{T_k}^{\rho\wedge \gamma\wedge(T_k+1)} \delta_s \Ind{A\cap B} \, ds,
\end{displaymath}
where we used Doob's theorem and the fact that $A, B \in \mathcal{G}_k$.

Since between $T_k$ and $\rho \wedge \gamma\wedge (T_k+1)$ the quantity $\iota(\Zz_s)$ does not change and the CRW does not close into a cycle, we can use \eqref{eq:same-column-intensity-1} to estimate
\begin{align}\label{eq:bound-on-P}
  P \ge &\E (\rho\wedge \gamma\wedge(T_k+1) - T_k) m \Ind{A\cap B} \nonumber\\
  =& m \E (\rho - T_k)\Ind{A\cap B\cap \mathcal{E}} + m\E (\gamma\wedge (T_k+1) - T_k)\Ind{A\cap B} \nonumber\\
  &- m \E(\gamma\wedge (T_k+1) - T_k)\Ind{A\cap B\cap \mathcal{E}}\nonumber\\
  & = m\E (\gamma\wedge (T_k+1) - T_k)\Ind{A\cap B} - m\E (\gamma\wedge (T_k+1) - \rho)\Ind{A\cap B\cap \mathcal{E}}.
\end{align}
Note that
\begin{align*}
	\E (\gamma\wedge (T_k+1) - \rho)\Ind{A\cap B\cap \mathcal{E}} & \le \frac{1}{m} \E \int_\rho^{\gamma\wedge(T_k+1)} \nu_s \, ds \Ind{A\cap B\cap \mathcal{E}} \\
   & = \frac{1}{m} \E (L_{\gamma\wedge(T_k+1)} - L_\rho) \Ind{A\cap B\cap \mathcal{E}} = \frac{P}{m},
\end{align*}
where in the first equality we used Doob's theorem (note that $A\cap B\cap \mathcal{E} \in \mathcal{F}_{\rho}$) and in the second one the observation that on $\mathcal{E}$ we have $L_{\gamma\wedge(T_k+1)} - L_\rho \le 1$.

Combining the above inequality with \eqref{eq:bound-on-P}, we get
\begin{align}\label{eq:2nd-bound-on-P}
  P \ge \frac{m}{2}\E (\gamma\wedge (T_k+1) - T_k)\Ind{A\cap B}.
\end{align}
Integrating by parts we get
\begin{align*}
  P &\ge \frac{m}{2} \int_0^1 \E \Ind{A\cap B} \p( \gamma - T_k > s|\mathcal{G}_k) \, ds \\
  &= \frac{m}{2} \int_0^1 \E \Ind{A\cap B} \p(\{L_{T_k+s} - L_{T_k} = 0\}|\mathcal{G}_k) \, ds.
\end{align*}
Using \eqref{eq:same-column-intensity-2} and the concentration estimate from Lemma \ref{lm:concentration-upper-bound} with $r=0$, we get that
\[
\p(\{L_{T_k+s} - L_{T_k} = 0\}|\mathcal{G}_k) \ge \p(X = 0),
\]
where $X$ is a Poisson variable with parameter $\Theta \beta s$. Thus
\begin{displaymath}
  P \ge \p(A\cap B) \frac{m}{2}\int_0^1 e^{-\Theta \beta s} \, ds.
\end{displaymath}
Setting $p = \frac{m}{2}\int_0^1 e^{-\Theta \beta s} \, ds $ and using the definition of $P$ together with the fact that $A\in \mathcal{G}_k$, we get
\begin{displaymath}
  \E \left( \p(\mathcal{E}|\mathcal{G}_k)\Ind{A}\Ind{B} \right) = \p(\mathcal{E}\cap A\cap B) \ge \E p \Ind{A} \Ind{B}
\end{displaymath}
for all $B\in \mathcal{G}_k$, which implies \eqref{eq:same-column} and concludes the proof.
\end{proof}
\paragraph{Excursions.}
Let $t$ be a stopping time. \hfff{def:excursion}{We will say that the CRW makes an \emph{excursion} of length $k$ starting at time $t$}, the event which we denote by $\EE_{t}(k)$, if there exists $s > 0$ such that the following conditions hold:
\begin{itemize}
\item $\Zz_{[t, t+s]} \cap (\Zz_{t-} \cup L_{0}) = \emptyset$,
\item $\X_{t+s}$ is at distance $k$ from $\X_{t}$ in $G_{t+s}$ or $t+s = T$.
\end{itemize}
Note that in particular if $t+s < T$ then the CRW has to discover at least $k$ previously unexplored vertices and it is possible for $\X_{[t, t+s]}$ to intersect itself. The condition $t+s = T$ is included as we are interested in the CRW only up to time $T$.

\begin{lemma}[Excursions are quite likely]\label{lm:excursions}
There exists $C,c,q >0$ such that for any $l \geq 1$ we have
\begin{align}\label{eq:excursion-T-k}
\Pp\left( \EE_{T_{l}}(n^\varepsilon) | \mathcal{G}_{l} \right) \geq q\cdot \id_{\cbr{\crw_{T_l}\notin L_0}\cap \cbr{T_l \leq T}},
\end{align}
with probability at least $1-\wsp$.
\end{lemma}

\begin{proof}

Throughout the proof we assume that $\cbr{\crw_{T_l}\notin L_0}\cap \cbr{T_l \leq T}$ holds.

Let $\tau = \inf \{ u \in (T_{l}, \infty) \, \colon \, \mbox{$\crw_{u}$ makes an internal jump}\}$ and $\sigma = \inf \cbr{u \in (T_{l}, \infty) \colon \crw_u \text{ makes a direct jump to } L_0}$. For $k,s \in \N$ let $s' = s \wedge (T - T_{l})$ and
\begin{equation}\label{eq:tree_exploration}
	\tilde{\mathcal{E}}(k,s)  := \cbr{\tau > T_l+s} \cap \cbr{\forall_{u\leq s'} \PP_{T_l+u} -\PP_{T_l-}> 0}\cap (\cbr{\PP_{T_l+s} -\PP_{T_l}\geq k}\cup\{T_l+s>T\})\cap \cbr{\sigma > T_l+s}.
\end{equation}

The lemma will follow once we prove $\tilde{\mathcal{E}}(k,s) \subset \mathcal{E}_{T_l}(k)$ and that with high probability $\Pp(\tilde{\mathcal{E}}(n^\varepsilon,s)|\mathcal{G}_{l} )\geq q$ for some $s > 0$ and some constant $q > 0$.

The first two conditions of \eqref{eq:tree_exploration} imply that the subgraph $G^l_{s'}$ of $G_{T_l+s'}$ induced by the exploration $\crw_{[T_l, T_{l} + s']}$ is a tree and the CRW does not revisit $\Zz_{T_l-}$, i.e., $\Zz_{[T_l, T_l+s']}\cap \Zz_{T_l-} = \emptyset$. Indeed, as there are no internal jumps the CRW can revisit $\Zz_{T_l-}$ only by backtracking the bridge used at time $T_l$. This happens only when all vertices $\Zz_{[T_l, T_l+u]}$ are dead at some time $u\leq s'$ (we again use the fact that there are no internal jumps). This is equivalent to $\PP_{T_l+u} -\PP_{T_l-}=0$ which is impossible. Once we know that $\Zz_{[T_l, T_l+s']}\cap \Zz_{T_l-} = \emptyset$ and there are no internal jumps during $[T_l, T_l+s']$, the only possibility left is that the exploration is a tree. As a corollary we observe that the first two conditions of \eqref{eq:tree_exploration} imply that the CRW does not close into a cycle, i.e., $\tau_c\geq T_l+s'$.

If $T_l + s > T$, then the second condition from the definition of an excursion is trivially satisfied. Assume therefore that $T_l + s \leq T$, so that $s' = s$ in the argument above. As $G^l_{s}$ is a tree, one sees that the distance $d_{G^l_{s}}$ of the vertex $\crw_{T_l+s}$ from $\crw_{T_l}$ is $f-b$, where $f$ (resp. $b$) is the number of fresh jumps (resp. backtracks) during time $(T_l, T_l+s]$.

Notice that $f=|\Zz_{(T_l, T_l+s]}| - 1 = |\Zz_{T_l+s}| - |\Zz_{T_l}|$, as there are no internal jumps.

Moreover, we have $b\leq s$, since a backtrack occurs only once a whole bar has been exhausted (i.e., the corresponding vertex has become dead) and the CRW moves at unit speed. Altogether this implies
\begin{equation*}
	d_{G^l_{s}}(\crw_{T_l+s}, \crw_{T_l})  \geq  |\Zz_{T_l+s}| - |\Zz_{T_l}| - s  = \PP_{T_l+s} -\PP_{T_l}.
\end{equation*}
To conclude we notice that $d_{G_{T_l+s}}(\crw_{T_l+s}, \crw_{T_l}) = d_{G^l_{s}}(\crw_{T_l+s}, \crw_{T_l})$ and by the third condition of \eqref{eq:tree_exploration} the assumption $T_l+s \le T$ implies that the right hand side above is at least $k$.

The final condition of \eqref{eq:tree_exploration} together with $\cbr{\crw_{T_l}\notin L_0}$ ensures that $\Zz_{[T_l, T_l+s]}\cap L_0 = \emptyset$.

These arguments proved that $\tilde{\mathcal{E}}(k,s) \subset \mathcal{E}_{T_l}(k)$, with the excursion taking total time $s'$. Now we are left with showing that with probability at least $1 - \wsp$ we have $\Pp(\tilde{\mathcal{E}}(n^\varepsilon,s)|\filG_l)\geq q$ for some $s > 0$ and $q>0$. For this part we fix $k=n^\varepsilon$ and $s = k/a$, for some $a\in(0,\Theta^{-1} \beta - 1/2)$.  We recall that the first two conditions of \eqref{eq:tree_exploration} imply $\tau_c\geq T_l+s'$. Thus we have
\begin{align*}
	\p(\tilde{\mathcal{E}}(k, s)|\filG_l) &\geq \p(\cbr{\tau > T_l+s} \cap \cbr{\forall_{u\leq s'} \PP_{T_l+u} -\PP_{T_l-}> 0}\cap \cbr{\sigma >T_l+s}|\filG_l) \\
	&\qquad \qquad \qquad- \p(\cbr{\PP_{T_l+s} -\PP_{T_l} < k}\cap\{T_l+s\le T\}\cap \cbr{\tau_c\geq T_l+s'}|\filG_l)\\
	&\geq \p(\forall_{u\leq s'} \PP_{T_l+u} -\PP_{T_l-}> 0|\filG_l) - \p({\tau \leq T_l+s}|\filG_l) - \p({\sigma \leq T_l+s}|\filG_l) \\
	&\qquad \qquad \qquad- \p(\cbr{\PP_{T_l+s} -\PP_{T_l} < k}\cap\{T_l+s\le T\}\cap \cbr{\tau_c\geq T_l+s}|\filG_l),
\end{align*}
where on the event $\{T_l+s\le T\}$ we replaced $s'$ with $s$.

By Lemma \ref{lm:good-potential} with probability at least $1 - \wsp$ (for some $C,c > 0$) we have $ \p(\forall_{u\leq s'} \PP_{T_l+u} -\PP_{T_l-}> 0|\filG_l) \geq q$ for some $q>0$ (notice that $\PP_{T_l-} + 1 = \PP_{T_l}$). Thus to conclude the proof it is enough to show that the other terms are $o(1)$. Fix $\delta > 2 \Theta \beta$. We have
\begin{equation*}
	\p(\tau \leq T_{l}+s|\filG_l) \leq \p(\cbr{\tau \leq T_{l}+s}\cap \cbr{T_{l+\lceil \delta s\rceil} \geq T_{l}+s}|\filG_l) + \p({T_{l+\lceil \delta s\rceil} < T_{l} + s}|\filG_l).
\end{equation*}
The second term is $o(1)$ by Lemma \ref{lm:fresh-vertices-are-not-often}. For the first term we write
\begin{equation*}
	\p(\cbr{\tau \leq T_{l}+s}\cap \cbr{T_{l+\lceil \delta s\rceil}>T_{l}+s}|\filG_l)  = \p(\cbr{N_s \geq 1}\cap \cbr{T_{l+\lceil \delta s\rceil}>T_{l}+s}|\filG_l),
\end{equation*}
where $N_u = I_{T_l+u} - I_{T_l}$ is the number of internal jumps during time $[T_l, T_l+u]$. Recall that by Corollary \ref{cor:isoperimetry-conditional} the event $\cbr{\iota(\Zz_{T_l})\leq n^{\varepsilon}}$ has high conditional probability. On this event for $u\leq s\wedge(T_{l+\lceil \delta s\rceil}-T_l)$  we bound the intensity $\lambda_u$ of $N_u$ by \eqref{eq:intensity-lambda} from Lemma~\ref{le:intensity}
\begin{equation*}
	\lambda_u \leq \frac{4 \Theta \beta}{n}\iota(\Zz_{T_l+u}) \leq \frac{4 \Theta \beta}{n}(\iota(\Zz_{T_l}) + \lceil \delta s\rceil) \leq \frac{4 \Theta \beta}{n}(n^\varepsilon + \delta n^\varepsilon/a+1) =: \bar \lambda.
\end{equation*}
Now we have $\bar\lambda\cdot s \leq \bar\lambda\cdot n^\varepsilon/a = C n^{2\varepsilon-1}$, for some $C>0$. Thus by Markov's inequality, Doob's theorem and monotonicity of $N$,
\begin{align*}
\p(\cbr{\tau \leq T_{l}+s}\cap \cbr{T_{l+\lceil \delta s\rceil}>T_{l}+s}|\filG_l)\ind{\iota(\Zz_{T_l})\leq n^{\varepsilon}} &\le \E(N_{s\wedge(T_{l+\lceil\delta s\rceil} - T_l)}\ind{\iota(\Zz_{T_l})\leq n^{\varepsilon}}  |\mathcal{G}_l) \\
&
\le \E (\bar{\lambda}s \ind{\iota(\Zz_{T_l})\leq n^{\varepsilon}}|\mathcal{G}_l) \leq C n ^{2 \varepsilon - 1} = o(1),
\end{align*}
 and consequently with probability at least $1 - \wsp$ (for some $C,c > 0$)
\begin{equation*}
	\p(\tau \leq T_{l}+s|\filG_l) = o(1).
\end{equation*}

Analogously one can show that
\begin{equation*}
	\p(\sigma \leq T_{l} + s|\filG_l) = o(1).
\end{equation*}
Indeed, by Lemma~\ref{le:intensity} we can easily estimate the intensity of direct jumps by $\lambda_u^d\leq \bar \lambda$, as $|A_u \cap L_0|\leq 1$.

Finally, by the choice of $s$ above and \eqref{eq:good_potential_eq1} in Lemma \ref{lm:good-potential} we have
\begin{equation}
	\p(\cbr{\PP_{T_l+s} -\PP_{T_l}< k}\cap\{T_l+s\le T\}\cap \cbr{\tau_c\geq T_l+s'}|\filG_l) = o(1),
\end{equation}
thus the proof is finished.

\end{proof}

As in the forthcoming proofs we will need to apply Lemma \ref{lm:l-jumps} and Lemma \ref{lm:excursions} for random $k$, we state the following easy corollary

\begin{corollary}\label{cor:random-stopping-time}
Let $\tau$ be a stopping time such that with probability one $\tau \in \{T_i\colon i \le n^2\}$ and let $\tau'$ be such that $\tau' = T_{i+1}$ on the event $\{\tau = T_{i}$\}. Then with probability at least $1 - Ce^{-c\log^2 n}$,
\begin{align}\label{eq:l-jumps-random}
\Pp( \{ \crw_{\tau'} \in D(\crw_{\tau}) \} \cap \{ \tau' < \tau + 1 \}| \mathcal{F}_{\tau}) \geq p \id_{\{ \tau \leq T \}}
\end{align}
and
\begin{align}\label{eq:excursion-tau}
\Pp\left( \EE_{\tau}(n^\varepsilon) | \mathcal{F}_{\tau} \right) \geq q\cdot \id_{\cbr{\crw_{\tau}\notin L_0}\cap \cbr{\tau \leq T}}.
\end{align}
\end{corollary}
\begin{proof}
It is enough to note that
\begin{align*}
\Pp( \{ \crw_{\tau'} \in D(\crw_{\tau}) \} \cap \{ \tau' < \tau + 1 \}| \mathcal{F}_{\tau}) = \sum_{k=1}^{n^2} \Pp( \{ X_{k+1} \in D(X_{k}) \} \cap \{ T_{k+1} < T_k + 1 \}| \mathcal{G}_{k})\Ind{\{\tau = T_k\}},
\end{align*}
so if \eqref{eq:l-jumps-random} does not hold then \eqref{eq:l-jumps} fails for some $k \leq n^2$. Now the statement follows from Lemma \ref{lm:l-jumps} and a union bound over $k$.

Analogously, we can write
\begin{align*}
  \Pp\left( \EE_{\tau}(n^\varepsilon) | \mathcal{F}_{\tau} \right) & = \sum_{l=1}^{n^2}  \Pp\left( \EE_{\tau}(n^\varepsilon) | \mathcal{F}_{\tau} \right)\Ind{\{\tau = T_l\}}  = \sum_{l=1}^{n^2}  \Pp\left( \EE_{T_l}(n^\varepsilon) | \mathcal{G}_l \right)\Ind{\{\tau = T_l\}},
\end{align*}
and \eqref{eq:excursion-tau} fails only if \eqref{eq:excursion-T-k} fails for some $l \le n^2$. Again we finish by Lemma \ref{lm:excursions} and a union bound.
\end{proof}

\end{subsection}

\begin{subsection}{The bad set}\label{sec:bad-set}

Recall that the bad set $\badSet_{t}$, defined in \eqref{eq:defBadSet}, consists of vertices which are close in $G_{t}$ to $L_{0}$. Our goal is to show that the CRW is unlikely to ever hit the bad set. Since it is difficult to know exactly which part of $\badSet_{t}$ is accessible to the CRW at time $t$, we will in fact prove a stronger statement, namely that with high probability $\badSet_{t}$ itself is small for all $t \leq T$.

\paragraph{The core of $G_t$.} We now introduce a special subgraph of $G_t$, which will play an important role in the analysis of bad jumps and backtracks.

For a graph $G$ and a vertex $v\in V_G$ denote by $\deg_G(v)$ the degree of $v$ in $G$, counted with multiplicities. Below, to simplify the notation we will often write $v \in G$ instead of $v \in V_G$.

\hfff{def:core}{Let $\core_{t}$ be the subgraph of $G_{t}$ obtained by successively removing dead vertices of degree one} (i.e., we remove dead vertices of degree one in $G_t$, obtaining the graph ${G}_t^{(1)}$, next we remove dead vertices of degree one in ${G}_t^{(1)}$, etc. until no more vertices can be removed). We will call the graph $\core_t$ the \emph{core} of $G_t$ (this is similar to what is called the $2$-core of $G_t$ in graph theory, except that we allow possibly two vertices of degree one which are not dead). See Figure \ref{fig:core} for an example of a graph $G_{t}$ and its core.

The procedure described above corresponds to removing trees consisting of dead vertices, connected to the core. Note that $\core_{t}$ can still contain dead vertices and it is not necessarily nondecreasing in $t$. The role the graph $\core_{t}$ will play in subsequent arguments is twofold. In the analysis of the size of the bad set, it will be convenient to handle the intersection of the bad set with the core and the trimmed trees separately. In the subsequent part we will also use the special structure of the core to show that if the CRW is outside the bad set then it is very unlikely to backtrack all the way to $L_0$.

\begin{figure}[h]
	\centering
		\includegraphics[scale=0.4]{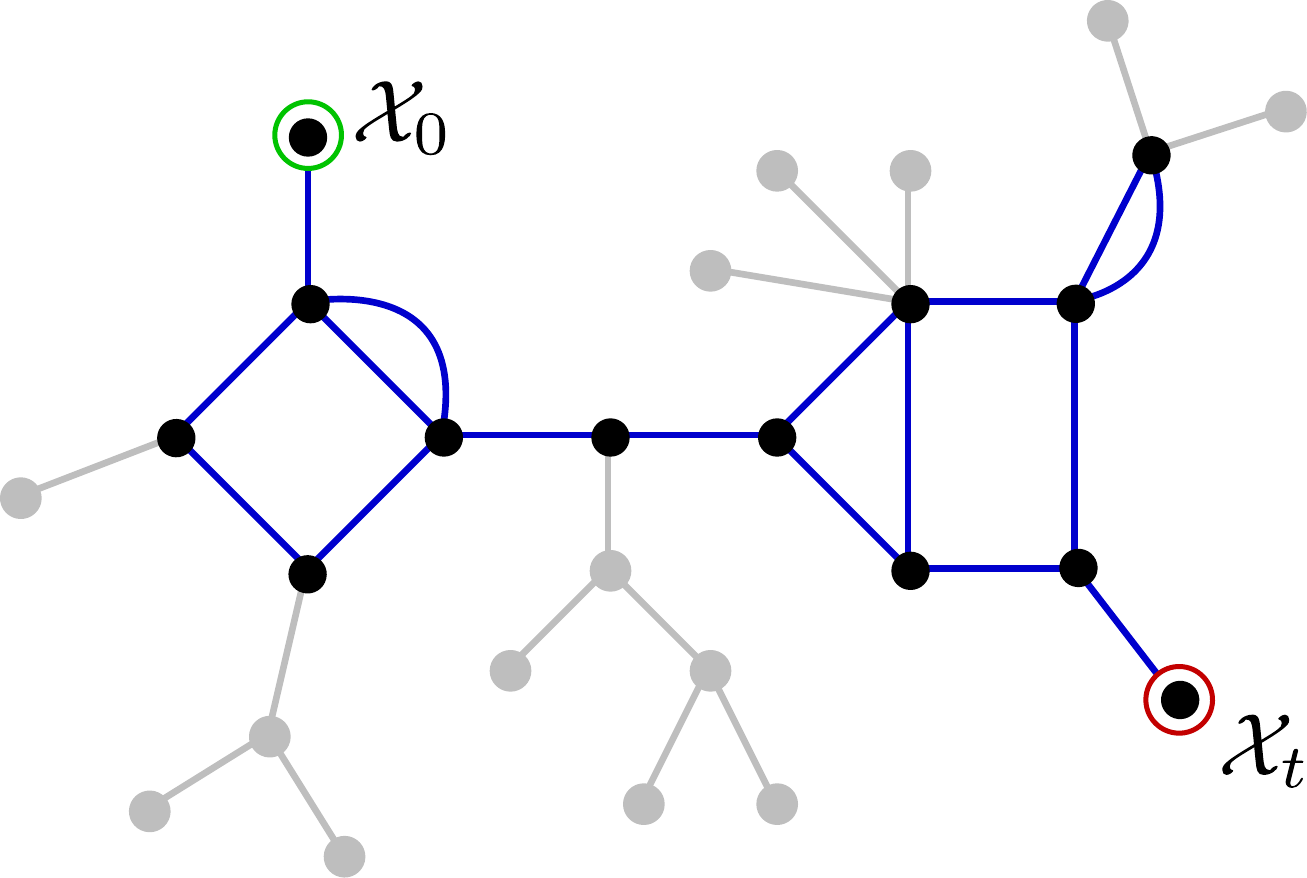}
	\caption{Graph $G_t$ and its core $\core_t$ (edges shown as thicker blue lines). Parts of the graph removed in the construction of the core are shown in grey. The starting and end vertices of the CRW, $\crw_{0}$ and $\crw_{t}$, have been marked in green and red, respectively.}
\label{fig:core}
\end{figure}

The degrees of vertices in $\core_t$ are in a simple relation with the number of internal jumps $I_t$.
\begin{lemma}
For any $t \ge 0$,
\begin{displaymath}
	\sum_{v \in \core_t} (\deg_{\core_t}(v) - 2) - 2I_t = - 2.
\end{displaymath}
\end{lemma}

\begin{proof}
We will first prove that for any $t \ge 0$,
\begin{displaymath}
	\sum_{v \in G_t} (\deg_{G_t}(v) - 2) - 2I_t = - 2.
\end{displaymath}

The expression on the left hand side equals $-2$ at time $t = 0$ and changes only when the CRW makes a fresh jump. The increase of the degrees caused by such a jump is compensated either by the summand $-2$ (if the CRW explores a new vertex) or by the increase of $I_t$ (in the case of internal jumps). This proves the above formula.

To pass from $G_t$ to $\core_t$ note that whenever one removes a vertex of degree one then the sum of degrees decreases by $2$ and the number of vertices decreases by one, so the sum in question does not change.

\end{proof}

\hfff{lab:event-d-t}{For $\delta > 0$ let us define}
\begin{equation*}
\mathcal{D}_{t}^\delta := \left\{ \sum\limits_{v \in \core_{t}} \left( \deg_{\core_t}(v) - 2 \right) \leq n^{\delta} \right\}.
\end{equation*}
Combining the above lemma with Lemma \ref{lm:internal-hits} we obtain immediately

\begin{lemma}[The core has few excess edges]\label{lm:internal-hits-degree3}
There exist $C, c > 0$ such that
\[
\Pp\left( \forall_{t\leq T} \mathcal{D}_{t}^{2\alpha}\right) \ge 1 - C e^{- c \log^2 n}.
\]
\end{lemma}

Moreover, the intersection of the core and the bad set is small, as asserted in the following lemma
\begin{lemma}\label{lm:core-bad-ball-is-small}
There exist $C, c > 0$ such that
\[
\pr{\sup\limits_{t \leq T} |V_{\core_{t}}\cap \badSet_t| \geq n^{4\varepsilon} } \leq C e^{- c \log^2 n}.
\]
\end{lemma}

\begin{proof}
We will first estimate the size of the $k$-neighborhood of any vertex in $\core_{t}$ for arbitrary $k \geq 1$.

To this end, fix any vertex $r \in \core_t$ and consider a spanning tree $\mathcal{T}$ of $\core_{t}$ obtained by a breadth first search starting from $r$, so that the distances between $r$ and any other vertex of $\core_t$ are the same in $\core_t$ and in $\mathcal{T}$.

Let $B_{i}$ denote the ball $B_{\mathcal{T}}(r, i)$ (we assume $B_{-1} = \emptyset$) and let $\deg_{\mathcal{T}}(w)$ be the degree of $w$ in $\mathcal{T}$. Obviously we have $\deg_{\mathcal{T}}(w) \leq \deg_{\core_t}(w)$. We have $|B_{0}| = 1$, $|B_{1}| - |B_{0}|= \deg_{\mathcal{T}}(r)$ and for any $i \geq 1$,
\[
|B_{i+1}| - |B_{i}| = \sum\limits_{w \in B_{i} \backslash B_{i-1}} \left( \deg_{\mathcal{T}}(w) - 1 \right),
\]
since $\mathcal{T}$ is a tree. For $j \geq 1$ we can sum these equalities from $i=0$ to $j-1$, getting
\[
|B_{j}| - 1 = \sum\limits_{w \in B_{j-1}} \left( \deg_{\mathcal{T}}(w) - 1 \right) + 1,
\]
so
\[
|B_{j}| = \sum\limits_{w \in B_{j-1}} \left( \deg_{\mathcal{T}}(w) - 2 \right) + |B_{j-1}| + 2.
\]
In particular
\begin{equation}\label{eq:chain-of-balls}
|B_{j}| \leq \sum\limits_{w \in B_{j-1}} \left( \deg_{\core_t}(w) - 2 \right) + |B_{j-1}| + 2.
\end{equation}

Note that $\core_t$ contains at  most two vertices of degree one, $\X_0$ and $\X_t$. Therefore
\begin{displaymath}
\sum\limits_{w \in B_{j-1}} \left( \deg_{\core_t}(w) - 2 \right)  \le 2+ \sum\limits_{w \in \core_t} \left( \deg_{\core_t}(w) - 2  \right).
\end{displaymath}
By Lemma \ref{lm:internal-hits-degree3} with probability at least $1 - Ce^{-c \log^2 n}$ (for some $C,c > 0$) for all $t \leq T$ the right hand side above is bounded by $n^{2\alpha} + 2 \leq n^\varepsilon + 2$. From this and \eqref{eq:chain-of-balls} we obtain
\[
|B_{j}| \leq  n^{\varepsilon} + 2 + |B_{j-1}| + 2
\]
with high enough probability. This in turn implies that for any $k \geq 1$ we have $|B_{k}| \leq k (n^{\varepsilon} + 5)$.

We now apply this estimate to bound the size of the intersection of the core with the bad set. For any $v \in \Zz_t\cap L_0$ let $G(v)$ be the induced subgraph of $\core_t$ with the set of vertices equal to $B_{G_t}(v,n^\varepsilon)\cap V_{\core_t}$. Note that it may happen that $v \notin \core_t$, but if $G(v)$ is nonempty, it is connected and of diameter at most $2n^\varepsilon$. Indeed, $\core_t$ is connected and distances in $\core_t$ between any two vertices $w,u
\in \core_t$ are the same as  in $G_t$ (since the shortest paths between any elements of the core are disjoint from the trees which are removed during its construction).

Now choose any $r \in G(v)$ and let $B_k$ be defined as in the first part of the proof. Taking $k = 2n^{\varepsilon} \ge \diam(G(v))$ we get $G(v) \subset B_k$, which by our bound on $|B_k|$ proves that for $n$ large enough with probability at least $1 - Ce^{-c \log^2 n}$, for any $t \le T$ and  $v \in \Zz_t\cap L_0$ we have
\begin{equation}\label{eq:size-of-core-ball}
|B_{G_t}(v,n^\varepsilon)\cap V_{\core_t}| \leq 3 n^{2 \varepsilon}.
\end{equation}
Now using the definition of the bad set we get
\begin{displaymath}
V_{\core_t}\cap \badSet_t   =
\bigcup_{v \in \Zz_t \cap L_0} B_{G_t}(v,n^\varepsilon)\cap V_{\core_t}.
\end{displaymath}
By Assumption \ref{lm:isoperimetry} with probability at least $1 - C_1 e^{- c_1 \log^2 n}$, for some $C_1, c_1 > 0$, we have $|\Zz_t \cap L_0| \le n^\varepsilon$. This estimate together with \eqref{eq:size-of-core-ball} concludes the proof.
\end{proof}
To prove that the whole bad set is small up to time $T$, we will show that dead trees removed in the construction of $\core_{t}$ cannot be too large and then estimate the total number of such trees.

\begin{proposition}[The bad set is small]\label{prop:bad-ball-is-small}
There exist $C, c > 0$ such that
\[
\pr{ |\badSet_{T}| \geq n^{7\varepsilon} } \leq C e^{- c \log^2 n}.
\]
\end{proposition}
\begin{proof}
Let $K$ be the graph removed from $G_T$ in the construction of $\core_T$. We will first prove that

\begin{equation}\label{eq:bound-on-Delta}
\pr{ |V_K\cap \badSet_T| \geq n^{6\varepsilon}} \leq C e^{- c \log^2 n}.
\end{equation}

To this end we first estimate the size of the largest connected component of $K$ (i.e., the largest tree removed in the construction of the core).
Let $\mathcal A$ be the event that there is a component $K'$ of size $l  \ge n^\varepsilon$. Note that after entering $K'$ for the first time, say at time $T_k$, the CRW traverses the whole tree $K'$ in a depth first search manner, exhausting all the bars corresponding to visited vertices. Due to the constant speed of the CRW, the time needed for this equals exactly $l$. Since the number of vertices and edges in a tree differ by one, during that time the CRW makes $l-1$ jumps to previously unexplored vertices and the same number of backtracks. It follows that
\begin{equation}\label{eq:potential_tmp}
	\mathcal{P}_{T_k+l} - \mathcal{P}_{T_k} = l-1 - l = -1.
\end{equation}
Using \eqref{eq:good_potential_eq3} from Lemma \ref{lm:good-potential} with $\eta = T_k$ and $s = n^{\varepsilon}/2$ we get that for fixed $k,l$  the probability that the condition \eqref{eq:potential_tmp} is fulfilled is smaller that $C e^{-c\log^2n}$ for some $C,c > 0$. Observing that $k,l\leq n^2$ and applying a union bound we get
\begin{equation}\label{eq:tmp_estimate_77}
\pr{\mathcal{A}} \leq C_1 e^{-c_1 \log^2 n},
\end{equation}
for some $C_1, c_1>0$.
To prove \eqref{eq:bound-on-Delta} it remains to bound the number of  trees with nonempty intersection with $\badSet_T$. Note that such a tree either
\begin{itemize}
\item[a)] is attached to a vertex from $\badSet_T\cap V_{\core_T}$

\item[b)] or  contains an element from $\Zz_T \cap L_0$.
\end{itemize}

The number of the former trees is at most the sum of degrees of vertices in $\badSet_T\cap V_{\core_T}$. The number of the latter trees equals at most $\iota(\Zz_T)$.

Now we show that with probability at least $1 - C e^{-c n^{\varepsilon}}$, for some $C, c > 0$, no vertex has degree greater than $n^{\varepsilon}$. Indeed, for a fixed vertex $w$ we can estimate the number of bridges incident to $w$ by using Lemma \ref{lm:poisson-approximation} (which gives a general bound on the number of bridges in a given subset of $E \times [0,1)$ in terms of a Poisson process). More specifically, we apply the second part of the lemma with $A = E_{w} \times [0,1)$, where $E_{w}$ is the set of all edges incident to $w$. Note that the Lebesgue measure of $A$ satisfies $|A| = 2(n-1)$. Thus we obtain that with probability at least $1 - C' e^{-c' n^{\varepsilon}}$, for some $C', c' > 0$, the total number of bridges incident to $w$ is at most $n^{\varepsilon}$. A union bound over $w$ finishes the argument.

Using Lemma \ref{lm:core-bad-ball-is-small}  with probability at least $1 - C e^{-c \log^2 n}$ we have $|V_{\core_{T}}\cap \badSet_T|\leq n^{4\varepsilon}$. Recall that the size of a single dead tree is smaller than $n^{\varepsilon}$ with high probability, see \eqref{eq:tmp_estimate_77}. Thus the total number of vertices belonging to trees from case a) above is at most $n^{4 \varepsilon} \cdot n^{\varepsilon} \cdot n^{\varepsilon} = n^{6 \varepsilon}$ with high probability.

On the other hand, by Assumption \ref{lm:isoperimetry} the total number of vertices from trees satisfying case b) is at most
$n^\varepsilon \cdot n^\varepsilon = n^{2\varepsilon}$ with high probability. Combining the two cases yields \eqref{eq:bound-on-Delta}.

Thus with the required probability, $|\badSet_{T}| = |\badSet_{T} \cap V_K| + |\badSet_{T} \cap V_{\core_{T}}| \leq n^{4\varepsilon} + n^{6\varepsilon} \leq n^{7\varepsilon}$ for $n$ large enough.

\end{proof}
\paragraph{Bad hits process.}\label{sec:bad-hits}

Here we prove that $I^b_T$, i.e., the number of bad jumps up to time $T$ is small with high probability. The argument is similar as in Lemma \ref{lm:internal-hits}, but more subtle.

\begin{lemma}\label{lm:bad-hits-process}
There exist $C,c,\delta >0 $ such that for any stopping time $\eta$  we have
\[
\Pp \left(I^{b}_{T} - I^{b}_{\eta} \geq 1 \big| \Ff_{\eta}\right) \leq \frac{1}{n^{\delta}},
\]
with probability at least $1 - \wsp$.
\end{lemma}

\begin{proof}
We will work on the event $\{\eta \leq T\}$ (on its complement the probability in question vanishes due to monotonicity of $I^b$). Let us set
\[
\tau = \inf \{ t \geq 0 \colon \iota(\Zz_{t}) > n^{\varepsilon} \mbox{ or }  |\badSet_{t}| > n^{7 \varepsilon} \}.
\]
Using a union bound we write
\begin{equation}\label{eq:bad_jumps_increase}
\Pp \left(I^{b}_{T} - I^{b}_{\eta} \geq 1 \big| \Ff_{\eta} \right) \leq \Pp \left(I^{b}_{T \wedge \tau} - I^{b}_{\eta} \geq 1 \big| \Ff_{\eta} \right) + \Pp\left( \tau \leq T \big| \Ff_{\eta} \right).
\end{equation}
Observe that $\Pp\left( \tau \leq T \big| \Ff_{\eta} \right) \leq e^{-c \log^2 n}$ with probability at least $1 - \wsp$ (for some $C,c > 0$), since by Proposition \ref{prop:bad-ball-is-small} the bad set is small with high probability (and thus, by Lemma \ref{le:easy-lemma}, also conditionally on $\Ff_{\eta}$), and by Corollary \ref{cor:isoperimetry-conditional} $\iota(\Zz_t)$ is small with high probability conditionally on $\Ff_{\eta}$.

To deal with the first term we estimate the intensity $\lambda^b$ of $I^b$. For any $t$ such that $\eta \leq t < T \wedge \tau$ let $D(\badSet_{t})$ (resp. $L(\badSet_{t})$) denote the set of columns (resp. rows) which have non-empty intersection with $\badSet_{t}$. Formally, $D(\badSet_{t}) = \cbr{D_i: D_i\cap \badSet_{t}\neq \emptyset}$ (and likewise for $L(\badSet_{t})$). Let $K_{t} = D(\badSet_{t}) \cup L(\badSet_{t})$. Recall that by $A_{t}$ we denote the set of vertices accessible at time $t$ by a fresh jump. Observe that if $w=\X_{t}$ is a vertex in a row or column belonging to $K_t$, then $|A_t \cap \badSet_t|\leq n^{7\varepsilon}$, otherwise  $|A_t \cap \badSet_t| = 0$. Note also that $A_t=\emptyset$ if $t\geq \tau_c$. By Lemma~\ref{le:intensity} we have
\[
\Lambda := \int\limits_{\eta}^{T \wedge \tau} \lambda^{b}_{t} \, dt \leq \frac{ \Theta \beta}{n-1} \int\limits_{\eta}^{T \wedge \tau}  | A_{t}\cap \badSet_{t}| \,  dt \leq \frac{ \Theta \beta}{n-1} n^{7\varepsilon} \int\limits_{\eta}^{T \wedge \tau \wedge \tau_c} \id_{\{ \X_{t} \in \bigcup K_{t} \}} \, dt,
\]
where $\X_{t} \in \bigcup K_{t}$ means that $\X_t$ is at a vertex belonging to a row or column having nonempty intersection with $\badSet_{t}$. As $\badSet_{t}$ is nondecreasing in $t$, so is $K_t$, thus we can further estimate
\[
\Lambda \leq \frac{ \Theta \beta}{n-1} n^{7 \varepsilon} \int\limits_{\eta}^{T \wedge \tau \wedge \tau_c} \id_{\{ \X_{t} \in \bigcup K_{T \wedge \tau} \}} \, dt = \frac{\Theta \beta}{n-1} n^{7 \varepsilon} \sum_{F\in  K_{T \wedge \tau}}\int\limits_{\eta}^{T \wedge \tau \wedge \tau_c} \id_{\{ \X_{t} \in F \}} \, dt.
\]
For any row or column $F\in K_{T \wedge \tau}$ we have
\begin{equation*}
	\int\limits_{\eta}^{T \wedge \tau \wedge \tau_c} \id_{\{ \X_{t} \in F \}} \, dt = \sum\limits_{v \in F} \int\limits_{\eta}^{T \wedge \tau \wedge \tau_c}\id_{\{\X_{t} = v\}}\, dt \leq |F\cap \Zz_{T\wedge \tau}| \leq \iota(\Zz_{T\wedge \tau}).
\end{equation*}
We used the fact that for a fixed vertex $v$ the integral is bounded by $1$, since $t\leq \tau_c$ and the bar corresponding to $v$ has height $1$. Combining the above facts we obtain
\begin{equation*}
	\Lambda \leq \frac{ \Theta \beta}{n-1} n^{7 \varepsilon}  |K_{T \wedge \tau}|\iota(\Zz_{T\wedge \tau}).
\end{equation*}
We have $|\badSet_{ T\wedge \tau}|\leq n^{7\varepsilon}$, thus $|K_{T \wedge \tau}| \leq 2 n^{7\varepsilon}$. Moreover, $\iota(\Zz_{T\wedge \tau})\leq n^{\varepsilon}$ and thus $\Lambda \leq 4 \Theta \beta n^{15\varepsilon -1} $.

As $I^{b}_{t} - \int\limits_{0}^{t} \lambda^{b}_{u} \, du$ is a martingale, we have by Doob's theorem
\[
\E \left( (I^{b}_{T \wedge \tau} - I^{b}_{\eta})\ind{T\wedge \tau \ge \eta}  | \Ff_{\eta} \right) = \E \left( \ind{T\wedge \tau \ge \eta} \int\limits_{\eta}^{T \wedge \tau} \lambda^{b}_{t} \, dt | \Ff_{\eta} \right) \leq 4 \Theta \beta n^{15 \varepsilon - 1}.
\]
Now by an application of the conditional Markov inequality we bound the first term of \eqref{eq:bad_jumps_increase} by $n^{-\delta}$ for some $\delta>0$. This concludes the proof.
\end{proof}

\end{subsection}
\begin{subsection}{Backtracks}\label{sec:backtracks}

We will now show that after an internal jump the CRW is unlikely to backtrack its steps back to $L_{0}$.

We start with a deterministic lemma about the structure of the core $\core_{t}$. Recall that a path in a graph is a sequence of pairwise distinct vertices $v_1,\ldots,v_k$ such that for all $i < k$ vertices $v_i$ and $v_{i+1}$ are adjacent.

\begin{lemma}\label{lm:long-path}
Suppose that $\core_{t}$ contains at most $n^{\delta}$ vertices of degree greater than $2$, for some $ \delta \in (0, \varepsilon)$. Then there exists $n_0 \geq 1$ such that for any $n \geq n_0$ and any $v \in L_{0} \cap \core_{t}$ the following holds: every simple path in $G_{t}$ which connects $v$ with $\badSet^c_{t}\cap \core_{t}$ must contain a subpath of length at least $\log^2 n$ consisting only of vertices of degree two in $\core_{t}$.
\end{lemma}
\begin{proof}
By the definition of the bad set any path connecting $v$ with $\badSet^c_{t}$ has length $n^{\varepsilon}$. Note that if the end vertex of the path is in $\core_{t}$ then the whole path is fully contained in $\core_{t}$ (otherwise to return to $\core_{t}$ the path would have to repeat one of its vertices). Dividing the path into consecutive subpaths of length $\log^2 n$ we obtain $n^{\varepsilon} / \log^2 n$ subpaths. Since $n^{\delta} < n^{\varepsilon} / \log^2 n$ for $n$ large enough and we assumed there are at most $n^{\delta}$ vertices of degree more than $2$ in $\core_{t}$, there must be at least one subpath with only vertices of degree $2$.
\end{proof}

Let $P$ be a path in $\core_{t-}$ consisting only of vertices of degree two in $\core_{t-}$. We will call such a path \emph{straight}. If $t < \tau_c$, we will call the \emph{potential of $P$} the total Lebesgue measure of these parts of bars corresponding to vertices of $P$ which have not been explored up to time $t$. For $t \ge \tau_c$ we set the potential to be zero. Formally, if we denote the potential of $P$ by $\PP(P)$, we have
\begin{displaymath}
  \PP(P) := \Ind{\{t< \tau_c\}}{\rm Leb}\Big( (V(P) \times [0,1))\setminus \crw_{[0,t]}\Big),
\end{displaymath}
where $V(P)$ is the set of vertices of $P$. We stress that this notion as well as the property of being a straight path depends strongly on $t$, which is suppressed in the notation but should not lead to misunderstanding in the sequel.

Suppose that $t < \tau_{c}$. Let $t_1(P)$ be the time the CRW entered a vertex of $P$ for the first time, and $t_2(P)\le t$ the time when it left a vertex of $P$ for the last time before $t$. Because of the property \eqref{eq:potential} at time $t$ the potential of $P$ equals to $\PP_{t_{2}(P)-} - \PP_{t_{1}(P)}$.

Given $b > 0$, we will say that $P$ has \emph{large potential with constant} $b$ if
\begin{equation}\label{eq:large-potential}
\PP(P)\geq b |P|,
\end{equation}
where $|P|$ is the length of $P$, i.e., the number of its vertices minus one. Later on the parameter $b$ will be fixed and we will simply use the term ``large potential'', with $b$ being implicit.

We now introduce the event
\begin{equation}\label{eq:straight-paths}
\mathcal{Q}_{t}(b) := \left\{ \mbox{all straight paths in $\core_{t}$ of length $\lceil \log^2 n\rceil$ have large potential with constant $b$} \right\}.
\end{equation}

\begin{lemma}[Straight paths have large potential]\label{lm:all-paths-are-good}
There exist $C, c > 0$ and $b > 0$ such that
\[
\Pp\left(\forall_{t < T\wedge \tau_c } \mathcal{Q}_{t}(b) \right) \geq 1 - C e^{- c \log^2 n}.
\]\end{lemma}

\begin{proof}
Let $\gamma = 1 / 4 \Theta \beta_{1}$, fix arbitrary $a < \Theta^{-1} \beta_{0} - 1/2$ and let $b = a\gamma$. Suppose that there exists $t < T\wedge \tau_c$ and a straight path $P$ in $\core_{t}$ of length $m = \lceil \log^2 n\rceil$ which does not have large potential. In particular this implies that there exist $l \geq 1$ and $k \geq m$ such that
\begin{equation}\label{eq:potential-b}
\mathcal{P}_{T_{l+k}-} - \mathcal{P}_{T_{l}} < b m.
\end{equation}
Indeed, we can take $t_1(P) = T_l$ and let $T_{l+k} \leq t$ be the moment the CRW enters the other end vertex of $P$.
Note that $k \geq m$ since the CRW may have traversed some dead trees between  $T_l$ and $T_{l+k}$.

The lemma will be shown once we prove that
\begin{equation} \label{eq:goal_17}
	p := \Pp\left( \exists_{l\geq 1} \exists_{k\geq m } \mathcal{P}_{T_{l+k}-} - \mathcal{P}_{T_{l}} < b m  \;\textrm{and}\; T_{l+k}\leq T\right) \leq C e^{- c \log^2 n},
\end{equation}
for some $C,c > 0$. Note that writing $T_{l+k}\leq T$ we implicitly assume that the CRW visits at least $l+k$ vertices. Now, if the event in the definition of $p$ holds, then either after some time $T_l$ we have discovered at least $k \geq m$ additional vertices very quickly (so that $T_{l+k} - T_{l} \leq \gamma m$ and $T_{l+k}\leq T$), or after $T_l$ the potential failed to increase by $bm$ in a time interval of length at least $\gamma m$. Thus we may estimate
\begin{multline*}
	p \leq  \Pp\left( \exists_{l\geq 1} \exists_{ u \in[ \gamma m, \, T - T_{l}] } \mathcal{P}_{T_{l}+u} - \mathcal{P}_{T_{l}} < b m \;\textrm{and} \; \tau_c \geq T_l +\gamma m\right) \\+ \Pp(\exists_{l\geq 1} \exists_{k\geq m } \, T_{l+k} - T_{l} \leq \gamma m \;\textrm{and} \;T_{l+k}\leq T ).
\end{multline*}

Note that we have chosen $\gamma$ small enough so that by Lemma \ref{lm:fresh-vertices-are-not-often} we get
\begin{displaymath}
\Pp(\exists_{l\geq 1} \exists_{k\geq m } \, T_{l+k} - T_{l} <  \gamma m) \le \sum_{l=1}^{n^2} \sum_{k \ge m} e^{-ck} \le C_1 e^{-c_1 m}.
\end{displaymath}

Thus
\begin{equation*}
	p \leq  \Pp\left(\exists_{l\geq 1} \exists_{u \in[ \gamma m, \, T - T_{l}] } \, \mathcal{P}_{T_{l}+u} - \mathcal{P}_{T_{l}} < b m \;\textrm{and}\; \tau_c \geq T_l +\gamma m \right) + C_1 e^{-c_1 m},
\end{equation*}
for some $C_1, c_1 > 0$. Now using a union bound over $l$ (observe that $T_l = \infty$ for $l > n^2$) and \eqref{eq:good_potential_eq3} from Lemma \ref{lm:good-potential} (recall that $b = a \gamma$) we obtain \eqref{eq:goal_17}, thus concluding the proof.
\end{proof}

From now on we fix $b$ to be the constant guaranteed by Lemma \ref{lm:all-paths-are-good}. Let $S_t$ be the set of endpoints of all straight paths in $\core_{t-}$ with potential at time $t$ at least $b\log^2 n$ and let $\rho_k$ denote the $k$-th moment $t$ when the CRW enters a vertex from $S_{t}$ by a backtrack. More precisely, set $\rho_0 = 0$, and for $k \ge 1$,
\begin{displaymath}
   \rho_k = \inf\{t > \rho_{k-1}\colon \textrm{$\X_t \in S_t, \X_{t-} \notin S_t$ and the bridge $(\{\X_t,X_{t-}\},t)$ has been traversed before time $t$}\}.
\end{displaymath}

Note that $\rho_k$ are stopping times and $\rho_k < \tau_c$ (since if $t> \tau_c$, then $S_t =\emptyset$). Moreover, the minimal path in $\core_{\rho_k-}$ with potential at time $\rho_k$ at least $b\log^2 n$, and starting at $\crw_{\rho_k}$, is uniquely determined. Let us denote this path by $P_k$.

Let us say that the CRW completely covers $P_k$ if after entering its end vertex by a backtrack it eventually exhausts all the bars corresponding to vertices of $P_k$, possibly departing from them at some intermediate time intervals. Denote by $\mathcal{A}_k$ the event that $\rho_k < \infty$, the CRW completely covers $P_k$ before time $T$ and while traversing $P_k$ it does not make an excursion of length $n^\varepsilon$ (recall Definition  \ref{def:excursion}). Formally,
\begin{displaymath}
\mathcal{A}_k = \bigcup_{t\in (0,T]} \Big( \{t > \rho_k\}\cap \{P_k\times [0,1) \subset \crw_{[0,t]}\}\cap \bigcap_{s \in (0,t]} (\mathcal{E}_{s}(n^\varepsilon)\cap \{s> \rho_{k} \})^c\Big).
\end{displaymath}
\begin{proposition}[No straight paths of large potential are covered without excursions]\label{prop:backtrack}
There exist $ C, c  > 0$ such that
\begin{displaymath}
  \p(\bigcup_{k=1}^\infty \mathcal{A}_k) \le Ce^{-c\log^2 n}.
\end{displaymath}
In other words, the probability that before time $T$ the CRW completely covers some straight path of potential (at the time of entry) at least $b\log^2 n$ without making an excursion of length $n^\varepsilon$ in the process, is bounded by $Ce^{-c\log^2 n}$.
\end{proposition}

\begin{proof}
Note first that for some $C > 0$ we have $\p(\rho_{Cn^2} < \infty) \le e^{-n^2}$. Indeed, each $\rho_k$ involves a backtrack and before the CRW closes into a cycle each bridge can be used in at most one backtrack, so the number of moments $\rho_k$ is bounded by the total numer of bridges in the process. It follows from Lemma \ref{le:auxilliary-poisson} (applied with $t = 0$, $s = 1$ and $k = Cn^2$ for suitably chosen $C > 0$) that this number is at most $Cn^2$, for some $C > 0$, with the required probability.

Thus, as we can perform a union bound over $k\leq Cn^2$, it is enough to show that for every $k \geq 1$ we have $\p(\mathcal{A}_k) \leq Ce^{-c\log^2 n}$ for some $C,c > 0$.

To simplify the notation, in what follows we will drop the subscript $k$ and write simply $\rho, P$ for $\rho_k, P_k$. We will also identify $P$ with $(V(P)\times [0,1))\setminus \crw_{[0,\rho]}$, i.e., with the part of all the bars corresponding to vertices from $P$ which at time $\rho$ was unused.

Let us first introduce a change of time to merge into one interval all the random intervals of time during which the CRW stays on $P$. Define
\begin{displaymath}
  H_t = \int_0^t \Ind{P}(\mathcal{X}_{\rho+s}) \, ds,
\end{displaymath}
where in the case $\rho = \infty$ we interpret $\Ind{P}(\mathcal{X}_{\rho+s})$ as zero (we will use this convention throughout the proof). The process $H_t$ measures how much of the potential that the path $P$ had at time $\rho$ has been used up to time $t$. In particular $P$ gets completely backtracked up to time $t$ if and only if $H_{t-\rho}$ equals the potential of the path at time $\rho$.

Define now $\sigma_0 = 0$ and for $s > 0$
\begin{displaymath}
  \sigma_s = \left(\inf\{t>0 \colon H_t \ge s\}\right)\wedge (T-\rho)_{+}.
\end{displaymath}

Consider also the processes
\begin{itemize}
\item $J_s = |\Zz_s\setminus L_0|$,
\item
$M_t = \int_{0}^{t} \Ind{P}(\mathcal{X}_{(\rho + s)-}) \, dJ_{\rho + s}$ -- the number of jumps to previously unexplored vertices outside $L_0$ that the CRW makes from $P$ between times $\rho$ and $\rho+t$.

\item $N_s = M_{\sigma_s}$ -- the number of jumps to previously unexplored vertices outside $L_0$ the CRW makes during the first $s$ time units spent on $P$ during the backtrack.
\end{itemize}

Note that $\sigma_s$ is a stopping time with respect to the filtration $(\mathcal{F}_{\rho+t})_{t\ge 0}$. Moreover, if $\sigma_s < T - \rho$, then $\mathcal{X}_{(\rho+\sigma_s)-} \in P$.

Our strategy for proving that $\p(\mathcal{A}_k)$ is small is as follows. First we will show that on the set $\mathcal{A}_k\cap\{ \tau^\alpha_{iso}> T\}$ we have $N_{b  \log^2 n} > c\log^2 n$, i.e., the CRW makes many jumps to previously unexplored vertices while backtracking $P$. Then, to finish the argument, we will use Lemma \ref{lm:excursions} to show that the probability that none of those jumps is a beginning of a forbidden excursion is small.

Let us thus first estimate the intensity of $N_t$ for $t \le b \log^2 n$ with respect to the filtration $(\mathcal{F}_{\rho+\sigma_s})_{s\ge 0}$ on the event $\{ \sigma_{b \log^2 n} <  T-\rho\} \cap \{\tau^\alpha_{iso}>T\}$ .

Denote by $\mu^J_t$ is the intensity of $J_t$ and note that by Doob's theorem and the properties of integrals with respect to counting processes, the intensity of $M_t$ with respect to $(\mathcal{F}_{\rho+t})_{t\ge 0}$ equals $\Ind{P}(\mathcal{X}_{(\rho+t)-})\mu^J_{\rho+t}$.

Now, again we employ Doob's theorem together with a change of variables (note that function $s\mapsto \sigma_s$ is constant on intervals where $\Ind{P}(\mathcal{X}_{(\rho+\sigma_s)-})$ vanishes and otherwise increases linearly with speed one) to conclude that the intensity of $N_s$ with respect to $(\mathcal{F}_{\rho+\sigma_s})_{s\ge 0}$ equals
\begin{displaymath}
  \widetilde{\mu}_{s} = \Ind{P}(\mathcal{X}_{(\rho+\sigma_s)-})\mu^J_{\rho+\sigma_s}.
\end{displaymath}
Therefore we can use the same argument as in the proof of \eqref{eq:intensity-mu} in Lemma \ref{cor:intensity} to conclude that if $T < \tau_{iso}^\alpha$, then for Lebesgue almost all $s$ such that $\rho + \sigma_s < T \wedge \tau_c$ we have 
\begin{displaymath}
\widetilde{\mu}_s \ge \frac{\Theta^{-1} \beta}{n-1}(n - 2 - n^\alpha) \ge c' > 0
\end{displaymath}
for some $c' > 0$.
Since $\mathcal{A}_k \subset \{ \rho + \sigma_{b \log^2 n } \leq T \wedge \tau_{c}\}$ we have
\begin{displaymath}
  \Lambda_{b \log^2 n } := \int_0^{b \log^2 n} \widetilde{\mu}_s \, ds \ge c' b\log^2 n
\end{displaymath}
on the event $\mathcal{A}_k\cap \{\tau^\alpha_{iso}>T\}$ and thus by Lemma \ref{lm:concentration-lower-bound-super} and Assumption \ref{lm:isoperimetry}, we get for some constants $C, c>0$
\begin{align}\label{eq:can't-have-few-jumps}
  \p(\mathcal{A}_k \cap \{N_{b \log^2 n } \le c\log^2n\}) \le Ce^{-c\log^2 n}.
\end{align}

It remains to bound the probability $\p(\mathcal{A}_k \cap \{N_{b \log^2 n} > c\log^2n\})$.

Denote thus by $\gamma_1,\gamma_2,\ldots$ the times of subsequent jumps from $P$ to previously unexplored vertices outside $L_0$ made after time $\rho$. The condition $N_{b \log^2 n} > c\log^2n$ translates into $\gamma_{\lceil c\log^2 n \rceil} \le \rho + \sigma_{b\log^2 n}$. Thus $\mathcal{A}_k \cap \{N_{b \log^2 n} > c\log^2n\} \subset \mathcal{A}_k \cap \{\gamma_{\lceil c\log^2 n \rceil} < T\}$.

Let
\[
	J := \inf \left\{r\ge 1\colon \Pp\left( \EE_{\gamma_{r}}(n^\varepsilon) | \mathcal{F}_{\gamma_r} \right) < q\cdot \Ind{\{\mathcal{X}_{\gamma_r} \notin L_0\}}\Ind{\cbr{\gamma_r \leq T}}\right\}.
\]
with $q$ as in Corollary \ref{cor:random-stopping-time}. $J$ is a stopping time with respect to the discrete time filtration $(\mathcal{F}_{\gamma_r})_{r\ge 1}$. By Corollary \ref{cor:random-stopping-time}, for each $r$ the probability that the inequality in the definition of $J$ holds is bounded by $Ce^{-c\log^2n}$. Taking the union bound over $r \le \lceil c\log^2 n\rceil$ we get $\p(J \le \lceil c\log^2 n\rceil) \le C'e^{-c'\log^2 n}$ for some $C', c' > 0$.

We can now estimate
\begin{align*}
	\p(\mathcal{A}_k \cap &\{N_{b \log^2 n} > c\log^2n\})  \le \p(\mathcal{A}_k\cap\{\gamma_{\lceil c\log^2 n\rceil} < T\}) \\
	&\le \p\left(\bigcap_{r=1}^{\lceil c\log^2n \rceil} \mathcal{E}_{\gamma_r}(n^\varepsilon)^c \cap \left\{\gamma_{\lceil c\log^2 n\rceil} \leq T\right\}\cap \left\{J > \lceil c
  \log^2 n\rceil\right\}\right) + C'e^{-c' \log^2 n}\\
&\le (1-q)^{c\log^2 n} + C'e^{-c' \log^2 n} \le 2 e^{-c''\log^2 n},
\end{align*}
where the third inequality is obtained by a sequence of conditionings with respect to $\mathcal{F}_{\gamma_r}$, $r = \lceil c\log^2 n\rceil,\ldots,1$. Together with \eqref{eq:can't-have-few-jumps} this shows that for all $k$,
\begin{displaymath}
  \p(\mathcal{A}_k) \le Ce^{-c\log^2 n},
\end{displaymath}
which ends the proof of the proposition.
\end{proof}

\end{subsection}

\begin{subsection}{Isoperimetry upper bound}\label{sec:proof-of-isoperimetry}
Let $\tau_{0} = 0$ and for $k \geq 1$ let
\begin{equation*}\label{eq:tau_k}
\tau_{k} := \inf\{t > \tau_{k-1} \colon \X_{t} \in L_{0}, \, \X_{t} \notin \Zz_{t-} \}
\end{equation*}
be the $k$-th time a new vertex from $L_{0}$ is visited by $\X$.

First we present the main technical lemma stating that the CRW does not visit $L_0$ too often. This result will be used to prove the forthcoming Corollary~\ref{cor:intersections-with-l0} claiming good isoperimetry.

\begin{lemma}\label{lm:main-lemma}
There exists $C, c, p>0 $ such that for any $k \in \N$
\[
\Pp\left(  \tau_{k + 1} - \tau_{k} >  n \wedge (T - \tau_{k})  | \Ff_{\tau_{k}} \right) \geq p \cdot \id_{\{\tau_k \leq T - 1\}}
\]
holds with probability at least $1 - \wsp$.
\end{lemma}

\begin{proof}
Throughout the proof we will use ``with high probability'' as a shorthand for ``with probability at least $1 - \wsp$'' for some constants $C,c > 0$ (whose values may change from line to line).

At time $\tau_{k}$ a new vertex $w \in L_{0}$ is visited (unless $k=0$ and $\crw_{0} \notin L_{0}$, which is easily dealt with below). Our first aim is to show that the CRW escapes from $L_0$ and then from the bad set (with conditional probability uniformly bounded away from $0$).
Define the stopping time $\sigma = \inf\{t > \tau_k\colon \mathcal{X}_t \notin \badSet_{t}\}$.
By Lemma \ref{lm:l-jumps} \whp~the CRW has a chance bounded away from $0$ of performing (in time less than $1$) a jump from $w$ to a previously unvisited vertex from the same column, thus leaving $L_{0}$. From the new vertex the CRW can, with probability bounded away from $0$, make an excursion described in Corollary \ref{cor:random-stopping-time}. Note that if this happens, then $\sigma < \tau_{k+1}$. Thus \whp~we have $\p(\sigma < \tau_{k+1}|\mathcal{F}_{\tau_k}) > q$ for some $q > 0$, independent of $n$.

Note that if $\tau_{k+1} > \sigma > T$ then, $\tau_{k+1} - \tau_k > \sigma - \tau_k > T - \tau_k$, so it is enough to prove that for some $q' > 0$  we have
\begin{equation}\label{eq:long-excursion-without-l0}
\Pp\left(\Zz_{[\sigma,(\sigma +  n)\wedge T]} \cap L_{0} = \emptyset | \Ff_{\sigma} \right) \geq q'\ind{\sigma \le T, \sigma < \tau_{k+1}},
\end{equation}
\whp. This also takes care of the case $k=0$, $\crw_{0} \notin L_{0}$, since then $\sigma = \tau_{0}$.

From now on we will implicitly work on the event $\{\sigma \le T, \sigma < \tau_{k+1}\}$. Denote by $\tau$ the time of the first hit of $L_0$ after time $\sigma$. This hit can happen either by a direct jump or by a backtrack. Denote the event that the former (resp. the latter) situation happens and $\tau_{k+1} \le (\sigma+n)\wedge T$ by $\RR$ (resp. by $\Kk$). Clearly on $\RR$ and $\Kk$ we also have $\tau \le (\sigma+n)\wedge T$. Moreover on $\Kk$ we have $\tau < \tau_{k+1} < \infty$, so in particular $\tau < \tau_c$.

Recall that $I_t^d$ is the number of  direct jumps to $L_0$ in the time interval $[0,t]$.
By Lemma \ref{le:intensity} the intensity $\lambda_t^d$ of $I_t^d$ is at most $\frac{\Theta \beta}{n-1}$ as long as the CRW is outside $L_{0}$, in particular on the interval $[\sigma,\tau)$ (the intensity can be $0$ if the only possible vertex in $L_{0}$ is dead).
Denoting $\Lambda_t = \int_0^t \lambda_t^d$, we thus get
\begin{displaymath}
\p(\RR|\mathcal{F}_\sigma) = \p\left(\left\{I_{\tau \wedge (\sigma+n)\wedge T}^d - I_\sigma^d \ge 1 \right\}\cap \left\{ \Lambda_{ \tau\wedge (\sigma+n)} - \Lambda_\sigma \le \frac{n}{n-1}\Theta \beta \right\} \Big|\mathcal{F}_\sigma\right)
\end{displaymath}
and by Lemma \ref{lm:concentration-upper-bound} the right hand side is bounded from above by $\p(X \ge 1)$, where $X$ is a Poisson random variable with parameter $\frac{n }{n-1} \Theta\beta \leq 2 \Theta \beta$. We conclude that almost surely $\Pp(\RR|\Ff_\sigma)$ is uniformly bounded away from $1$ or, equivalently, for some $q > 0$ we have
\begin{align}\label{eq:costamcostam}
\p(\RR^c|\mathcal{F}_\sigma) \ge q \Ind{\{\sigma \le T\wedge \tau_{k+1}\}} \geq q.
\end{align}

As $\Kk \subset \RR^c$, to obtain \eqref{eq:long-excursion-without-l0} it thus suffices to show that $\p(\Kk | \mathcal{F}_\sigma) = o(1)$ with high probability. We will first show that we can restrict to an event on which for all $t \le T$ the sets $\badSet_{t}$ have at most $n^{\varepsilon/2}$ vertices of degree greater than 2 in $\core_{t}$, there are no bad jumps for $t \leq \tau$, and all straight paths in $\core_{t}$ of length at least $\log^2 n$ have potential greater than $b\log^2 n$. To this end recall the notation of Lemmas \ref{lm:internal-hits-degree3}, \ref{lm:bad-hits-process} and \ref{lm:all-paths-are-good},
set $\delta = \varepsilon/2$ and define
\begin{displaymath}
 \mathcal{C} := \bigcap_{t \le T} \DD_t^\delta \cap \{I^{b}_{T} = 0\} \cap \bigcap_{t\le T}\mathcal{Q}_t (b),
\end{displaymath}
with $b$ as guaranteed by Lemma \ref{lm:all-paths-are-good}.

We see that $\mathcal{C}$ satisfies all the properties mentioned above and moreover by the aforesaid lemmas (and Lemma \ref{le:easy-lemma}), with probability at least $1-Ce^{-c\log^2 n}$ we have $\p(\mathcal{C}|\mathcal{F}_\sigma) \ge 1 - e^{-c\log^2 n}$. Thus it is enough to show that $\p(\Kk \cap \mathcal{C}|\mathcal{F}_\sigma) = o(1)$ with high probability.

By definition we have $\crw_\sigma \notin \badSet_{\sigma}$. Let
\begin{displaymath}
t = \max\{s \in  [\sigma,\tau) \colon \X_s \neq \X_{s-}, \textrm{ and } \X_s \notin \badSet_{s}\}
\end{displaymath}
i.e., $t$ is the time of the last jump of the CRW before $\tau$ such that $\crw_{t} \notin \badSet_{t}$ (note that $t < T$). On $\Kk \cap \mathcal{C}$ the assumptions of Lemma \ref{lm:long-path} are satisfied, so every simple path in $G_t$ from $\crw_{t}$ to $L_0$ must contain a straight subpath in $\core_t$ of length at least $\log^2 n$. By the definition of the event $\mathcal{C}$ at time $t$ every such path has potential at least $b\log^2 n$. Since on $\Kk \cap \mathcal{C} \subset \RR^c \cap \mathcal{C}$ there are no bad jumps or direct hits to $L_0$, to get from $\crw_t$ to $L_0$ the CRW must completely cover at least one such path (we give a formal proof of this intuitively clear fact below).

Moreover, between $t$ and $\tau$ the CRW does not make an excursion of length $n^{\varepsilon}$, since at the end of such an excursion it would be outside the current bad set, which would contradict the definition of $t$. Therefore, by Proposition \ref{prop:backtrack},  $\p(\Kk \cap \mathcal{C}) \le Ce^{-c\log ^2 n}$, which by Lemma \ref{le:easy-lemma} shows that $\p(\Kk \cap \mathcal{C}|\mathcal{F}_\sigma) \le e^{-\frac{c}{2}\log ^2 n}$ with probability at least $1 - Ce^{-\frac{c}{2}\log^2 n}$. We have thus proved that $\p(\Kk|\mathcal{F}_\sigma) = o(1)$ with probability at least $1 - Ce^{-c\log^2 n}$ for some $C,c > 0$, which together with \eqref{eq:costamcostam} proves \eqref{eq:long-excursion-without-l0}.

We finish with a formal proof of the existence of a straight path with large potential which is covered by the CRW between time $t$ and $\tau$ (assuming that the event $\mathcal{C}$ holds). First, it is easy to see that the next vertex visited by the CRW after time $t$ must belong to $\badSet_{t}$. Indeed, assume that the next jump happened at time $s$. Then by the definition of $t$, $\crw_{s} \in \badSet_{s}$. Moreover the shortest path in $G_s$ from $\crw_s$ to $L_0$ avoids $\crw_t$ (otherwise we would have $\crw_t \in \badSet_t$). Thus this path uses only edges from $G_t$, and so $\crw_s \in \badSet_t$.  

Now, between $t$ and $\tau$ there are finitely many jumps. Denote by $v_1,\ldots,v_M$ consecutive vertices from $\badSet_t$, visited by the CRW between times $t$ and $\tau$. Let $t_1,\ldots,t_M$ be the times of visits to $v_1,\ldots,v_M$. Note that in particular $t_M = \tau$, since by the definition of $\tau$ and the event $\Kk$, $\X_\tau \in G_\sigma \subset G_t$.

We will show by induction that each $v_k$ is connected to $\badSet_t^c \cap \core_t$ by a path in $G_t$ consisting only of vertices which have been visited by the CRW between times $t$ and $t_k$.

This is true for $v_1 = \crw_s$ because it belongs to $\badSet_t = \badSet_{s-}$ and so (due to absence of bad jumps) the jump from $\crw_t$ to $\crw_s$ is necessarily a backtrack.  Assuming that the statement in question holds for $v_1,\ldots,v_k$, we have three possibilities:
\begin{itemize}
\item The vertex $v_{k+1}$ is visited directly after $v_k$, in which case the statement extends to $v_{k+1}$ since due to absence of bad jumps between times $t$ and $t_{k+1}$ the bridge used for this jump corresponds to an edge in $G_t$.

\item The CRW leaves $\badSet_t$ after visiting $v_k$ and the re-entry into $\badSet_t$ happens by a bridge corresponding to an edge in $G_t$.  In this case the statement in question also holds for $v_{k+1}$, since then $\X_{t_{k+1}-} \in \core_t$).

\item The CRW leaves $\badSet_t$ after visiting $v_k$ and the re-entry into $\badSet_t$ happens via a bridge which up to time $t$ was unexplored. Since $v_{k+1} \in \badSet_{t}$ and after time $t$ there were no bad jumps, in this situation, the vertex $v_{k+1}$ must equal $v_l$ for some $l\le k$, since the bridge used for re-entry must have been used for the first time only after time $t$ and it could have been used only to leave $\badSet_t$.
\end{itemize}

Thus in particular $\crw_{\tau}$ is connected with $\badSet_{t}^c\cap \core_t$ by a path $P$ in $G_t$ consisting only of vertices visited by the CRW between times $t$ and $\tau$. Note that since $\tau<\tau_c$, we must have $\crw_\tau \in \core_{t}$. By Lemma \ref{lm:long-path} the path $P$ contains a straight subpath $P'$ of length at least $\log^2 n$. It is now easy to see that this subpath is completely covered between $t$ and $\tau$. Indeed, using again the absence of bad jumps and the fact that on $\Kk$ we have $\tau< \tau_c$, we see that the order in which the vertices of $P'$ are visited by the CRW between times $t$ and $\tau$ is uniquely determined -- the first entry into each consecutive vertex of $P'$ must be made by a backtrack from the previous one and after each departure from $P'$ the CRW returns by backtrack using the same bridge through which it has left. This shows that the whole path $P'$ must be exhausted before time $\tau$. Denote by $t'$ the first time after $t$ when the CRW enters a vertex of $P'$. Then clearly $P'$ is a straight path in $\core_{t'-}$ and therefore by the definition of the event $\mathcal{C}$ it has large potential at time $t'$. In particular $t' = \rho_k$ for some $k$ (recall the definition of $\rho_k$ given before Proposition \ref{prop:backtrack}).
\end{proof}

\begin{corollary}\label{cor:intersections-with-l0}
There exist $C, c >0$ such that
\[
\Pp \left( |\Zz_{T} \cap L_{0}| \geq C \log^2 n  \right) \leq C e^{- c \log^2 n}.
\]
\end{corollary}

\begin{proof}
Let $\mathcal{H}_{k} = \Ff_{\tau_{k}}$ and let $A_{k}$ denote the event that the estimate from Lemma \ref{lm:main-lemma} holds. For $k \geq 0$ consider the events
\[
\EE_{k} = \big( \{\tau_{k + 1} - \tau_{k} >  n \wedge (T - \tau_{k})\} \cap A_{k} \cap\{\tau_k \le T-1\}\big) \cup A_{k}^{c} \cup \{\tau_k > T-1 \}.
\]
We have
\[
\Pp( \EE_{k} | \mathcal{H}_{k}) = \id_{A_{k}\cap\{\tau_k \le T-1\}} \Pp( \{\tau_{k + 1} - \tau_{k} >  n \wedge (T - \tau_{k})\} | \mathcal{H}_{k} ) + \id_{A_{k}^{c} \cup \{\tau_k > T-1 \}},
\]
and by applying Lemma \ref{lm:main-lemma} we can estimate
\begin{equation}\label{eq:lower-bound-bernoulli}
\Pp( \EE_{k} | \mathcal{H}_{k}) \geq p \id_{A_{k}} \id_{\{\tau_k \leq T-1 \}} + \id_{A_{k}^{c} \cup \{\tau_k > T-1 \}} \geq p.
\end{equation}
Now let $K = \frac{2}{p} \lceil \log^2 n \rceil$ and consider the sum $\id_{\EE_{1}} + \ldots + \id_{\EE_{K}}$. Introducing $\xi_{k} := \id_{\EE_{k}} - p$, we observe that by \eqref{eq:lower-bound-bernoulli} the sum $M_N = \xi_1 + \ldots + \xi_{N}$ (with $M_0 = 0$) forms a submartingale with increments bounded by $1$. As any submartingale can be written as a martingale plus a nonnegative predictable term and here the martingale part has bounded increments, by Azuma's inequality for martingales with bounded increments (\cite{mcdiarmid_1989}) we get
\[
\pr{ M_N \leq - t } \leq e^{ - \frac{t^2}{2N}}
\]
for any $t \geq 0$ and $N \geq 1$. Taking $N = K$, $t = \frac{pK}{2}$ and rewriting the inequality in terms of $\id_{\EE_{k}}$ we obtain the estimate
\[
\pr{\id_{\EE_{1}} + \ldots + \id_{\EE_{K}} \leq \frac{pK}{2}} \leq e^{-\frac{p^2 K}{8}}.
\]
Therefore with high probability at least $pK / 2$ of the events $\EE_{k}$ hold. Since $\pr{A_{k}} \geq 1 - C e^{- c \log^2 n}$ for some $C,c > 0$, by doing a union bound over $k$ we can assume that none of the events $A_{k}^{c}$ hold. This implies that either $\tau_{k} > T-1$ for some $k = 1, \ldots, K$, or the event $\{\tau_{k+1} - \tau_{k} > n \}$ holds at least $pK / 2$ times, which implies $\tau_{K} \geq n p K / 2 \wedge (T - 1)$. As $T \leq n \log^2 n$, in both cases we have $\tau_{K} \geq T-1$, so with probability at least $1 - e^{-c \log^2 n}$ (for some $c > 0$) we have at most $K = \lceil \frac{2}{p} \log^2 n \rceil$ vertices from $L_{0}$ visited up to time $T-1$. An easy estimate shows that with high enough probability there are at most $c \log^2 n$ visits to $L_{0}$ between times $T-1$ and $T$, which ends the proof.
\end{proof}

\paragraph{Proofs of Lemma \ref{le:bootstrap-isoperimetry} and Proposition \ref{prop:isoperimetry-log2n}}\label{sec:conclusions}

Using the results of Section \ref{sec:proof-of-isoperimetry} the proof of Lemma \ref{le:bootstrap-isoperimetry} is now immediate.

\begin{proof}[Proof of Lemma \ref{le:bootstrap-isoperimetry}]

Since in Corollary \ref{cor:intersections-with-l0} the starting vertex of the CRW was arbitrary, by symmetry of the graph $H_{n}$ the claim of the corollary holds with $L_{0}$ replaced by any other row or column. Thus by performing a union bound over all rows and columns, and all starting vertices $v \in V$, we obtain that there exist constants $C, C', c' > 0$ such that
\[
\Pp \left( \forall v\in V \, \iota(\Zz_{T} (v) ) \leq C \log^2 n  \right) \geq 1 - C' e^{- c' \log^2 n}.
\]
Since we worked under Assumption \ref{lm:isoperimetry}, this finishes the proof of Lemma \ref{le:bootstrap-isoperimetry}.

\end{proof}

With this lemma we can finally prove Proposition \ref{prop:isoperimetry-log2n}.

\begin{proof}[Proof of Proposition \ref{prop:isoperimetry-log2n}]

Fix $\alpha \in (0, 1/100)$ and $T = n^{\alpha / 2}$. Consider the cyclic random walk $\crw = \crw(v)$ started at a vertex $v$.
By Lemma \ref{lm:fresh-vertices-are-not-often} we have
\[
\Pp \left( T_{\lceil 4 \Theta \beta n^{\alpha / 2} \rceil} < \lfloor n^{\alpha / 2} \rfloor \right) \leq C e^{-c n^{\alpha / 2}}
\]
for some $C , c > 0$. Since $\lceil 4 \Theta \beta n^{\alpha / 2} \rceil < n^{\alpha}$ and $e^{-c n^{\alpha / 2}} \leq e^{- c \log^2 n} $ for $n$ large enough, we obtain that with probability at least $1 - C_{1} e^{ - c_{1} \log^2 n}$, for some $C_{1}, c_{1} > 0$, until time $T$ fewer than $n^{\alpha}$ vertices have been explored by the CRW. In particular, this implies $\iota(\Zz_{T}) < n^{\alpha}$.

Now by a union bound over starting vertices $v$ and Lemma \ref{le:bootstrap-isoperimetry} we obtain that for some $C, C_{2}, c_{2} > 0$ we have
\[
\Pp \left( \forall v \in V \, \iota(\Zz_{T}(v) ) \leq C \log^2 n  \right) \geq 1 - C_2 e^{- c_2 \log^2 n}.
\]
The rest of the argument is inductive. Suppose that for some $T \in [n^{\alpha / 2}, n^{1 - \alpha / 2} \log^2 n ]$ we have
\begin{equation}\label{eq:bootstrap-induction}
\Pp \left( \forall v \in V \, \iota(\Zz_{T}(v) ) \leq C \log^2 n \right) \geq 1 - C' e^{- c' \log^2 n}
\end{equation}
for some $C', c' > 0$. Consider the CRW started at a fixed vertex $w$ and run up to time $T' = \lfloor T \rfloor \lfloor n^{\alpha / 2} \rfloor$. Divide the time interval $[0, T']$ into $k = \lfloor n^{\alpha / 2} \rfloor$ intervals $I_{i} = [S_{i}, S_{i+1})$ of length $S = \lfloor T \rfloor$. Observe now that for any $v \in V$ by construction of the cyclic random walk we have $\Zz_{[S_{i}, S_{i+1}]}(v) = \Zz_{[0, S]}(\crw_{S_{i}}(v))$. Since $S\leq T$ and the bound in \eqref{eq:bootstrap-induction} is uniform over all vertices,  we obtain
\[
\Pp \left( \forall_{i = 1, \ldots, k} \, \forall v\in V \, \iota(\Zz_{[S_{i}, S_{i+1}] }(v) ) \leq C \log^2 n  \right) \geq 1 - C' e^{- c' \log^2 n}.
\]
for some $C', c' > 0$.
Finally, by subadditivity of $\iota$ we can bound $\iota(\Zz_{T'}(v))$ by the sum of $\iota(\Zz_{[S_{i}, S_{i+1}]}(v))$ for $i = 1, \ldots, k$, obtaining
\[
\Pp \left( \forall v \in V \, \iota(\Zz_{T'}(v) ) \leq C  \lfloor n^{\alpha / 2} \rfloor \log^2 n  \right) \geq 1 -  C' e^{- c' \log^2 n}.
\]
As $C  \lfloor n^{\alpha / 2} \rfloor \log^2 n < n^{\alpha}$ for $n$ large enough, we obtain that for some constants $C_{1}, c_{1} > 0$
\[
\Pp \left( \forall v \in V \, \iota(\Zz_{T'}(v) ) \leq  n^{\alpha} \right) \geq 1 - C_{1} e^{- c_{1} \log^2 n}.
\]
Now an application of Lemma \ref{le:bootstrap-isoperimetry} gives for some constants $C, C_{2}, c_{2} > 0$
\[
\Pp \left( \forall v \in V \, \iota(\Zz_{T' }(v) ) \leq C \log^2 n \right) \geq 1 - C_{2} e^{- c_{2} \log^2 n},
\]
which finishes the inductive step.

Now, since we started from $T = n^{\alpha / 2}$ and at each step we increase the time by a factor of $\lfloor n^{\alpha / 2} \rfloor$, after at most $\lceil \frac{2}{\alpha} \rceil + 1$ steps we obtain
\[
\Pp \left( \forall v \in V \, \iota(\Zz_{n \log^2 n }(v) ) \leq  C \log^2 n \right) \geq 1 - C e^{- c \log^2 n}
\]
for some constants $C, c > 0$ depending on $\alpha$, but not on $n$. This finishes the proof.
\end{proof}

\end{subsection}

\begin{subsection}{Isoperimetry lower bound}\label{sec:isoperimetry-lower-bound}

In this section we prove the isoperimetry lower bound given by Proposition \ref{prop:isoperimetry-log2n-lower-bound}. The proof is independent of the previous section (we will only make use of Lemma \ref{le:intensity} and Lemma \ref{lm:fresh-vertices-are-not-often}, which do not require Assumption \ref{lm:isoperimetry}).

\begin{proposition}\label{prop:lower-bound}

Fix $\theta > 0$, an admissible function $\mathcal{C}$ and $\beta_{0}, \beta_{1} > 0$. Consider $\beta \in [\beta_{0}, \beta_{1}]$ and let $\crw(v) := \crw^{\beta, \theta, \mathcal{C}}(v)$ be the cyclic random walk associated to $\mu_{\beta, \theta, \mathcal{C}}$, started at $v$. There exist $C, c >0$ (depending only on $\beta_{0}$, $\beta_{1}$, $\theta$) such that
\[
\Pp \left(\{ |Z_{n \log^2 n}(v) \cap L_{0}| \leq c \log^2 n \} \cap \{ T_{n \log^2 n}(v) < \infty \} \right) \leq C e^{- c \log^2 n}.
\]
\end{proposition}

\begin{proof}

Let $\kappa < 1 / 2 \Theta \beta_{1}$ and $T \ge \kappa n \log^2 n$. Let $Y_{t} = | \Zz_{t} \cap L_{0}|$ and let $\lambda_{t}$ be the corresponding intensity of the process $Y_{t}$. Fix $\delta < \beta_{0} \Theta^{-1} \kappa/2$. We will first show that
\begin{equation}\label{eq:lower-bound-for-t}
\pr{ \{ Y_{T} \leq \delta\log^2 n \} \cap \{ \tau_{c} \geq T \} } \leq C e^{- c \log^2 n}
\end{equation}
for some $C,c > 0$.

Consider $\Lambda_{t} = \int\limits_{0}^{t} \lambda_{s} \, ds$.
Observe that
\begin{align}\label{eq:inclusion}
  \{Y_T \le \frac{\kappa}{2} \log^2 n\} \cap \{\tau_c \ge T\} \subset \{\Lambda_T \ge \frac{\beta_{0} \Theta^{-1} \kappa}{2}\log^2 n\}.
\end{align}
Indeed, up to time $\tau_c$, unless the CRW is at a vertex from a column containing an already visited vertex from $L_{0}$, the intensity of making a direct jump to $L_{0}$ (and thus necessarily discovering a previously unvisited vertex from $L_{0}$) is bounded from below by $\frac{\beta \Theta^{-1}}{n-1} >  \frac{\beta_{0} \Theta^{-1}}{n}$ by Lemma \ref{le:intensity}. Note that there are at most $nY_T$ such bad vertices and until time $\tau_{c}$ the CRW can spend time at most $1$ at any given vertex. Thus on the event $\{Y_T \le \frac{\kappa}{2} \log^2 n\} \cap\{\tau_c \geq T\}$ the total time spent at bad vertices before $T$ is at most $n Y_T \le \frac{\kappa}{2}n\log^2n$. We thus have
\[
\Lambda_{T} = \int\limits_{0}^{T} \lambda_{t} \, dt \geq \frac{\beta_{0} \Theta^{-1}}{n} \left(T - \frac{\kappa}{2}n \log^2 n \right) \ge \frac{\beta_{0} \Theta^{-1} \kappa}{2} \log^2 n,
\]
proving \eqref{eq:inclusion}.
As $\delta < \beta_{0} \Theta^{-1} \kappa/2$ we thus get
\begin{align*}
\pr{\{Y_{T} \leq \delta\log^2 n\}\cap\{\tau_c \ge T\}} \leq & \pr{ \{ Y_{T} \leq \delta \log^2 n \} \cap \{ \Lambda_{T} \geq \frac{\beta_{0} \Theta^{-1}\kappa}{2} \log^2 n \} }
\end{align*}
and by Lemma \ref{lm:concentration-lower-bound-super} the right-hand side above is bounded by $C e^{- c \log^2 n}$ for some $C,c > 0$. This proves \eqref{eq:lower-bound-for-t}.

Now we prove the statement of the proposition. Let $M = n \log^2 n$ and $A = \{ |Z_{M} \cap L_{0}| \leq \delta \log^2 n \} \cap \{ T_{M} < \infty \}$, with the same $\delta$ as above. We have
\begin{align*}
& \Pp \left( A \right) \leq \Pp \left( A \cap \{ T_{M} \leq \kappa n \log^2 n \} \right) + \Pp \left( A \cap \{ T_{M} > \kappa n \log^2 n \} \right) \\
& \leq  \Pp \left( T_{M} \leq \kappa n \log^2 n \right) + \Pp \left( A \cap \{ T_{M} > \kappa n \log^2 n \} \right).
\end{align*}
As $1 / \kappa > 2 \Theta \beta_{1} \geq 2 \Theta \beta $, by Lemma \ref{lm:fresh-vertices-are-not-often} the first term on the right hand side does not exceed $e^{-c'\log^2 n}$ for some $c'>0$. For the second term we write
\begin{align*}
& \Pp \left( A \cap \{ T_{M} > \kappa n \log^2 n \} \right) = \Pp \left( \{ |Z_{M} \cap L_{0}| \leq \delta \log^2 n \} \cap \{ T_{M} < \infty \} \cap \{ T_{M} > \kappa n \log^2 n \} \right) \\
& \le \Pp \left( \{ |\Zz_{\kappa n \log^2 n} \cap L_{0}| \leq \delta \log^2 n \} \cap \{ \tau_{c} \geq \kappa n \log^2 n \} \right)
\end{align*}
and by \eqref{eq:lower-bound-for-t} the right hand side is small enough, which finishes the proof.

\end{proof}

The proof of Proposition \ref{prop:isoperimetry-log2n-lower-bound} is now rather straightforward.
\begin{proof}[Proof of Proposition \ref{prop:isoperimetry-log2n-lower-bound}]

Fix a starting vertex $v$. Let $\tau_{0} = 0$ and let
\[
\tau_{k} = \inf\{t > \tau_{k-1} \colon \X_{t} \in L_{0}, \, \X_{t} \notin \Zz_{t-} \}
\]
be the $k$-th time a new vertex from $L_{0}$ is visited. Let $J_{t}$ be the total number of fresh jumps made up to time $t$. By Lemma \ref{le:intensity} the intensity $\lambda$ of $J$ is bounded from above by $2 \beta \Theta$. By Lemma \ref{lm:concentration-upper-bound} we obtain
\[
\p(J_{\tau_{k} + 1} - J_{\tau_{k}} \ge 1 | \mathcal{F}_{\tau_{k}}) \id_{\{\tau_k < \infty\}} \le \p(X \ge 1),
\]
where $X$ is a Poisson variable with parameter $2 \beta \Theta$. In particular this implies that for some $p>0$ bounded away from $0$ the following holds: each time the CRW visits a new vertex $\crw_{\tau_{k}} = w$ from $L_{0}$, with probability at least $p$ it makes no fresh jumps for time $1$, thus exhausting the whole bar of $w$. In particular with probability $p>0$ the vertex $w$ enters the orbit $\mathcal{O}(v)$.

The rest of the proof is a rather standard concentration estimate. Let $K = \lceil c \log^2 n \rceil$, with $c$ as in Proposition \ref{prop:lower-bound}. Let $\mathcal{H}_{k} = \mathcal{F}_{\tau_{k}}$ and
\[
A_{k} = \left( \{ \tau_{k} < \infty\} \cap \{ \crw_{\tau_{k}} \in \mathcal{O}(v) \} \right) \cup \{\tau_{k} = \infty\}.
\]
By the argument above we have $\p(A_{k} | \mathcal{H}_{k}) \geq p$, which implies that for $\xi_k := \id_{A_k} - p$ the sum $M_N = \xi_1 + \ldots + \xi_N$ (with $M_0 = 0$) forms a submartingale with increments bounded by $1$. As any submartingale can be written as a martingale plus a nonnegative predictable term and here the martingale part has bounded increments, by Azuma's inequality for martingales with bounded increments (\cite{mcdiarmid_1989}) we get
\[
\pr{ M_N \leq - t } \leq e^{ - \frac{t^2}{2N}}
\]
for any $t \geq 0$ and $N \geq 1$. Taking $N = K$ and $t = \frac{pK}{2}$ we obtain the estimate
\[
\pr{\id_{A_{1}} + \ldots + \id_{A_{K}} \leq \frac{pK}{2}} \leq e^{-\frac{p^2 K}{8}}.
\]
Together with Proposition \ref{prop:lower-bound} this implies that with probability at least $1 - C' e^{-c' \log^2 n}$ (for some $C', c' > 0$) we have either
\begin{itemize}
\item $T_{n \log^2 n}(v) = \infty$, which implies $|\mathcal{O}(v)| < n \log^2 n$,
\item or $\tau_{K} \leq T_{n\log^2 n} < \infty$ and $\sum\limits_{k=1}^{K} \id_{A_{k}} \geq \frac{pK}{2}$.
\end{itemize}
In the latter case we get that the event $\{ \tau_{k} < \infty\} \cap \{ \crw_{\tau_{k}} \in \mathcal{O}(v) \}$ holds at least $\frac{cp}{2} \log^2 n$ times for $k \leq K$, implying in particular $|\mathcal{O}_{n \log^2 n}(v) \cap L_{0}| \geq c'' \log^2 n$ for some $c'' > 0$ (note that for $ k\ge 1$ if $T_k < \infty$, then $T_k \le k-1$).

To finish the proof we note that by symmetry of $H_{n}$ the above argument is valid with $L_{0}$ replaced by any other row or column (as Proposition \ref{prop:lower-bound} has the same symmetry) and the starting vertex $v$ was arbitrary. Thus by performing a union bound over starting vertices $v \in V$ and all rows and columns we obtain the desired bound on $\chi(\mathcal{O}_{n \log^2 n}(v))$.

\end{proof}

\end{subsection}

\end{section}

\begin{section}{General transposition processes}\label{sec:transposition_process}

In this section we introduce the notion of a general transposition process, which is another perspective on the permutation model $\mu_{\beta, \theta, \mathcal{C}}$ defined in \eqref{eq:permutation_associated} and \eqref{eq:measure}  In this formalism we will state Lemma \ref{lem:almost_uniform} and Proposition \ref{cor:splitting}. These are all prerequisites needed in the next, final section.

For $X\in \mathfrak{X}$ let $(e_1, t_1), (e_2, t_2), \ldots, (e_{|X|}, t_{|X|})$ be the points of $X$ sorted by the second coordinate. We define $\{\sigma_t\}_{t\in \{0, 1, \ldots, |X|\}}$ by $\sigma_0 := \text{id}$ and for $t \in \{1, \ldots, |X|\}$
\[
\sigma_t := {e_t} \circ \ldots \circ {e_1},
\]
where any edge is identified with the transposition of its endpoints. This becomes a stochastic process when $X$ is sampled according to $\mu_{\beta, \theta, \mathcal{C}}$. Note that $\sigma_{|X|} = \sigma(X)$, with the latter permutation defined in \eqref{eq:permutation_associated}.

In what follows we will consider $(\sigma_i)_{i=0}^k$ conditionally on $\Omega_k := \{|X| = k\}$.
We use $(\fil_i)_{i=0}^k$ to denote the filtration on $\Omega_k$ associated with the process.

Obviously this transposition model depends on the parameters of $\mu_{\beta, \theta, \mathcal{C}}$. What might be surprising is that its evolution is not far from the i.i.d. transposition process on the edges of $H_n$. This observation will play a crucial role in the proof of forthcoming Proposition~\ref{cor:splitting}.

\begin{lemma}\label{lem:almost_uniform}
  Let $\beta, \theta>0$ and $\mathcal{C}$ be an admissible function. Let $X$ be sampled from $\mu_{\beta, \theta, \mathcal{C}}$ and let $\{\sigma_{t}\}_{t \geq 0}$ be the associated transposition process. For any $i,k\in \N$, $i < k$, $e\in E$ we have
  \begin{equation}\label{eq:semiuniform}
    \p(\sigma_{i+1}\circ \sigma_{i}^{-1} = e | \fil_i , |X|=k) \in \left[\frac{\Theta^{-2}}{|E|}, \frac{\Theta^2}{|E|} \right].
  \end{equation}
\end{lemma}

\begin{proof}
Let us fix $e \in E$ and a sequence of permutations $\sigma = (\sigma_1,\ldots,\sigma_i)$ such that for any $ j \in \{1, \ldots,  i - 1\}$ the composition $\sigma_{j+1}\circ \sigma_j^{-1}$ is a transposition. Let $U_{\sigma, e}^{i+1}$ be the set of all $X\in \mathfrak{X}$ such that $|X|=k$, $e_{i+1} = e$ and the transposition process associated to $X$ agrees with $\sigma$ up to time $i$. For any $\tilde{e} \in E$ let $T:U_{\sigma, e}^{i+1} \mapsto U_{\sigma, \tilde{e}}^{i+1}$ be the mapping which swaps the $(i+1)$-th point of $X$ from $e$ to $\tilde{e}$, i.e., $(e_{i+1}, t_{i+1}) = (e, t_{i+1})$ in $X$ is replaced by $(\tilde{e}, t_{i+1})$. Note that $T$ is a bijection which preserves the Poisson point process $\mathcal{B}$. Thus we have
	\begin{equation*}
		\frac{\mu_{\beta, \theta, \mathcal{C}}(U_{\sigma, \tilde{e}}^{i+1})}{\mu_{\beta, \theta, \mathcal{C}}(U_{\sigma, {e}}^{i+1} )} =
		 \frac{\int_{\mathfrak{X}}\Ind{U_{\sigma, \tilde{e}}^{i+1}}(X) \theta^{\mathcal{C}(X)} \mathcal{B}(dX)}{\int_{\mathfrak{X}}\Ind{U_{\sigma, e}}^{i+1}(X) \theta^{\mathcal{C}(X)} \mathcal{B}(dX)} =  \frac{\int_{\mathfrak{X}}\Ind{U_{\sigma, {e}}^{i+1}}(X) \theta^{\mathcal{C}(T(X))} \mathcal{B}(dX)}{\int_{\mathfrak{X}}\Ind{U_{\sigma, {e}}^{i+1}}(X) \theta^{\mathcal{C}(X)} \mathcal{B}(dX)}.
	\end{equation*}
By the Lipschitz property \eqref{eq:Lipschitz} of $\mathcal{C}$ clearly we have $|\mathcal{C}(T(X)) - \mathcal{C}(X)|\leq 2$. Thus the integrands in the numerator and the denominator above can differ by a factor of at most $\theta^2$, which leads to the estimate
\[
	{\mu_{\beta, \theta, \mathcal{C}}(U_{\sigma, \tilde{e}}^{i+1})}/{\mu_{\beta, \theta, \mathcal{C}}(U_{\sigma, {e}}^{i+1} )}\leq \Theta^2.
\] From this and an analogous argument for the lower bound it is straightforward to obtain \eqref{eq:semiuniform}.
\end{proof}

In the final arguments we will need statements holding uniformly in a large enough time window before time $|X|$. Thus, conditionally on $|X| = k$, we let $k' = \max\{0, k - 2\lceil n^{11/6}\rceil\}$ and define the time interval $\interval := \{k',\ldots, k - 1 \}$. We will now prove that on $\interval$ the transposition process corresponding to the measure $\mu_{\beta,\theta,\mathcal{C}}$ behaves in a certain sense similarly to the mean-field case (corresponding to $\theta=1$ on the complete graph), i.e., for most values of $k$ the process, when conditioned on $|X| = k$, with high probability has splitting and merging probabilities comparable to the mean-field case.

To formalize this intuition let us introduce a stopping time $\tau$, corresponding to the moment when the cycles lose good isoperimetric properties. Denote by \hfff{def:orbits}{$\orb_{\sigma}(v)$ the cycle of the permutation $\sigma$ containing $v\in V$ and by $\orb_{\sigma}^{\ell}(v)$} its first $\ell \wedge |\sigma|$ elements. We will write $\orb_{t}^\ell(v)$ as a shorthand for $\orb_{\sigma_t}^\ell(v)$.

For constants $c_1, c_2 > 0$ which will be fixed later we define
\begin{align*}
  \tau^\iota & := \inf\{ s \in I\colon \exists{v\in V} \, \iota(\orb_s^{n\log^2 n}(v)) > c_1\log^2 n \}, \\
   \tau^\chi  & := \inf\{s \in I \colon \exists{v\in V} \, |\orb_s(v)| \ge n\log^2 n \;{\rm and} \; \chi(\orb_s^{n\log^2 n}(v)) < c_2\log^2 n)\}, \\
   \tau & := \tau^\iota \wedge \tau^\chi.
\end{align*}
Note that for each $k$, $\tau$ is a stopping time with respect to the natural filtration of the process $(\sigma_t)_{t=0}^k$ on $\{|X| = k\}$.

\begin{proposition} \label{cor:splitting}
Let $\beta, \theta > 0$ be such that $\beta > \Theta/2$ and let $\mathcal{C}$ be an admissible function. Let $X$ be sampled from $\mu_{\beta, \theta, \mathcal{C}}$ and let $\{\sigma_{t}\}_{t \in \{0,\ldots,|X|\}}$ be the associated transposition process.

Then there exists $K_n \subset \N$, $C,c>0$ and constants $c_1,c_2$ in the definition of $\tau$, depending only on $\beta, \theta$ and $\mathcal{C}$, such that the following properties hold.

\begin{enumerate}[(i)]
\item $\p(  |X|\in K_n ) \geq 1 - Ce^{-c \log^2 n}$
\item For $k \in K_n$,  $\p(\tau = \infty| |X| = k) \ge  1 - Ce^{-c \log^2 n}$.
\item Consider  $\ell\geq n\log^2 n$. Denote by $\mathcal{D}_{i}$ the event that in the transition from step $i$ to $i+1$ a cycle of $\sigma_i$  is split into two cycles, one of which has size smaller than $\ell$.  Then for every $i\in I$,
\begin{equation*}
    \p(\mathcal{D}_{i}|\fil_i, |X|=k) \ind{k\in K_n} \ind{\tau > i} \leq C\frac{ \ell}{n^2}.
  \end{equation*}

\item Let $\mathcal{C}_1, \mathcal{C}_2$ be two cycles of $\sigma_i$ such that $|\mathcal{C}_j|\geq n \log^2n$, for $j\in\{1,2\}$. Denote by $\mathcal{M}_{i}$ the event that they are merged in the transition from step $i$ to $i+1$. Then for every $i \in I$,
\begin{equation}
  \p(\mathcal{M}_{i} | \mathcal{F}_i, |X|=k) \geq c\frac{|\mathcal{C}_1| |\mathcal{C}_2| }{n^4} \ind{k\in K_n} \ind{\tau > i}.
\end{equation}
\end{enumerate}

\end{proposition}

The proof is deferred to the end of the next subsection.

\subsection{Isoperimetry and its consequences}

Here we use the notation from the previous section and assume that $\{ \sigma_{t} \}_{t \in \{0, \ldots, |X|\}}$ is a transposition process associated to the distribution $\mu_{\beta, \theta, \mathcal{C}}$. The proof of Proposition \ref{cor:splitting} is rather easy once we know that any (long enough) fragment of $\sigma_t$ is ``spread evenly on the graph''. This is formalised in events $\mathcal{I}^{\iota}$ and $\mathcal{I}^{\chi}$ defined below.

Recall that $\interval = \{k',\ldots, k - 1 \}$, where $k' = \max\{0, k - 2\lceil n^{11/6}\rceil\}$, conditionally on $|X| = k$. Let $c_1, c_2 > 0$ and consider the events
\begin{align}
& \mathcal{I}^{\iota} := \{ \tau^{\iota} \geq |X| \} = \left\{ \forall s \in \interval \, \forall v \in V \, \iota \left( \orb_{s}^{n \log^2 n}(v) \right) \le c_{1} \log^2 n \right\}, \label{eq:assumption1} \\
& \mathcal{I}^{\chi} := \{ \tau^{\chi} \geq |X| \} = \left\{ \forall s \in \interval \, \forall v \in V \, \chi \left( \orb_{s}^{n \log^2 n}(v) \right) \geq c_{2} \log^2 n  \ind{|\orb_{s}(v)| \geq n \log^2 n}\right\}. \label{eq:assumption2}
\end{align}
Finally, let \hfff{def:assumptionI}{$\mathcal{I} := \mathcal{I}^{\iota} \cap \mathcal{I}^{\chi}$}.

We will show that the event $\mathcal{I}$ holds (for appropriate choice of $c_1$, $c_2$) with high probability. As a first step we prove that the conclusions of Proposition \ref{prop:isoperimetry-log2n} and Proposition \ref{prop:isoperimetry-log2n-lower-bound} hold uniformly for cyclic random walks using their bars only up to a certain level.

More precisely, for $s,t \in (0,1)$ let $\mathfrak{X}^{s,t}$ be the restriction of $\mathfrak{X}$ to $[s,t)$, i.e., the space of finite subsets of $E \times [s,t)$. We can define measures $\mu_{\beta, \theta, \mathcal{C}}^{s,t}$ on $\mathfrak{X}^{s,t}$ by a formula analogous to \eqref{eq:measure}, i.e.,
\begin{equation}\label{eq:measure-1}
	  \mu^{s,t}_{\beta, \theta, \mathcal{C}}(U) := \frac{1}{Z^{s,t}_{\beta, \theta, \mathcal{C}}} \int_{\mathfrak{X}^{s,t}}\Ind{U}(Y) \theta^{\mathcal{C}(Y)} \mathcal{B}(dY).
	\end{equation}

If $X \in \mathfrak{X}$ and $X = \{(e_1, t_1), \ldots, (e_k, t_k)\}$, we define the restriction of $X$ to $[s,t)$, denoted by $X^{s,t}$, by including only these pairs $(e_i, t_i)$ for which $t_i \in [s, t)$. We have $X^{s,t} \in \mathfrak{X}^{s,t}$. For simplicity we will write $\mathfrak{X}^{t}$ instead of $\mathfrak{X}^{0,t}$, etc.

For $Y \in \mathfrak{X}^{t}$ let $\Zz^{Y}(v)$ denote the trace of the cyclic random walk (started at $(v,0)$) using the bridges of $Y$ and running on bars of height $t$ instead of height $1$. Note that if $Y$ is distributed according to $\mu_{\beta, \theta, \mathcal{C}}^{t}$ (for some $\beta, \theta, \mathcal{C}$), then the process $(\Zz^{Y}_{T}(v), T \geq 0)$ has the same distribution as $(\Zz^{X}_{T / t}(v)), T \geq 0)$, where $X \in \mathfrak{X}$ is distributed according to $\mu_{t \beta, \theta, \mathcal{C}}$ (this follows directly by properties of the Poisson point process $\mathcal{B}$).

Furthermore, if $X$ is distributed according to $\mu_{\beta, \theta, \mathcal{C}}$, then the law of $X^{t}$ under $\mu_{\beta, \theta, \mathcal{C}}$ is given by $\mu^{t}_{\beta, \theta, \mathcal{C}^{t}}$, where $\mathcal{C}^{t}: \mathfrak{X}^{t}\mapsto \R$ is a function defined by
\begin{equation*}
  \theta^{\mathcal{C}^{t}(Y)} := \int_{\mathfrak{X} } \theta^{\mathcal{C}(Y \cup Z_{|E\times [t,1)})} \mathcal{B}(dZ),
\end{equation*}
where $Z_{|E\times [t,1)} = Z \cap (E \times [t,1))$. It is easy to check that if $\mathcal{C}$ is admissible, then so is $\mathcal{C}^t$.

Let us define the analogues of \eqref{eq:assumption1} and \eqref{eq:assumption2} for the CRW. Recall the sets $\mathcal{O}_k,
\mathcal{O}$, defined in \eqref{eq:very-big-ooo}, and denote by $\mathcal{O}_{k}^Y,\mathcal{O}^Y$ the analogous sets for the process $Y$. For fixed $t_0\in (0,1)$ and $c_1, c_2>0$ we set
\begin{align*}
  \mathcal{A}^{\iota} &:=  \left\{ \forall t \in [t_{0}, 1) \, \forall v \in V \, \iota \left(\mathcal{Z}^{X^{t}}_{t n\log^2 n}(v) \right) \leq c_1 \log^2 n \right\},  \\
  \mathcal{A}^{\chi} &:=  \left\{ \forall t \in [t_{0}, 1) \, \forall v \in V \, \chi \left(\mathcal{O}^{X^{t}}_{n\log^2 n}(v) \right) \geq c_2 \log^2 n \ind{ |\mathcal{O}^{X^{t}}(v)| \geq n \log^2 n} \right\}.
\end{align*}
We also set $\mathcal{A} := \mathcal{A}^{\iota}\cap \mathcal{A}^{\chi}$. We will now prove

\begin{proposition}\label{prop:isoperimetry-log2n-uniform}

Fix $\beta$, $\theta$, $\mathcal{C}$ such that $\beta > \Theta / 2$ and let $t_{0} \in (0,1)$ be such that $t_{0} \beta > \Theta / 2$. Let $X$ be distributed according to $\mu_{\beta, \theta, \mathcal{C}}$. Then there exist $C, c > 0$ and $c_1,c_2>0$ in the definition of $\mathcal{A}^{\iota}$ and  $\mathcal{A}^{\chi}$ above such that
\begin{equation*}
  \Pp \left( \mathcal{A}  \right) \geq 1 - C e^{- c \log^2 n}.
\end{equation*}
\end{proposition}

\begin{proof}
It will be convenient to divide the time interval $[0,1)$ into subintervals small enough so that each of them contains at most one bridge, as then it will be enough to control isoperimetry at the endpoints and use a union bound.

Let $\kappa > 0 $ (to be specified later in the proof). Let $t_0 = s_{0} < s_1 < \ldots < s_k = 1$ be such that $k = \lceil\beta n^2 e^{\kappa \log^2 n}\rceil$ and $|s_{i+1} - s_{i}|\leq \beta^{-1} n^{-2}e^{-\kappa \log^2 n}$ for $ i \in \{0, 1, \ldots, k-1 \}$. For simplicity we will write $X^{i} :=  X^{s_{i}}$ and $X^{i, i+1} := X^{s_{i}, s_{i+1}}$ (and likewise for $\mathfrak{X}$).

Consider the event that there is at most one bridge in each interval $[s_{i}, s_{i+1})$
\begin{equation*}
   \mathcal{E} := \{ \forall_{i  \in \{0, \ldots, k-1 \} } |X^{i,i+1}| \leq 1\}.
\end{equation*}
First note that
\begin{equation}\label{eq:not-too-many-bridges}
 \p(\mathcal{E}^c) \leq C e^{- c \log^2 n}
\end{equation}
for some $C,c > 0$. Indeed, for any fixed $i = 0, \ldots, k - 1$ by Lemma \ref{le:auxilliary-poisson} applied with $t = s_{i}$, $s = |s_{i+1} -  s_{i}|$ we have
\[
\pr{|X^{i,i+1}| > 1} \leq e^{C \lambda_{i}} \pr{Y_{\lambda_{i}} \geq 2},
\]
for some $C > 0$, where $\lambda_{i} \leq \Theta e^{-\kappa \log^2 n}$ and $Y_{\lambda_{i}}$ is a Poisson variable with parameter $\lambda_{i}$. Using the simple estimate $\pr{Y_{\lambda_{i}} \geq 2} \leq \lambda_{i}^2$ (valid for $\lambda_{i}$ small enough) we obtain
\begin{align*}
& \Pp\left( |X^{i,i+1}| > 1 \right) \leq e^{C \Theta e^{-\kappa \log^2 n}} \Theta^2 e^{-2\kappa \log^2 n} \leq C' e^{- 2\kappa \log^2 n}
\end{align*}
for some $C' > 0$. Applying union bound over all $i = 0, \ldots, k-1$, gives $\pr{\mathcal{E}} \leq k \cdot C' e^{- 2\kappa \log^2 n} \leq 2 C' \beta n^2 e^{- \kappa \log^2 n} \leq C e^{- c \log^2 n}$ for some $C, c > 0$.

Now we prove that each of the events defining $\mathcal{A}$ holds with high enough probability. We start with the event $\mathcal{A}^{\iota}$. Recall that $\mathcal{Z}^{X^{t}}(v)$ is the trace of the CRW using bridges of $X^{t}$. Consider the event
\begin{equation*}
   \mathcal{J} := \left\{\forall_{i  \in \{0, 1, \ldots, k \} } \, \forall_{v\in V} \, \iota \left(\mathcal{Z}^{X^{s_{i}}}_{s_{i} n\log^2 n}(v) \right) \leq C \log^2 n \right\}
 \end{equation*}
with $C$ as in Proposition \ref{prop:isoperimetry-log2n}.

As remarked before, for any fixed $i \in \{0, 1, \ldots, k \}$ the law of $X^{s_{i}}$ under $\mu_{\beta, \theta, \mathcal{C}}$ is given by $\mu_{\beta, \theta, \mathcal{C}^{s_{i}}}^{s_{i}}$ and the trace $\mathcal{Z}^{X^{s_{i}}}_{T}(v)$ has the same distribution as the trace $\mathcal{Z}^{\tilde{X}^{i}}_{T / s_{i}}(v)$, where $\tilde{X}^{i}$ has distribution $\mu_{\beta_{i}, \theta, \mathcal{C}^{s_{i}}}$ with $\beta_{i} = s_{i}\beta$ (note that the latter CRW uses bars of height $1$). Note that $\beta_{i} > \Theta / 2$, as by assumption $t_{0}\beta > \Theta / 2$, and we have $\beta_{i} \in [t_{0}\beta, \beta]$.

Therefore, by Proposition \ref{prop:isoperimetry-log2n} we obtain for any fixed $i$

\[
\pr{\forall_{v\in V} \, \iota \left(\mathcal{Z}^{X^{s_{i}}}_{s_{i} n\log^2 n}(v) \right) \leq C \log^2 n} \geq 1 - C e^{- c \log^2 n},
\]
with constants $C,c$ depending only on $\Theta$, $\beta$ and $t_{0}$ (but not on $\mathcal{C}^{s_{i}}$).

Applying a union bound over $i=0,\ldots,k$ we get
\[
\pr{\mathcal{J}^{c}} \leq (k+1) C e^{- c \log^2 n} \leq 4 C \beta n^2 e^{\kappa \log^2 n} e^{- c \log^2 n}.
\]
Now we can fix $\kappa < c/2$ to obtain that $\pr{\mathcal{J}} \geq 1 - C' e^{-c' \log^2 n}$ for some $C', c' > 0$.

On the event $\mathcal{E}$ there is at most one bridge in each interval $[s_{i}, s_{i+1})$, which implies for any $v \in V$
\[
\sup\limits_{t \in [s_{i}, s_{i+1})} \iota \left(\Zz^{X^{t}}_{t n\log^2 n}(v)\right) \leq \max \left\{ \iota \left(\Zz^{X^{s_{i}}}_{s_{i} n\log^2 n}(v)\right), \iota \left(\Zz^{X^{s_{i+1}}}_{s_{i+1} n\log^2 n}(v)\right) \right\}.
\]
Thus on the event $\mathcal{E} \cap \mathcal{J}$ we obtain
\[
\sup\limits_{t \in [t_{0}, 1)} \max\limits_{v \in V} \iota \left(\Zz^{X^{t}}_{t n\log^2 n}(v) \right) \leq C \log^2 n.
\]
Since each event $\mathcal{E}$ and $\mathcal{J}$ occurs with high enough probability, we can take $c_1 = C$ to obtain that $\mathcal{A}^{\iota}$ holds with the required probability.

The proof for the event $\mathcal{A}^{\chi}$ is analogous. Using the same notation as above, let
\begin{equation*}
   \mathcal{K} := \left\{\forall_{i  \in \{0, 1, \ldots, k \} } \, \forall_{v\in V} \, \chi \left(\mathcal{O}^{X^{s_{i}}}_{n\log^2 n}(v)\right) \geq C \log^2 n \text{ or } |\mathcal{O}^{X^{s_{i}}}(v)| < n \log^2 n  \right\}
 \end{equation*}
with $C$ as in Proposition \ref{prop:isoperimetry-log2n-lower-bound}.

As before,  $(\Zz^{X^{s_{i}}}_{T}(v), T \geq 0)$ has the same distribution as $(\Zz^{\tilde{X}^{i}}_{T / s_{i}}(v), T \geq 0)$, where $\tilde{X}^{i}$ is distributed according to $\mu_{\beta_{i}, \theta, \mathcal{C}^{s_{i}}}$ with $\beta_{i} = s_{i}\beta$. Applying Proposition \ref{prop:isoperimetry-log2n-lower-bound} and performing a union bound over $i$ gives us that $\pr{\mathcal{K}} \geq 1 - C' e^{-c' \log^2 n}$ for some $C', c' >0$ depending on $\theta$, $\beta$, $t_{0}$.

On the event $\mathcal{E}$ for each $t  \in [s_{i}, s_{i+1})$ we have
either $\mathcal{O}^{X^{t}}_{t n\log^2 n}(v)= \mathcal{O}^{X^{s_{i}}}_{s_i n\log^2 n}(v)$ and  $\mathcal{O}^{X^{t}}(v)= \mathcal{O}^{X^{s_{i}}}(v)$, or
$\mathcal{O}^{X^{t}}_{t n\log^2 n}(v)= \mathcal{O}^{X^{s_{i+1}}}_{s_{i+1} n\log^2 n}(v)$ and  $\mathcal{O}^{X^{t}}(v)= \mathcal{O}^{X^{s_{i+1}}}(v)$.
Thus $\mathcal{E}\cap \mathcal{K} \subseteq \mathcal{A}^{\chi}$.
We finish by observing that both $\mathcal{E}$ and $\mathcal{K}$ hold with high enough probability, so in the definition of $\mathcal{A}^{\chi}$ we can take $c_2 = C$ with $C$ as in the definition of $\mathcal{K}$.
\end{proof}

Now we can prove that the good isoperimetric properties $\mathcal{I}^\iota$ and $\mathcal{I}^\chi$ defined in \eqref{eq:assumption1} and \eqref{eq:assumption2} hold in the discrete time setting (for some $c_1$, $c_2$) with high probability. Recall that $\mathcal{I} = \mathcal{I}^\iota \cap \mathcal{I}^\chi$.

\begin{lemma}\label{eq:good_isoperimetry_is_common} Let $\beta, \theta>0$ be such that $\beta > \Theta/2$ and $\mathcal{C}$ be an admissible function. There exist $c_1, c_2>0$ in the definitions \eqref{eq:assumption1} and \eqref{eq:assumption2},  $C, c, c'>0$ and $K_n\subset \N$ such that
\[
\p\left(\mathcal{I}\big| |X|= k\right)\geq (1 - Ce^{-c \log^2 n})\ind{k\in K_n}\quad\text{and}\quad\p(|X|\in K_n)\geq 1 - e^{-c' \log^2 n}.
\]
\end{lemma}

\begin{proof}
Let $X$ be sampled from $\mu_{\beta, \theta, \mathcal{C}}$ and fix $t_{0} \in (0,1)$ such that $t_{0} \beta > \Theta / 2$.

We first prove that the event $\mathcal{I} = \mathcal{I^{\iota}} \cap \mathcal{I^{\chi}}$ holds with probability at least $1 - C e^{-c \log^2 n}$, for some $C,c >0$, when $c_1$, $c_2$ are chosen appropriately.

Consider the bridges $\sigma_{s} = (e_{s}, t_{s})$ for $s = |X| - 2\lceil n^{11/6}\rceil,\ldots, |X|$. The event
\[
\mathcal{S} = \{ |X| \geq 2 \lceil n^{11/6}\rceil \} \cap \{ \forall_{ s \in \{|X| - 2 \lceil n^{11/6}\rceil,\ldots, |X| \}} \, t_{s} > t_{0} \}
\]
holds with high probability. Indeed, if $|X| < 2 \lceil n^{11/6}\rceil$ or $t_{s} \leq t_{0}$ for some $s$, then necessarily $|X \cap (E \times [t_{0}, 1))| \leq 2 \lceil n^{11/6} \rceil$. An application of the second part of Lemma \ref{le:auxilliary-poisson} with $k_n = 2 \lceil n^{11/6}\rceil$ shows that with probability at least $1 - C e^{-c n^2}$, for some $C,c >0$, this does not happen.

From now on we work on the event $\mathcal{S}$. As in the proof of Proposition \ref{prop:isoperimetry-log2n-uniform} let $X^{t_{s}}$ denote the restriction of $X$ to the interval $[0, t_{s})$ and let $\Zz^{X^{t_{s}}}(v)$ be the trace of the corresponding cyclic random walk started at $v$.

By the construction of the cyclic random walk we have $\orb_{s}(v) \subseteq \Zz^{X^{t_{s}}}(v)$ and $\orb_{s}^{n \log^2 n}(v) \subseteq \Zz^{X^{t_{s}}}_{t_{s} n \log^2 n}(v)$. Recalling the definition of the event $\mathcal{A}^{\iota}$, by Proposition \ref{prop:isoperimetry-log2n-uniform} there exist $C, c > 0$ such that with probability at least $1 - C e^{- c \log^2 n}$ we have
\[
\sup\limits_{t \in [t_{0}, 1)} \max\limits_{v \in V} \iota \left(\mathcal{Z}^{X^{t}}_{t n\log^2 n}(v) \right) \leq C \log^2 n.
\]
As on $\mathcal{S}$ we have $t_{s} > t_{0}$ for all $s = |X| - 2 \lceil n^{11/6} \rceil, \ldots, |X|$, together with the observation about the orbits this shows that $\mathcal{I}^{\iota}$, with $c_1 = C$ in the definition \eqref{eq:assumption1}, holds with probability at least $1 - C' e^{-c' \log^2 n}$ for some $C',c' > 0$.

For the proof that the event $\mathcal{I}^{\chi}$ holds with high probability, note that $\mathcal{O}^{X^{t_{s}}}_{t_s n \log^2n}(v) = \orb_{s}^{n \log^2 n}(v)$, in particular $\orb_{s}(v) = \mathcal{O}^{X^{t_{s}}}(v)$. Therefore, recalling the definition of $\mathcal{A}^{\chi}$, we can use Proposition \ref{prop:isoperimetry-log2n-uniform} to conclude that with probability at least $1 - C' e^{- c' \log^2 n}$, for some $C',c' >0$, we have
\[
\forall s\in I \, \forall v \in V \, \chi \left(\orb^{n\log^2 n}_{s}(v) \right) \geq C \log^2 n \ind{ |\orb_{s}(v)| \geq n \log^2 n}.
\]
As before on the event $\mathcal{S}$ we have $t_{s} > t_{0}$, so if we take $c_{2} = C$ in \eqref{eq:assumption2} we obtain that the event $\mathcal{I}^{\chi}$ holds with probability at least $1 - C' e^{- c' \log^2 n}$ for some $C', c' > 0$.

Now suppose that $C,c > 0$ are such that $\pr{\mathcal{I}} \geq 1 - C e^{- c \log^2 n}$. Let
\[
K_{n} := \left\{ k \in \N \colon \Pp(\mathcal{I} \big| |X|= k) \geq 1 - C e^{- \frac{c}{2} \log^2 n}\right\}.
\]
We write by definition of $K_{n}$
\begin{align*}
\pr{\mathcal{I}^{c}} = & \sum\limits_{k \in K_{n}} \Pp(\mathcal{I}^{c} \big| |X|=k) \pr{|X| = k} + \sum\limits_{k \notin K_{n}} \Pp(\mathcal{I}^{c} \big| |X|=k) \pr{|X| = k} \\
& \geq \pr{|X| \notin K_{n}} \cdot C e^{- \frac{c}{2} \log^2 n}
\end{align*}
and now the lower bound on $\pr{\mathcal{I}}$ together with a simple calculation gives us
\[
\pr{|X| \in K_{n}} \geq 1 - \frac{C e^{-c \log^2 n}}{C e^{- \frac{c}{2} \log^2 n}} = 1 - e^{-\frac{c}{2} \log^2 n}
\]
as desired, which proves the second assertion of the lemma with $c' = c/2$.
\end{proof}

We will now relate good isoperimetric properties of cycles to the probabilities of splits and merges in the corresponding transposition process.

\begin{lemma}\label{lem:intensityOfSplit}
Let $\sigma$ be a permutation and let $e = \{u,w\}$ be an edge chosen at random according to a distribution $\{p_e\}_{e\in E}$ satisfying $ c / |E|\leq  p_e \leq C/|E|$, for some $C,c > 0$. Let $(u,w)$ denote the transposition of endpoints of $e$.
\begin{enumerate}[(i)]
\item Suppose that for some $k,m \in \N$ and each $v$ we have
	\begin{equation}\label{eq:isoperimetricAssumption}
		\iota(\orb_\sigma^k(v)) \leq m.
	\end{equation}
Then for any $\ell \geq k$ the probability that a cycle of $\sigma$ is split in $(u,w) \circ \sigma$ into two cycles, one of which has size smaller than $\ell$, is at most
	\[
		 \frac{4 C\ell }{k n} m.
	\]
\item Suppose that for some $k,m \in \N$ and each $v$ satisfying $|\orb_{\sigma}(v)| \geq k$ we have
	\begin{equation}\label{eq:isoperimetricAssumption}
		\chi(\orb_\sigma^k(v)) \geq m.
	\end{equation}
Then given two cycles $\mathcal{C}_{1}$, $\mathcal{C}_{2}$ of $\sigma$ of length at least $k$, the probability that they are merged in $(u,w) \circ \sigma$ into one cycle is at least
	\[
		 \frac{c}{2} \frac{|\mathcal{C}_{1}||\mathcal{C}_{2}| m^2}{n^2 k^2}.
	\]
\end{enumerate}
\end{lemma}

\begin{proof}
We start with the proof of (i). Fix a vertex $v \in V$, let $D,L$ denote respectively the column and the row containing $v$, and let $\orb^{-\ell}_{\sigma}(v) = \orb^{\ell}_{\sigma^{-1}}(v)$. The number of $w \in V$ for which $(v,w)$ is an edge such that a cycle of $\sigma$ is split in $(v,w) \circ \sigma$ into two cycles, one of which has size smaller than $\ell$, is equal to
\[
     |(\orb^{\ell}_\sigma(v)\cup \orb^{-\ell}_\sigma(v))\cap(D\cup L \setminus \{v\})| \leq  2\iota(\orb^{\ell}_\sigma(v)\cup \orb^{-\ell}_\sigma(v)).
\]
By dividing $\orb^{\ell}_\sigma(v)\cup \orb^{-\ell}_\sigma(v)$ into pieces of length $k$ and exploiting subadditivity of $\iota$ we obtain
\[
	\iota(\orb^{\ell}_\sigma(v)\cup \orb^{-\ell}_\sigma(v)) \leq \left\lceil \frac{2\ell}{k} \right\rceil \max_{v'\in V} \iota(\orb_\sigma^k(v')) \leq \left\lceil \frac{2\ell}{k} \right\rceil m.
\]
Thus for fixed $v$ there are at most $\lceil \frac{2 \ell}{k}\rceil m$ edges with one endpoint equal to $v$ which would cause a cycle of $\sigma$ to split with one of the resulting pieces smaller than $\ell$. By our assumptions each such edge $e$ is chosen with probability $p_{e} \leq \frac{C}{|E|}$. As each vertex has degree $2(n-1)$ in $H_{n}$, we obtain that the total probability of such a split is at most
  \[
    \frac{C}{n-1}\left\lceil \frac{2\ell}{k} \right\rceil m \leq \frac{4 C\ell}{kn}m,
  \]
as desired.

For the proof of (ii), consider two cycles $\mathcal{C}_{1} = \orb_{\sigma}(v_1)$, $\mathcal{C}_{2} = \orb_{\sigma}(v_2)$ of length at least $k$. By the assumption of the lemma we have $\chi\left(\orb_{\sigma}^{k}(v_j)\right) \geq m$ for $j=1,2$. By dividing $\mathcal{C}_{j}$ into segments of length $k$ and recalling the definition of $\chi$, we obtain that each $\mathcal{C}_{j}$ has at least $\lfloor \frac{|\mathcal{C}_{j}|}{k}\rfloor m \geq \frac{1}{2}\frac{ |\mathcal{C}_{j}|m}{k}$ vertices in each row and each column of $H_{n}$. This implies that there are at least $2n\left( \frac{1}{2}\frac{ |\mathcal{C}_{1}|m}{k}\right)\left( \frac{1}{2}\frac{ |\mathcal{C}_{2}|m}{k}\right) = \frac{n}{2} \frac{|\mathcal{C}_{1}||\mathcal{C}_{2}|m^2}{k^2}$ edges joining a vertex from $\mathcal{C}_{1}$ with a vertex from $\mathcal{C}_{2}$.

For each such an edge $e$ we have $p_{e} \geq \frac{c}{|E|}$. As choosing such an edge results in a merge between $\mathcal{C}_{1}$ and $\mathcal{C}_{2}$, we obtain
\[
\p(\mbox{$\mathcal{C}_{1}$ and $\mathcal{C}_{2}$ are merged in $(u,w) \circ \sigma$}) \geq \frac{c}{n^2 (n-1)} \cdot \frac{n}{2} \frac{|\mathcal{C}_{1}||\mathcal{C}_{2}|m^2}{k^2} \geq \frac{c}{2} \frac{|\mathcal{C}_{1}||\mathcal{C}_{2}| m^2}{n^2 k^2},
\]
as desired.

\end{proof}

Now we can finally employ Lemmas \ref{lem:almost_uniform}, \ref{eq:good_isoperimetry_is_common} and \ref{lem:intensityOfSplit} to prove Proposition \ref{cor:splitting}.

\begin{proof}[Proof of Proposition \ref{cor:splitting}]
Let $K_n \subset \N$ and $C,c,c' >0$ be as in Lemma \ref{eq:good_isoperimetry_is_common} (in particular $(i)$ of Proposition \ref{cor:splitting} is satisfied). We have
\[
\Pp \left(\mathcal{I} \big| |X| = k \right) \geq \left(1 - C e^{-c \log^2n}\right) \ind{k \in K_n},
\]
which in particular implies $(ii)$, i.e.,
\begin{equation}\label{eq:tau-large-conditionally}
\Pp \left( \tau = \infty\big| |X| = k \right) \geq \left(1 - C e^{-c \log^2n}\right) \ind{k \in K_n}.
\end{equation}
Fix $i \in I$ and let $\mathcal{C}_{1} = \orb_{i}(v_{1})$, $\mathcal{C}_{2} = \orb_{i}(v_{2})$ for some $v_{1},v_{2} \in V$.  Let $\ell \geq n \log^2 n$. As on the event $\{\tau > i\}$ we have $\iota \left( \orb_{i}^{n \log^2 n}(v) \right) \le c_{1} \log^2 n$ for all $v$, by Lemma \ref{lem:almost_uniform} and part (i) of Lemma \ref{lem:intensityOfSplit} we obtain
\[
\p(\mathcal{D}_{i}|\mathcal{F}_i, |X|=k) \ind{k\in K_n} \ind{\tau > i} \leq \frac{4 \Theta^{-1} \ell}{n \cdot n \log^2 n} c_{1} \log^2 n = C' \frac{\ell}{n^2}
\]
for some $C' > 0$, giving $(iii)$.

Let us now pass to the proof of (iv). Fix $i \in I$ and let $\mathcal{C}_{1} = \orb_{i}(v_{1})$, $\mathcal{C}_{2} = \orb_{i}(v_{2})$ be two cycles of length at least $n\log^2 n$. On the event $\{\tau > i\}$ we have $\chi\left(\orb_{i}^{n \log^2 n}(v)\right) \geq c_2 \log^2 n$ for $j=1,2$. Thus by Lemma \ref{lem:almost_uniform} and part (ii) of Lemma \ref{lem:intensityOfSplit} we obtain
\[
\p(\mathcal{M}_{i}|\mathcal{F}_i, |X|=k) \geq  \frac{\Theta^{-1}}{2} \frac{|\mathcal{C}_{1}||\mathcal{C}_{2}| (c_2 \log^2 n)^2}{n^2 (n \log^2 n)^2} \ind{k\in K_n} \ind{\tau > i} \geq C'' \frac{|\mathcal{C}_{1}||\mathcal{C}_{2}|}{n^4} \ind{k\in K_n} \ind{\tau > i}
\]
for some $C''>0$ as desired.

\end{proof}

\end{section}

\begin{section}{Macroscopic cycles}

In this section we will prove our main results, namely Theorem \ref{thm:main_theorem_upperbound} and Theorem \ref{thm:main_theorem_lowerbound}. We will need one more ingredient, which is a variant of an argument due to Schramm, described in the next subsection.

In most of this section we will work with the general transposition process introduced in Section \ref{sec:transposition_process}.

\begin{subsection}{Schramm's argument}\label{sec:schramm-argument}
In this standalone part we develop ideas of \cite{Schramm:2011fk}. Our contribution is to rephrase them in terms of an abstract split-merge process. We will then apply the results of this subsection to transposition processes coming from measures $\mu_{\beta, \theta, \mathcal{C}}$ (see Section \ref{sec:proof main_theorem_upperbound}) and we believe that they might be useful in more general cases.

Let $V$ be a finite set and $h\in \N$. We say that $\{\sigma_k \}_{k\in \{0,1,\ldots, h\}}$ is a \emph{random split-merge process} over $V$ if any $\sigma_k$ is a (random) partition of $V$. We refer to the sets of this partition as components. We call a random split-merge process \emph{simple} if in transition from $\sigma_k$ to $\sigma_{k+1}$ (which we will call step $k$) we allow only
 \begin{itemize}
   \item  a component to be split into two,
   \item  two components to be merged.
 \end{itemize}
We denote by $\{\fil_k\}_{k\in \{0,1,\ldots, h\}}$ the natural filtration of the process and write \hfff{def:cycleSize}{$\vx{k}{\ell}$ for the set of $v\in V$ which belong to components of $\sigma_k}$ of size at least $\ell \in \N$. We set $\interval_{h} := \{0, 1, \ldots, h-1\}$.

Now we state the main result of this section. Its proof is essentially the same as proofs in Section 2 of \cite{Schramm:2011fk}, with only very slight modifications.

\begin{lemma}\label{le:Schramm_key_lemma}
Let $h\in \N$ and $\{\sigma_k \}_{k\in \{0,1,\ldots, h\}}$ be a simple split-merge process over a finite set $V$. Let $t_0\in \interval_h$, $\delta \in (0,1]$, $\varepsilon \in (0,1/8)$ and $j\in \N$ be such that $ 2^j \le \varepsilon \delta |V|$. Assume moreover that for a stopping time $\tau$ the following conditions hold
\begin{enumerate}[(i)]
\item there exists $c_1 > 0$ such that for any $k\in \interval_h$, $\ell \geq 2^{j}$  we have
\begin{align*}
    \p(\text{some component }&C\text{ of $\sigma_k$ is split in transition from step $k$ to $k+1$ into }C_1, C_2 \\
    &\text{ and } \min(|C_1|, |C_2|)\leq \ell |\fil_k) \ind{\tau > k} \leq  c_1  \frac{\ell}{|V|}.
\end{align*}
\item there exists $c_2>0$ such that for any $k\in \interval_h$ and any two components $C_1, C_2 \in \sigma_k $ such that $|C_1|,|C_2|\geq 2^j$ we have
\begin{equation*}
  \p(\text{components }C_1, C_2 \text{ are merged in transition from step $k$ to $k+1$} |\fil_k) \geq c_2 \frac{|C_1||C_2|}{|V|^2} \ind{\tau > k}.
\end{equation*}
\end{enumerate}
Then there exist $c_3, c_4>0$, depending only on $c_1, c_2$, such that if
\begin{equation}\label{eq:growing_interval}
t_1 := t_0 + \lceil\Delta t\rceil, \quad \Delta t =  c_3 \delta^{-1}\frac{|V|}{2^j}\log_2\left(\frac{|V|}{2^j}\right),
\end{equation}
satisfies $t_1 \leq h$, then
\begin{align}\label{eq:Schramm-ineq}
\E \Big(| \vx{t_0}{2^j}\setminus \vx{t_1}{\varepsilon \delta |V|}| \big|\mathcal{F}_{t_0} \Big) \ind{ | \vx{t_0}{2^j} | \geq \delta |V|} \le c_4 \delta^{-1}\varepsilon |\log_2(\varepsilon\delta)| |V| + c_{4} \delta^{-1} |V| \p(\tau \leq t_1|\mathcal{F}_{t_0}).
\end{align}
\end{lemma}

In other words, if sufficiently many vertices are in components of size at least $2^j$ at time $t_0$, most of them will be in components of size $\varepsilon \delta |V|$ at time $t_1$ (unless the split-merge properties (i) and (ii) fail, which is reflected in the second term).

The proof is an implementation of the following simple idea. Fix $K := \lceil \log_2(\varepsilon \delta |V|)\rceil$ and set milestones $t_0 = T_j < T_{j+1} < \ldots < T_{K} = t_1$. In each epoch $\{T_j, T_j+1, \ldots, T_{j+1}\}$ we expect the size of moderately large components to grow by a factor of two. The subtlety lies in formalising this statement and finding the correct lengths of the epochs. They need to be long enough so that most components have a chance to merge into bigger ones and at the same time to be short enough so that not too many splits occur. It turns out that the proper choice is
\begin{equation}\label{eq:epoch_lenghts}
  T_{i+1} - T_{i} =  m_i := \lceil a_i\rceil, \quad a_i  := \frac{4}{c_2} \delta^{-1} \frac{|V|}{2^i}  \log_2  \Big(\frac{|V|}{2^i} \Big).
\end{equation}
Now the form of $\Delta t$ in \eqref{eq:growing_interval} is rather natural.

In proofs below we will use $N$ as shorthand for $|V|$ and $\log$ for $\log_2$. All unspecified constants are assumed to be independent of $\delta$ and $\varepsilon$. We implicitly assume that the condition $t_1 \leq h$ is met. We can also assume that $\p(\{|\vx{t_0}{2^j}| \ge \delta |V|) > 0$. For notational simplicity fix an arbitrary $A \in \fil_{t_0}$ such that $\p(A\cap \{|\vx{t_0}{2^j}| \ge \delta |V|\}) > 0$ and set $\cp(\cdot) = \p(\cdot|A\cap \{|\vx{t_0}{2^j}| \ge \delta |V|\})$. We denote by $\cE(\cdot)$ the corresponding expectation. The inequality \eqref{eq:Schramm-ineq} is then equivalent to
\begin{displaymath}
  \cE|\vx{t_0}{2^j}\setminus \vx{t_1}{\varepsilon \delta N}| \le  c_4 \delta^{-1}\varepsilon|\log(\varepsilon \delta)|N + c_4 \delta^{-1} N \cp(\tau \leq t_1)
\end{displaymath}
for all admissible sets $A$.

The idea behind the proof of Lemma \ref{le:Schramm_key_lemma} consists of showing that for any epoch $i\in \{j, \ldots, K-1\}$ the number of vertices in $\vx{T_{i}}{2^{i}} \setminus \vx{T_{i+1}}{2^{i+1}}$ must be small.

The first reason for a vertex to fall in this set is splitting. Namely, by $\bar{S}_{i}$ we denote the set of vertices $v \in V$ which at some time $t \in \{T_{i}, \ldots, T_{i+1}-1\}$ belong to a component which in transition to time $t+1$ is split and $v$ ends up in a component of size smaller than $2^{i+1}$. In Lemma \ref{lem:number-of-splits} we show that $\bar{S}_{i}$ is small.

The second reason is failure of the components to merge. We define $\bar{M}_{i} := \vx{T_{i}}{2^{i}} \setminus (\vx{T_{i+1}}{2^{i+1}} \cup \bar{S}_{i})$, the set containing vertices whose components did not split, but failed to merge into a bigger one. In Lemma \ref{lem:number-of-non-merges} we analyze $\bar{M}_{i}$ in more detail and prove that it is small as well.

 We also denote
\begin{equation*}
S_{k} := \bigcup_{i = j}^{k} \bar{S}_i, \quad M_{k} := \bigcup_{i = j}^k \bar{M}_i.
\end{equation*}
The following inclusion reveals the rationale behind the above definitions:
\begin{equation}\label{eq:loss_of_vertices_in_Schramm_argument}
  \vx{T_j}{2^j} \setminus \vx{T_{k+1}}{2^{k+1}} \subset S_k \cup M_k.
\end{equation}

We first deal with splits

\begin{lemma}\label{lem:number-of-splits}
Under assumptions of Lemma \ref{le:Schramm_key_lemma} there exists $c>0$ such that
\begin{equation}\label{eq:first-part}
  \cE \left( |{S}_{K-1}| \ind{\tau > T_K} \right) \leq c \varepsilon|\log_2(\varepsilon \delta)| N.
\end{equation}
\end{lemma}

\begin{proof}
Fix $i\in \{j, \ldots, K-1\}$ and let
\[
\mathcal D_t := \{\textrm{a split occured in transition from time $t$ to $t+1$ creating a component of size smaller than}\; 2^{i+1}\}.
\]
Whenever a split creates one or possibly two components of size smaller than $2^{i+1}$, at most $2 \cdot 2^{i+1}$ vertices are added to the set $\bar{S}_i$. Thus
  \begin{displaymath}
  |\bar{S}_i| \ind{\tau > T_i} \le 2\cdot 2^{i+1} \sum_{t=T_i}^{T_{i+1}-1}\Ind{\mathcal D_t} \ind{\tau > T_i}.
  \end{displaymath}
Applying assumption (i) of Lemma \ref{le:Schramm_key_lemma} we obtain $\ind{\tau > T_i}\cp(\mathcal D_t) \le  C 2^{i+1}/N,$ for some $C > 0$. As $\{\tau > T_K\} \subset\{\tau > T_i\}$, we have
  $$\cE \left( |{S}_{K-1}| \ind{\tau > T_K} \right) \leq 8 C \sum_{i=j}^{K-1} m_i 2^{2i}/N.$$
The rest of the proof follows by calculations employing the form of $m_i$.
\begin{align*}
\frac{1}{N}\sum_{i=j}^{K-1} m_i 2^{2i} &\le \frac{2}{N}\sum_{i=j}^{K-1} a_i 2^{2i} = \frac{8}{c_{2}} \delta^{-1} \sum_{i=j}^{K-1} 2^{i} \log  \Big(\frac{N}{2^i} \Big) \leq \frac{8}{c_{2}} \delta^{-1} \sum_{i=0}^{K-1} 2^{i} \log  \Big(\frac{N}{2^i} \Big)
\\
& = \frac{8}{c_{2}} \delta^{-1} \Big((2^K - 1)\log N - 2 + 2^{K+1} - 2^K K\Big)\\
& \le \frac{8}{c_{2}} \delta^{-1} \Big( ( \log N  - K)2^K +2^{K+1}\Big).
\end{align*}
Recalling $K = \lceil \log_2(\varepsilon \delta N)\rceil$, we check easily that $2^{K+1} \le 4\varepsilon \delta N$ and $(\log N  - K)2^K \le 2 |\log (\varepsilon \delta)|\varepsilon\delta N$, therefore \eqref{eq:first-part} follows.
\end{proof}

Analysing $M_i$ is somewhat tricky. We introduce an additional index
\begin{equation} \label{eq:schramm_stopping_time}
\sigma := \min\left\{ i\in \{j, \ldots, K-1\} \colon |\vx{t}{2^i}| < \delta N/2\text{ for some } t\in \{T_i +1, \ldots, T_{i+1}\}\right\},
\end{equation}
with the convention $\sigma = +\infty$ when the set is empty. Now we can state

\begin{lemma}\label{lem:number-of-non-merges}
Under assumptions of Lemma \ref{le:Schramm_key_lemma} we have
\begin{equation}\label{eq:second-part}
  \cE \left( |{M}_{(\sigma - 1) \wedge (K-1)}| \ind{\tau > T_K} \right)\leq 2\varepsilon \delta N.
\end{equation}
\end{lemma}

\begin{proof}
Consider an epoch $i\in \{j, \ldots, K-1\}$ and let $v$ be any vertex. Recalling the definition of $\bar{M}_{i}$, if $v \in \bar{M}_i$, then at time $T_i$ the vertex $v$ is in a component of size at least $2^i$, at time $T_{i+1}$ the component of $v$ is smaller than $2^{i+1}$, and $v \notin \bar{S}_{i}$, so there is no splitting between these two times which would put $v$ in a component smaller than $2^{i+1}$. Therefore, we have $v \in \vx{t}{2^i}\setminus \vx{t}{2^{i+1}}$ for all $t \in \{T_i,\ldots, T_{i+1}-1\}$. For any step $t$ of the epoch consider the event $\mathcal{A}_t$ that the component of $v$ merges with another component of size at least $2^i$. Notice that $\{v\in \bar{M}_i\} \subset \bigcap_{t}\mathcal{A}_t^c$, where the intersection is over $t \in \{T_i,\ldots, T_{i+1}-1\}$. Indeed, had any $\mathcal{A}_t$ happened the component of $v$ would have been of size at least $2^{i+1}$ and, as $v \notin \bar{S}_i$, it would have survived until the end of the epoch. Denote also the event $\mathcal{E}_t := \{|\vx{t}{2^i}| \geq \delta N/2 \} \cap \{v \in \vx{t}{2^i}\setminus \vx{t}{2^{i+1}}\}$. By assumption (ii) of Lemma \ref{le:Schramm_key_lemma} we have
\begin{equation*}
    \cp(\mathcal{A}_t|\mathcal{F}_t) \geq c_2 2^i \frac{\delta N/2 - 2^{i+1}}{N^2} \id_{\mathcal{E}_t} \ind{\tau > t} \geq c_2 2^i \delta \frac{1/2 - 2\varepsilon}{N} \id_{\mathcal{E}_t} \ind{\tau > t} \geq \frac{c_2}{4} \frac{2^i\delta}{N} \id_{\mathcal{E}_t} \ind{\tau > t}.
  \end{equation*}
Using the facts above we conclude
\begin{align*}
  \cp(\{v \in \bar{M}_i\} &\cap \{ \sigma > i \} \cap \{ \tau \geq T_{i+1} \}) \leq \cE\left( \id_{\bigcap_{t\in \{T_i,\ldots, T_{i+1}-1\}} (\mathcal{A}_t^c \cap \mathcal{E}_t) \cap \{ \tau > t \}} \right) \leq \\
  & \cE\rbr{\id_{\bigcap_{t\in \{T_i,\ldots, T_{i+1}-2\}} (\mathcal{A}_t^c\cap \mathcal{E}_t) \cap \{ \tau > t \}} \id_{\mathcal{E}_{T_{i+1} -1}} \ind{ \tau > T_{i+1} -1 } \cp\big(\mathcal{A}_{T_{i+1}-1}^{c} |\mathcal{F}_{T_{i+1}-1}\big)} \leq \\
  & \Big( 1- \frac{c_2}{4}\frac{2^i\delta}{N}\Big) \cdot\cE\big( \id_{\bigcap_{t \in \{T_i,\ldots, T_{i+1}-2\}} (\mathcal{A}_t^c\cap \mathcal{E}_t) \cap \{ \tau > t \} } \big)\leq \ldots \leq \Big( 1- \frac{c_2}{4}\frac{2^i\delta}{N} \Big)^{m_i}.
\end{align*}
The choice of $m_i$ in \eqref{eq:epoch_lenghts} is such that
\begin{equation*}
  \cp(\{v \in \bar{M}_i\} \cap \{ \sigma > i \} \cap \{ \tau \geq T_{i + 1} \}|\mathcal{F}_{T_i}) \leq   2^i/N.
\end{equation*}
The rest of the proof follows by the estimate
\begin{align*}
\cE \left( |M_{(\sigma - 1)\wedge (K-1)}| \ind{\tau > T_K} \right)\le \sum_{i=j}^{K-1} \cE \left( |\bar{M}_i| \ind{\sigma > i} \ind{\tau \geq T_{i+1}} \right) \le  N  \frac{2^{K}}{N} \le 2\varepsilon \delta N.
\end{align*}
\end{proof}
We are now ready for
\begin{proof}[Proof of Lemma \ref{le:Schramm_key_lemma}]
By \eqref{eq:loss_of_vertices_in_Schramm_argument} on the event $\{\sigma < +\infty \}$  we have $\vx{T_j}{2^j} \setminus \vx{T_{\sigma}}{2^{\sigma}} \subset S_{\sigma - 1} \cup M_{\sigma - 1}$. Consider now $t \in \{T_{\sigma}+1,\ldots,T_{\sigma+1}\}$ such that $|\vx{t}{2^\sigma}| < \delta N/2$ and any $v \in \vx{T_j}{2^j} \setminus \vx{t}{2^\sigma}$. Recall that we work on the event $\{|C_{t_0}(2^j)| \ge \delta N\}$. If $v \notin \vx{T_{\sigma}}{2^{\sigma}}$ then $v \in S_{\sigma - 1} \cup M_{\sigma - 1}$, otherwise $v \in \vx{T_{\sigma}}{2^\sigma} \setminus \vx{t}{2^\sigma}$ and so $v \in \bar{S}_{\sigma}\subset S_{K-1}$. Thus $\vx{T_j}{2^j} \setminus \vx{t}{2^\sigma} \subset M_{(\sigma - 1)\wedge (K-1)}\cup S_{K-1}$. Since $|\vx{T_j}{2^j}| \ge \delta N$, we conclude  that $\{\sigma < +\infty \} \subset \{ |M_{(\sigma - 1)\wedge (K-1)}\cup S_{K-1}|\geq \delta N/2 \}$.

Now let $\Delta :=\cE|\vx{t_0}{2^j}\setminus \vx{t_1}{\varepsilon \delta N}|$. Using \eqref{eq:loss_of_vertices_in_Schramm_argument} we have $ \Delta \leq \cE |M_{K-1} \cup S_{K-1}| $. Furthermore,
\begin{align*}
   \Delta \leq&  \cE \left( |M_{K-1} \cup S_{K-1}|\ind{\sigma = +\infty}\right) + N \cp(\sigma < +\infty)  \\
      &\leq \cE |M_{(\sigma - 1)\wedge(K-1)} \cup S_{K-1}| + N \cp(|M_{(\sigma - 1)\wedge (K-1)}\cup S_{K-1}|\geq \delta N/2) \\
      & \leq (1 + 2/\delta) \cE |M_{(\sigma - 1)\wedge(K-1)} \cup S_{K-1}|\\
	& \leq (1 + 2/\delta) (\cE |M_{(\sigma - 1)\wedge(K-1)}| + \cE | S_{K-1}|),
\end{align*}
where in the second line we used Markov's inequality. Furthermore, we write
\begin{align*}
& \cE |M_{(\sigma - 1)\wedge(K-1)}| + \cE | S_{K-1}| \\
& \leq \cE \left(|M_{(\sigma - 1)\wedge(K-1)}| \ind{\tau > T_K}\right) + \left(\cE | S_{K-1}|\ind{\tau > T_K}\right) + 2 N \cp(\tau \leq T_K)
\end{align*}

Now applying Lemma \ref{lem:number-of-splits} and Lemma \ref{lem:number-of-non-merges} we get (recalling that $T_K = t_1$)
\begin{align*}
& \Delta \leq (1 + 2/\delta)(2 \varepsilon \delta N + c \varepsilon|\log(\varepsilon \delta)|N) + (2 + 4/\delta)N \cp(\tau \leq t_1).
\end{align*}
As $\delta < 1$ and $\varepsilon < 1/8$, one readily checks that the right hand side is bounded from above by
\[
c' \delta^{-1}\varepsilon|\log(\varepsilon \delta)|N + c' \delta^{-1} N \cp(\tau \leq t_1)
\]
for some $c' > 0$, which concludes the proof.
\end{proof}
\end{subsection}

\begin{subsection}{Mesoscopic cycles}\label{sec:mesoscoic cycles}

In this section we will work in the setting of general transposition process, introduced in Section \ref{sec:transposition_process}, for $X$ sampled from $\mu_{\beta, \theta, \mathcal{C}}$. All quantities like $\vx{t}{\cdot}$ are implicitly related to this process. We focus on the supercritical phase $\beta > \Theta/2$. Our aim, formalized in Proposition \ref{le:interval}, is to show that a substantial fraction of vertices belongs to mesoscopic cycles (of size at least $n\log^2 n$). This result will be used in the next section as an input to Lemma \ref{le:Schramm_key_lemma} to prove the existence of macroscopic cycles. The crucial ingredient that we use are the isoperimetric properties of the cycles stated in Proposition \ref{cor:splitting}.

\begin{proposition}\label{le:interval}
Let $\beta >\Theta/2$. There exist $\delta>0$ and sets $K_n \subset \{\lceil n^{11/6}\rceil, \lceil n^{11/6}\rceil + 1,\ldots\}$ such that $\lim_{n\to \infty} \p(|X|\in K_n)=1$ and
\[
   \lim_{n\to \infty} \inf_{k\in K_n } \min_{t \in \{k - \lceil n^{11/6}\rceil, \ldots, k\} } \p \left( |\vx{t}{n\log^2 n}| \geq \delta n^2 \big| |X|=k\right) = 1.
\]
 \end{proposition}

The proof of the proposition hinges on a coupling between the generalized transposition process and a random graph process. Let $s \in \{0, \ldots, |X|\}$ and \hfff{def:graphProcess}{consider a process $G^{s} = \{G^{s}_u\}_{u \in \{0, \ldots, |X|-s\}}$ of random graphs} on the vertex set $V$ defined as follows. Initially, $G_0^s$ is a graph whose connected components are the cycles of $\sigma_s$. There might be many graphs satisfying this property and for our purposes it will not matter which one is chosen. Next, for any $u \in \{1, \ldots, |X|-s\}$ the edge $e$ corresponding to transposition $e_{s+u}$ (i.e., $\sigma_{s+u} = e_{s+u} \circ \sigma_{s+u-1}$) is added to the edge set of $G^{s}_{u}$.

Recall that $\vx{t}{\ell}$ is the set of vertices which belong to cycles of length at least $\ell$ in $\sigma_{t}$. Correspondingly, let \hfff{def:graphComponents}{$\vg{s, u}{\ell}$ be the set of vertices} of $G^{s}_u$ which belong to connected components of size at least $\ell$. Importantly,  any cycle of $\sigma_{s+u}$ is contained in a connected component $G^{s}_u$. Hence, it follows that $\vx{s+u}{\ell}\subset \vg{s,u}{\ell}$ for any $s,u, \ell \in \N$.

There are two key ingredients in the proof Proposition \ref{le:interval}.  First, $G^{s}$ is monotonic and thus can be studied using standard random graph techniques. In particular Lemma \ref{lemma:sprinkle} below shows that macroscopic clusters emerge quickly in $G^{s}$. Second, on sufficiently short time intervals the difference $\vg{s,u}{\ell}\backslash \vx{s+u}{\ell}$ is small, which is formalized in Lemma \ref{lemma:schramm_estimate}.

Consider an interval $I \subset \{0,\ldots, |X|\}$, $k \in \N$ and let $c_1 > 0$ be the constant from Proposition \ref{cor:splitting}. Let
  \[
   \mathcal I_k(I) := \left\{\sup_{t \in I}\max_{v \in V}\iota(\orb^{k}_t(v)) \leq c_1 \log^2 n\right\}
  \]
denote the event that fragments of permutation orbits for $t \in I$ have good isoperimetric properties. Our first lemma quantifies the quality of the coupling between $\mathcal{C}_{s+u}$ and $G^{s}_{u}$.

\begin{lemma}\label{lemma:schramm_estimate}
Let $s \in \N$ and let $\Delta$ be an $|X|$-measurable $\N$-valued random variable. Suppose that $k, \ell \in \N$ satisfy $k\leq\ell$. Then for some $C>0$ on the event $s+\Delta \leq |X|$ we have
\[
\E\left[\max_{u \in \{0,\ldots, \Delta\}}| \vg{s,u}{\ell}\backslash \vx{s+u}{\ell}| \big| |X|\right]
\leq C\frac{\ell^2\Delta}{k n}\log^2 n + 2\ell\Delta\,\p\left(\mathcal I_k(\{s, \ldots, s+\Delta\})^c\big| |X|\right).
\]
\end{lemma}

\begin{proof}
The proof is an adaptation of \cite[Lemma 4.2]{milos-sengul} to the discrete time setting and the case $\Theta \neq 1$. Let $I$ be the set of $u \in [0,\Delta-1]$ such that $\sigma$ experiences a fragmentation at time $s+u$ which splits a cycle and at least one of the resulting cycles has length less than $\ell$.

From Lemma \ref{lem:almost_uniform} and point (i) of Lemma~\ref{lem:intensityOfSplit} we obtain that at any time $u$ the (conditional) probability of a fragmentation in which one piece is smaller than $\ell$ is at most
\[
	\frac{4\Theta^2\ell }{k n} c_{1} \log^2 n + \Ind{\{\mathcal I_k(\{u\})^c\}}.
\]
Hence we see that
\begin{equation}\label{eq:I_schramm}
    \E[|I|\big| |X|]\leq \frac{4\Theta^2\ell\Delta}{k n} c_{1}^{2} \log^2n + \Delta\,\p\left(\mathcal I_k(\{s, \ldots, s+\Delta\})^c\big| |X|\right).
\end{equation}

Let $u \in [1,\Delta]$ and consider any cycle $\gamma$ of $\sigma_{s+u}$ such that $\gamma \subset \vg{s,u}{\ell}\backslash \vx{s+u}{\ell}$, that is, $\gamma$ is contained in a component of $G^{s}_u$ of size at least $\ell$ and $|\gamma|<\ell$.
Then it follows that there must have been a vertex $v\in \gamma$ such that the cycle containing $v$ must have fragmented at some time in $\{s,\ldots, s+u\}$ producing a cycle of size smaller than $\ell$.

For $t\in \{s,\ldots, s+u\}$, let $\gamma_{t}^{(v)}$ be the cycle of $\sigma_{t}$ containing $v$.
Let $t'\in \{s,\ldots, s+u-1\}$ be the maximal time such that the size of $\gamma_{t'}^{(v)}$ jumps downwards, that is, the cycle containing $v$ experiences a fragmentation.
Then at this time $t'$, $\sigma$ experiences a fragmentation which splits a cycle into two and at least one of the resulting cycles has length less than $\ell$. Note that the cycle thus obtained is a part of $\gamma$. It follows that $t'\in I$ and consequently $| \vg{s,u}{\ell}\backslash \vx{s+u}{\ell}| \leq 2\ell |I|$.
Taking suprema and using~\eqref{eq:I_schramm} we obtain the desired result.
\end{proof}

The second lemma quantifies how quickly big clusters emerge in the random graph process $G^{s}_{u}$.
\begin{lemma}\label{lemma:sprinkle}
For any $\delta\in(0,1/8)$ there exists a sequence $\{a_n\}_{n\in \N}$ such that the following holds: $a_n \nearrow 1$ and for any $s, u, \ell, h \in \N$ satisfying $\log^2 n \leq \ell\leq n^2$ and  $u \geq (n^2/\sqrt{\ell})\log n$,  $s + u \leq h$ and  $\p(|\vg{s,0}{\ell}|\geq \delta n^2, |X| = h) > 0$, we have
\[
      \p\left(|\vg{s,u}{\delta n^2/8}| \geq \delta n^2/8 \big| |\vg{s,0}{\ell}|\geq \delta n^2, |X|=h\right)\geq a_n.
\]
\end{lemma}

As the proof is an adaptation of \cite[Lemma 4.3]{milos-sengul} and is of graph-theoretical nature, it is moved to Appendix \ref{ap:intensity}.

In the two subsequent lemmas we show that in the supercritical phase the random graph process has macroscopic clusters for times close to $|X|$.

\begin{lemma}\label{lemma:pre_giant_component}
Let $\beta > \Theta/2$. There exists $\delta > 0$ such that
\[
	\lim_{n\to+\infty}\p\left(|X| \ge 2\lceil n^2/\log n\rceil \text{ and } |\vg{0,|X| - {2\lceil n^2/\log n\rceil}}{\delta n^2}|\geq\delta n^2 \right) = 1.
\]
\end{lemma}
\begin{proof}
Let $I = [0,a)$, for $a < 1$, be an interval such that $\beta'=\beta|I|\Theta^{-1}>1/2$. We also set $J := [0,1)\setminus I$. Recall the notation used in Lemma \ref{lm:poisson-approximation}. By the monotonicity of the graph process it is enough to show that for some $c > 0$
\begin{equation}\label{eq:percolation_hamming}
	\lim_{n\to+\infty}\p\left(|\vg{0,|X_{E\times I}|}{c n^2}|\geq c n^2\right) = 1, \quad \lim_{n\to+\infty}\p\left(|X_{E\times J}|\geq 2 \lceil n^2/{\log n}\rceil\right) = 1.
\end{equation}
For $X \in \mathfrak{X}$ we set $\hat{X} := \{e: \exists_{t\in[0,1)} \:(e,t) \in X\}$. We intend to compare $\hat X_{E\times I}$ and $\hat X_{E\times J}$ with the Bernoulli percolation on $E$. To this end we use the Holley theorem \cite[Theorem 4.6]{GEORGII20011} with $\mathcal{L}=1$ and $S=\{0,1\}$, where $1$ indicates that $e$ is an open edge (i.e., belongs to a given set). For any $e\in E$ by the first part of Lemma \ref{lm:poisson-approximation} we get
\begin{align*}
 \p(e\in \hat X_{\{e\}\times I}| \hat X_{(E\setminus \{e\}) \times I}) &= \E \left(  \p(| X_{\{e\}\times I}|\geq 1 | X_{(E\times [0,1))\setminus (\{e\}\times I)}) \Big|  \hat X_{(E\setminus \{e\}) \times I} \right) \\
 &\geq 1 - e^{-\frac{\beta |I| \Theta^{-1}}{n-1}}= 1 - e^{-\frac{\beta'}{n-1}}\geq \frac{\beta'}{n-1} - o(1/n) =: p_n.
\end{align*}
This yields that $\hat X_{\{e\}\times I}$ is stochastically bounded from below by the Bernoulli percolation process with the probability of opening $p_n$. As $\beta'>1/2$, for $n$ large enough this process is in the supercritical phase. As a consequence, to get the first convergence in \eqref{eq:percolation_hamming} we can apply known results on the emergence of the giant component in supercritical percolation (see Theorem 1.1 in \cite{vanderHofstad2010} and the discussion therein; we note that the results of \cite{vanderHofstad2010} are formulated only for $p_n$ in the critical window, but the proof techniques carry over to the strictly supercritical case, see e.g., discussion in Section 3 of the cited paper).

Similarly, $\hat X_{E\times J}$ is bounded by a percolation process with the probability of opening $\geq 1 - e^{-\frac{\beta |J| \Theta^{-1}}{n-1}}\geq \frac{\beta |J| \Theta^{-1}}{2(n-1)}$. From this bound and the fact that $|E|=O(n^3)$ we infer that $|X_{E\times J}|\geq |\hat X_{E\times J}|\geq 2\lceil n^2/\log n\rceil$ with probability converging to $1$.
\end{proof}

\begin{lemma}\label{lemma:giant_component}
Let $\beta > \Theta/2$. There exist $\delta>0$, sets $K_n\subset I_n := \{\lfloor n^2 / \log n\rfloor,\ldots, \lfloor n^2 \log n\rfloor \}$ and a sequence $\{a_n\}_{n\geq 1}$ such that the following holds: $\lim_{n\to +\infty}\p(|X|\in K_n) = 1$, $a_n \nearrow 1$ and
	\begin{equation}\label{eq:giant_component_exists}
	 \p\left(|\vg{0,k-\lfloor n^2/\log n\rfloor}{\delta n^2}|\geq \delta n^2 \big| |X| = k \right) \geq a_n \ind{k\in K_n}.
	\end{equation}

\end{lemma}
\begin{proof}
Denote $\mathcal{A}_n := \{|\vg{0, |X| - 2\lceil n^2/\log n\rceil }{\delta n}|\geq \delta n^2\}$, with $\delta$ as in Lemma \ref{lemma:pre_giant_component}, and set $L_n := \{k\in \mathbb{N}:\p\left(\mathcal{A}_n \big| |X| = k \right)\leq c_n\}$, for $c_n\in(0,1)$ to be fixed later. As $\p(\mathcal{A}_n)\leq \p(|X|\in L_n) c_n + \p(|X|\not\in L_n)$, by a simple calculation we get
	\begin{equation*}
		\frac{1-\p(\mathcal{A}_n)}{1 - c_n}\geq\p(|X|\in L_n).
	\end{equation*}
By Lemma \ref{lemma:pre_giant_component} we have $\p(\mathcal{A}_n) \to 1$, so we can find $c_n, \delta$ such that $c_n\nearrow 1$ and the left-hand side converges to $0$. Consequently, we have $\p(|X|\in K_n')\to 1$ for $K_n' := \N \setminus L_n$. We set $K_n := K_n' \cap I_n$. Using Lemma \ref{le:auxilliary-poisson} we see that $\p(|X| \notin I_n) \to 0$ as $n \to \infty$.
Thus we get $\lim_{n\to +\infty}\p(|X|\in K_n) = 1$ and $\p\left(|\vg{0,k-2\lfloor n^2/\log n\rfloor}{\delta n}|\geq \delta n^2 \big| |X| = k \right) \geq c_n \ind{k\in K_n}$ as desired.
\end{proof}

The proof of Proposition \ref{le:interval} follows by making comparisons of the random graph process and the generalized interchange process on appropriate time intervals, as made possible by Lemma \ref{lemma:schramm_estimate}. Below we make only two such comparisons. It is possible to iterate Lemma \ref{lemma:schramm_estimate} more times on shorter and shorter time intervals, thus getting a tighter control on the difference $\vg{s,u}{\ell}\backslash \vx{s+u}{\ell}$. This method was used in \cite{milos-sengul} to prove the existence of cycles of size $n^{2 - \varepsilon}$ for any $\varepsilon > 0$. To the best of our knowledge this method alone cannot be pushed to obtain macroscopic cycles. Instead, in the next section we use modified Schramm's argument presented in Lemma \ref{le:Schramm_key_lemma}, together with Proposition \ref{le:interval} as a prerequisite.

\begin{proof}[Proof of Proposition \ref{le:interval}]
Let $\delta_1, K_n^1, a_n^1$ be $\delta, K_n, a_n$ asserted by Lemma \ref{lemma:giant_component}. We first apply Lemma \ref{lemma:schramm_estimate} with $k_1 = c_1 \log^2 n / 2, \ell_1 = n^{1/3} \log n$, $s_1=0$ and $\Delta_1 = h \leq n^2 \log n$, conditionally on $|X|=h$. Clearly, $\mathcal{I}_{k_1}(I) = \Omega$ for any interval $I$. Thus for $h \in K_n^1$ using Markov's inequality we get for some $C>0$
\[
\p\left[\max_{t \in \{0,\ldots, h\}}| \vg{0,t}{\ell_1}\backslash \vx{t}{\ell_1}|\geq \delta_1 n^{2}/2 \big| |X| = h\right] \leq \frac{C (n^{1/3} \log n)^2 n^2\log n}{(\log^2n/2)n(\delta_1 n^{2}/2)} \log^2 n = \frac{4C}{\delta_1} n^{-1/3} \log^3 n \to 0.
\]
Combining this with Lemma \ref{lemma:giant_component} we get
	  \begin{equation}\label{eq:first_step_to_bootstarp}
	  	\lim_{n\to +\infty}\inf_{h\in K_n^1}\p\left(\min_{t\in \{h-2 \lceil n^{11/6}\rceil, \ldots, h\}}|\vx{t}{\ell_1}|\geq \delta_1 n^2/2 \big|  |X|=h\right) = 1.
	  \end{equation}
Using this result we will be able to repeat the argument above on a short time interval contained in the supercritical phase. Crucially, on this interval we can use Lemma \ref{eq:good_isoperimetry_is_common}, which will let us obtain a much better estimate in Lemma \ref{lemma:schramm_estimate}.

Fix $\tau\in \{h - \lceil n^{11/6} \rceil, \ldots, h\}$, conditionally on $|X| = h$, and set $\Delta_2 := \lceil n^{11/6} \rceil$. Let $s_2 := \tau - \Delta_2$. Starting from \eqref{eq:first_step_to_bootstarp} we first apply Lemma \ref{lemma:sprinkle} with $\ell_1= n^{1/3}\log n$ (note that $\Delta_2 \ge \frac{n^2}{\sqrt{l_1}}\log n$), getting
\begin{equation}\label{eq:second_step_to_bootstrap}
	\lim_{n\to +\infty}\inf_{h\in K_n^1} \p\left(|\vg{s_2,\Delta_2}{\delta_2 n^2}| \geq \delta_2 n^2 \big| |X|=h\right)=1,
\end{equation}
for some $\delta_2>0$. Fix $\ell_2 =k_2=n \log^2 n$. Let  $K_n^2 = K_n^1 \cap K_n$, where $K_n$ is given by Lemma \ref{eq:good_isoperimetry_is_common}. Let $h\in K_n^2$, using  Lemma \ref{lemma:schramm_estimate} we estimate
\begin{align*}
	\sup_{h\in K_n^2}\E\left[|\vg{s_2,\Delta_2}{\ell_2}\backslash \vx{\tau}{\ell_2}| \big| |X|=h\right]
    &\leq \frac{C_1 (n\log^2n)^2 n^{11/6}}{(n \log^2 n) n} \log^2 n + 2(n \log^2 n)n^{11/6} e^{-c \log^2 n} \\
    &= C_2 n^{11/6} \log^4 n.
\end{align*}
for some $C_1, C_2, c>0$. Markov's inequality implies
\[
\sup_{h\in K_n^2}\p\left[|\vg{s_2,\Delta_2}{\ell_2}\backslash \vx{\tau}{\ell_2}| \geq n^{11/6} \log^6 n \big| |X| = h\right] \leq C_2 /\log^{2} n.\]
This combined with \eqref{eq:second_step_to_bootstrap} yields the statement of the proposition.
\end{proof}

\end{subsection}

\begin{subsection}{Macroscopic cycles in the supercritical phase $\beta>\Theta/2$. Proof of Theorem \ref{thm:main_theorem_upperbound}} \label{sec:proof main_theorem_upperbound}

Now we are ready to show our main result. Recall the general transposition process $\{\sigma_t\}$ introduced in Section \ref{sec:transposition_process} and, importantly, that $\sigma_{|X|} = \sigma(X)$, where $\sigma(X)$ defined in \eqref{eq:permutation_associated} is the main object of our study.

\begin{proposition}\label{prop:macroscopic_cycles_proposition}
Let $\beta >\Theta/2$. There exist sets $K_n \subset \mathbb{N}$ such that $\lim_{n\to \infty} \p(|X|\in K_n)=1$ and
\[
     \lim_{\varepsilon \to 0}\liminf_{n\to \infty} \inf_{k\in K_n} \p(\textrm{there exists a cycle of } \sigma(X) \text{ of length at least $\varepsilon n^2$}\big| |X|=k) = 1.
\]
\end{proposition}
\begin{proof}

Let $\delta$ and $K_n^1$ be respectively $\delta$ and $K_n$ asserted by Proposition \ref{le:interval}. This proposition shows that cycles of size at least $n \log^2 n$ are common. We will use this information to show the existence of macroscopic cycles. The key role in this proof is played by Schramm's argument, encapsulated in Lemma \ref{le:Schramm_key_lemma}, and isoperimetric properties of cycles. The latter imply that the split-merge process behaves similarly to the mean-field (the complete graph) case, which is stated conveniently in Proposition \ref{cor:splitting}. We denote sets $K_n$ from that proposition by $K_n^2$.

In the proof we work conditionally on $|X|=k$, where $k\in K^1_n\cap K^2_n$. Fix $\varepsilon\in(0,1/8)$ and consider the largest $j\in\mathbb{N}$ such that $2^{j}\leq n\log^2n$. Employing the notation from Lemma \ref{le:Schramm_key_lemma}, we set $t_0 = k - \Delta t$, where $\Delta t$ is given by \eqref{eq:growing_interval}. With this choice we have $t_1 = k$. Observe that for $j$ as above we have $\Delta t = o(n)$, in particular $t_{0} \geq k - \lceil n^{11/6}\rceil$. Let $\tau$ be the stopping time from Proposition \ref{cor:splitting}. One easily checks that conditions (i) and (ii) of Lemma \ref{le:Schramm_key_lemma} are fulfilled by assertions (iii) and (iv) of Proposition \ref{cor:splitting}. Consequently we get
\[
	\E \left(| \vx{t_0}{2^j}\setminus \vx{k}{\varepsilon \delta n^2}| \big|\mathcal{F}_{t_0}, |X|=k \right) \ind{|\vx{t_0}{2^j}| \geq \delta n^2} \le c \delta^{-1}\varepsilon |\log_2(\varepsilon\delta)| n^2 + c \delta^{-1} n^2 \p\left(\tau \leq k \big| \mathcal{F}_{t_0}, |X|= k \right),
\]
for some $c>0$.
Using Markov's inequality we get
\[
	\p \Big(| \vx{t_0}{2^j}\setminus \vx{k}{\varepsilon \delta n^2}|\geq \delta n^2/2 \big|\mathcal{F}_{t_0}, |X|=k \Big) \ind{|\vx{t_0}{2^j}| \geq \delta n^2} \le 2c \delta^{-2} \left( \varepsilon |\log_2(\varepsilon\delta)| + \p\left(\tau \leq k \big| \mathcal{F}_{t_0}, |X|= k \right) \right).
\]
Consequently,
\begin{align*}
& \p \Big( |\vx{k}{\varepsilon \delta n^2}|\geq \delta n^2/2 \big| |X|=k \Big) \ge \E \left( \p \Big(| \vx{t_0}{2^j}\setminus \vx{k}{\varepsilon \delta n^2}| < \delta n^2/2 \big|\mathcal{F}_{t_0}, |X|=k \Big) \ind{|\vx{t_0}{2^j}| \geq \delta n^2} \big| |X|=k \right) \\
& \geq \E \left( \ind{|\vx{t_0}{2^j}| \geq \delta n^2}  \left(1 - 2c \delta^{-2} \left( \varepsilon |\log_2(\varepsilon\delta)| + \p\left(\tau \leq k \big| \mathcal{F}_{t_0}, |X|= k \right) \right) \right) \right) \\
& \geq  \p\left({|\vx{t_0}{2^j}| \geq \delta n^2}\big| |X|=k\right)\left( 1 - 2c \delta^{-2} \varepsilon |\log_2(\varepsilon\delta)| \right) - 2c \delta^{-2} \p\left(\tau \leq k \big| |X|= k \right).
\end{align*}
By Proposition \ref{le:interval}, with our choice of $j$ and $t_{0}$ the first probability on the right hand side approaches $1$ as $n \to \infty$ uniformly over $k\in K^1_n\cap K^2_n$. The second probability goes to $0$ uniformly over $k\in K^1_n\cap K^2_n$ by (ii) of Proposition \ref{cor:splitting}. Noticing that $\lim_{n\to+\infty}\p(|X| \in K_n)=1$ and taking the limit $\varepsilon \to 0$ we obtain our result.
\end{proof}

\begin{proof}[Proof of Theorem \ref{thm:main_theorem_upperbound}]
Let $K_{n}$ be the sets claimed in Proposition \ref{prop:macroscopic_cycles_proposition}.
We write
\begin{align*}
	& \p(\text{ there exists a cycle of } \sigma(X) \text{ of length at least $\varepsilon n^2$}) \\
& \geq \sum\limits_{k \in K_n} \p(\text{ there exists a cycle of } \sigma(X) \text{ of length at least $\varepsilon n^2$} \big| |X| = k) \p(|X|=k) \\
& \geq \p(|X| \in K_n) \inf_{k\in K_n} \p(\textrm{there exists a cycle of } \sigma(X) \text{ of length at least $\varepsilon n^2$}\big| |X|=k).
\end{align*}
Since $\p(|X| \in K_n) \to 1$ as $n \to \infty$, by taking $\liminf$ over $n \to \infty$ and then the limit $\varepsilon \to 0$ using Proposition \ref{prop:macroscopic_cycles_proposition} we obtain the statement of the theorem.
\end{proof}

\begin{subsection}{Microscopic cycles in the subcritical phase $\beta<\Theta^{-1}/2$. Proof of Theorem \ref{thm:main_theorem_lowerbound}}\label{sec:proof-of-subcritical}

We will now sketch a proof of the statement about the behavior of cycle lengths in subcritical phase. This is a much easier task than in the supercritical phase. Our main result follows directly from the following
\begin{proposition}\label{prop:macroscopic_cycles_proposition}
Let $\beta <\Theta^{-1}/2$. There exist $C>0$ and sets $K_n \subset \mathbb{N}$ such that $\lim_{n\to \infty} \p(|X|\in K_n)=1$ and
\[
	\lim_{n\to \infty} \inf_{k\in K_n} \p(|\vx{k}{C\log n}|=0\big| |X|=k) = 1.
\]
\end{proposition}
\begin{proof}
Let $C>0$ and recall that $\vx{t}{C\log n}\subset \vg{0,t}{C \log n}$. For $X\in \mathfrak{X}$ we consider $\bar{X}:= \{e: (e,t)\in X \}$. Similarly as in Lemma \ref{lemma:giant_component} we can prove that $\bar{X}$ is stochastically bounded from above by the Bernoulli percolation process with the probability of opening an edge being $\frac{\beta'}{n-1}$ for some $\beta'<1/2$. Now the result follows by a rather standard argument using coupling with branching processes or a random walk (see e.g., \cite[Theorem 2.3.1]{MR2656427}).
\end{proof}
\end{subsection}

\end{subsection}
\end{section}
\appendix
\begin{section}{Appendix -- Concentration of point processes}\label{sec:appendix_concentation}
The proofs of our auxiliary lemmas concerning counting processes will be all based on the following well known result (see e.g., \cite{almostsure} or \cite[Chapter II.6]{bremaud}).

\begin{theorem}\label{thm:time-change}
Let $Y$ be a counting process with intensity $\lambda$. Let $\Lambda_t = \int_0^t \lambda_s ds$ be the compensator of $Y$. Then (on an enlarged probability space) there exists a Poisson process $N$ with intensity one such that almost surely for all $t \ge 0$, $X_t = N_{\Lambda_t}$.
\end{theorem}

\begin{proof}[Proof of Lemma \ref{lm:concentration-upper-bound}]
Let $A \in \mathcal{F}_\sigma$ be any event of nonzero probability. The process $\widetilde{Y}_t = Y_{\sigma+t} - Y_\sigma$ is a counting process with intensity $\widetilde{\lambda}_t = \lambda_{\sigma+t}$ with respect to the filtration $\widetilde{\mathcal{F}}_t = \mathcal{F}_{\sigma+t}$ and the conditional probability $\widetilde{\p} = \p(\cdot|A)$.

Set $\widetilde{\Lambda}_t = \int_0^t \widetilde{\lambda}_s ds$ and note that $\tau-\sigma$ is a stopping time with respect to the filtration $\widetilde{\mathcal{F}}_{t}$. Let $N$ be the Poisson process of intensity one, given for $\widetilde{Y}$ by Theorem \ref{thm:time-change}. We have
\begin{align*}
  \widetilde{\p}(\{Y_{\tau} - Y_\sigma \ge r\}\cap \{\Lambda_\tau - \Lambda_\sigma\le  \ell\} ) & = \widetilde{\p}(\{\widetilde{Y}_{\tau-
  \sigma} \ge r\}\cap \{\widetilde{\Lambda}_{\tau - \sigma} \le  \ell\})
  \le \widetilde{\p}(N_\ell \ge r).
\end{align*}
If $r \geq \ell$, by using the form of the Laplace transform for the Poisson distribution we get
\begin{displaymath}
  \widetilde{\p}(N_\ell \ge r) \le \inf_{u \ge 0} \exp\Big(\ell (e^u-1) - u r \Big) \le \exp\left(-r \log \left( \frac{r}{e \ell}\right) - \ell \right).
\end{displaymath}
Going back to the original probability measure, we conclude that for any $A \in \mathcal{F}_\sigma$,
\begin{displaymath}
  \p(\{Y_{\tau} - Y_\sigma \ge r\}\cap\{\Lambda_\tau - \Lambda_\sigma \le \ell\}\cap A) \le \p(A)\p(X
  \ge r),
\end{displaymath}
and for $r \geq \ell$
\[
\p(X \ge r) \le \exp\left(-r \log \left( \frac{r}{e \ell}\right) - \ell \right),
\]
which implies the lemma.
\end{proof}

\begin{proof}[Proof of Lemma \ref{lm:concentration-lower-bound-super}]
Again, consider any $A \in \mathcal{F}_\sigma$ with positive probability and the process $\widetilde{Y}_t = Y_{\sigma+t} - Y_\sigma$, which is a counting process with intensity $\widetilde{\lambda}_{t} = \lambda_{\sigma+t}$ with respect to the filtration $\widetilde{\mathcal{F}}_t = \mathcal{F}_{\sigma+t}$. Let $\widetilde{\p} = \p(\cdot|A)$. We have $\widetilde{\Lambda}_t = \int_0^t \widetilde{\lambda}_s ds = \Lambda_{\sigma+t} - \Lambda_\sigma$. In particular if $N$ is the Poisson process given for $\widetilde{Y}$ by Theorem \ref{thm:time-change}, we get
\begin{align*}
  \widetilde{\p}(\left\{ Y_{\tau}-Y_{\sigma}\leq\ell(1-\delta)\right\} \cap\left\{ \Lambda_{\tau} - \Lambda_{\sigma}\geq\ell\right\}) &=
  \widetilde{\p}(\{\widetilde{Y}_{\tau-\sigma} \le \ell(1-\delta)\}\cap \{\widetilde{\Lambda}_{\tau-\sigma} \ge \ell\}) \\
  &\le \widetilde{\p}(N_\ell \le \ell(1-\delta)).
\end{align*}
Using the form of the Laplace transform of $N_\ell$ and Chebyshev's inequality we obtain
\begin{align*}
  \widetilde{\p}(N_{\ell} \le \ell(1-\delta))&\leq\inf_{a\geq0}\exp\left((e^{-a}-1)\ell+a\ell(1-\delta)\right)\\
  &\leq\inf_{a\geq0}\exp\left(\frac{1}{2}a^{2}\ell-a\ell\delta\right)=\exp\left(-\frac{1}{2}\delta^{2}\ell\right),
\end{align*}
  where in the second step we have used the elementary inequality $e^{-a}-1+a\leq\frac{1}{2}a^{2}$
  valid for $a\geq0$.
Thus we get
\begin{displaymath}
  \p(\left\{ Y_{\tau}-Y_{\sigma}\leq\ell(1-\delta)\right\} \cap\left\{ \Lambda_{\tau} - \Lambda_{\sigma}\geq\ell\right\}|A) \le \p(X \le (1-\delta)\ell) \le \exp\left(-\frac{1}{2}\delta^{2}\ell\right),
\end{displaymath}
for arbitrary $A \in \mathcal{F}_\sigma$ of positive probability, which implies the lemma.
\end{proof}

\begin{lemma}\label{le:counting-process-half-line}
Let $Y$ be a counting process with bounded intensity $\lambda$. Consider two bounded stopping times $\sigma,\tau$. Then for any $\beta > 1$, with probability one
\begin{displaymath}
  \p(\{\exists_{u \in [0,\tau-\sigma]} Y_{\sigma+u} - Y_\sigma  < u-1\} \cap \{\forall_{u \in [0,\tau-\sigma]} \lambda_{\sigma+u} \ge \beta\}|\mathcal{F}_\sigma) \le 1-q,
\end{displaymath}
for some $q > 0$ depending only on $\beta$.
\end{lemma}

\begin{proof}
Fixing $A \in \mathcal{F}_\sigma$ with $\p(A) > 0$ and using notation from the proof of Lemma \ref{lm:concentration-upper-bound}, we have
\begin{align*}
&\widetilde{\p}(\{\exists_{u \in [0,\tau-\sigma]} Y_{\sigma+u} - Y_\sigma < u-1\} \cap \{\forall_{u \in [0,\tau-\sigma]} \lambda_{\sigma+u} \ge \beta\}) \\
&=  \widetilde{\p}(\{ \exists_{u \in [0,\tau-\sigma]} N_{\widetilde{\Lambda}_u}  < u-1\}\cap \{\forall_{u \in [0,\tau-\sigma]} \lambda_{\sigma+u} \ge \beta\}) \\
&\le \widetilde{\p}( \exists_{u \ge 0} N_{\beta u}  < u - 1).
\end{align*}
The law of large numbers and the Markov property for the Poisson process implies that for $\beta > 1$ the last probability is bounded by $1 - q$ for some $q > 0$ depending only on $\beta$. Since $A$ in the above argument is arbitrary, we obtain the lemma.
\end{proof}

\begin{lemma}\label{le:lower-bound-on-interval}
Let $Y$ be a counting process with bounded intensity $\lambda$. Consider two bounded stopping times $\sigma, \tau$. Then for any $\beta > 1$, $s \ge 0$, with probability one,
\begin{displaymath}
  \p(\{\exists_{u \in [s,\tau-\sigma]} Y_{\sigma + u} - Y_\sigma \le u + a s\}\cap \{\forall_{u \in [0,\tau-\sigma]} \lambda_{\sigma+u} \ge \beta\}|\mathcal{F}_\sigma) \le
  e^{-cs},
\end{displaymath}
where $a = \frac{\beta - 1}{2} > 0$ and $c = \frac{1}{2}\Big( 1 - \frac{1+a}{\beta}\Big)^2\beta$.
\end{lemma}

\begin{proof}
  Denote the event in question by $\mathcal{E}$ and fix $A \in \mathcal{F}_\sigma$ with positive probability. Using the notation from the proof of Lemma \ref{lm:concentration-upper-bound} and arguments from the proof of Lemma \ref{le:counting-process-half-line} we get
  \begin{displaymath}
    \widetilde{\p}(\mathcal{E})\le \p(\exists_{u \ge s} N_{\beta u} \le u +  a s) \le \p\Big(\exists_{u \ge \beta s} \frac{N_u}{u} \le 1-\rho\Big),
  \end{displaymath}
where $N$ is a Poisson process with intensity one and $\rho = 1 - \frac{1 + a}{\beta} \in (0,1)$. Let $\mathcal{G}_u$ be the $\sigma$-field generated by $\{N_t\colon t\ge u\}$ and note that for any
  $0 < u < t$ we have $\E (N_u/u|\mathcal{G}_t) = N_t/t$.
Thus for any $a \in \R$ the process $(\exp(a N_u/u))_{u > 0}$ is a reversed submartingale. Using Doob's maximal inequality for reversed submartingales we get
\begin{align*}
&\p\Big(\exists_{u \ge \beta s} \frac{N_u}{u} \le (1-\rho)\beta\Big) = \inf_{b>0} \p\Big( \sup_{u \ge \beta s} \exp(-b N_u/u) \ge \exp(-b (1-\rho))\Big) \\
&\le \inf_{b > 0} \E \exp\left(- b \frac{N_{\beta s}}{\beta s} + b(1-\rho)\right) = \inf_{b > 0} \exp\left(\beta s(e^{-b/(\beta s)}-1) + b(1-\rho)\right)\\
& = \inf_{b > 0} \exp(s\beta (e^{-b} - 1) + b\beta s(1-\rho)) \le \exp\Big(-\frac{1}{2}\rho^2  \beta s\Big).
\end{align*}
As in previous lemmas, since $A$ is arbitrary, this implies the assertion.
\end{proof}

\end{section}

\begin{section}{Appendix -- Estimates on the number of bridges}\label{ap:poisson}

Here we collect useful estimates enabling us to compare the distribution of bridges for general $\theta > 0$ with the i.i.d. case, i.e., $\theta = 1$.

In the first lemma we show that for $\theta$ not necessarily equal to $1$ the number of bridges using any subset of edges can still be approximated by a Poisson distribution (with parameter depending on $\theta$).

For any measurable $A\subset E\times[0,1)$ and a configuration $X\in \mathfrak{X}$ we denote $X_A := \{x \in X \colon x\in A\}$.

\begin{lemma}\label{lm:poisson-approximation}
	Let $\underline{\lambda}_n := \frac{\beta \Theta^{-1}}{n-1}$, $\bar{\lambda}_n := \frac{\beta \Theta}{n-1}$. Then for any measurable $A \subset E\times[0,1)$ we have
  \begin{equation}\label{eq:stochastic_domination_bounds}
	\p(|X_A|\geq 1 | X_{E\times[0,1)\setminus A}) \in \left[ 1- e^{-|A|\underline{\lambda}_n}, 1- e^{-|A|\bar{\lambda}_n} \right],
  \end{equation}
  and
  \begin{equation}\label{eq:stochastic_domination_bounds2}
	\p(|X_A|\geq k | X_{E\times[0,1)\setminus A}) \leq e^{|A|(\bar{\lambda}_n - \underline{\lambda}_n)} \pr{ Y \geq k},
  \end{equation}
where $Y$ has Poisson distribution with parameter $\bar{\lambda}_n |A|$.

\end{lemma}

\begin{proof}

For $X \in \mathfrak{X}$ and $A \subset E\times[0,1)$ let $U^{ = 0}_{A} = \{X\in \mathfrak{X} \colon |X_A| = 0 \}$ and $U^{\geq k}_A = \{X\in \mathfrak{X} \colon |X_A| \geq k \}$. Furthermore, let $V$ be any event measurable with respect to $X_{E\times[0,1) \setminus A}$ such that $\mathcal{B}(V)>0$. We have
\[
	\frac{\p(U^{\geq k}_A|V)}{\p(U^{=0}_A|V)} = \frac{\mu_{\beta, \theta, \mathcal{C}}(U^{\geq k}_{A}\cap V)}{\mu_{\beta, \theta, \mathcal{C}}(U^{=0}_A \cap V )} = \frac{\int_{\mathfrak{X}}\Ind{U_{A}^{\geq k}}(X_A) \Ind{V} \theta^{\mathcal{C}(X_A \cup X_{E\times[0,1) \setminus A})} \mathcal{B}(dX)}{\int_{\mathfrak{X}}\Ind{U_{A}^{=0}}(X_A) \Ind{V} \theta^{\mathcal{C}(X_{E\times[0,1) \setminus A})} \mathcal{B}(dX)}.
\]
By the Lipschitz property of $\mathcal{C}$ we have $|\mathcal{C}(X_A \cup X_{E\times[0,1) \setminus A}) - \mathcal{C}(X_{E\times[0,1) \setminus A})|\leq |X_A|$. Furthermore, using independence of $X_A$ and $X_{E\times[0,1) \setminus A}$ under $\mathcal{B}$ we get
\begin{equation}\label{eq:tmp_poisson}
	\frac{\p(U^{\geq k}_A|V)}{\p(U^{=0}_A|V)}\leq \frac{\int_{\mathfrak{X}}\Ind{U_{A}^{\geq k}}(X_A)\Theta^{|X_A|} \mathcal{B}(dX)}{\int_{\mathfrak{X}}\Ind{U_{A}^{=0}}(X_A) \mathcal{B}(dX)} = \sum_{\ell = k}^{\infty} \frac{(\bar{\lambda}_n|A|)^{ \ell}}{\ell!},
\end{equation}
where in the last equality we used the fact that under $\mathcal{B}$ the random variable $|X_A|$ is Poisson with parameter  $\beta|A|/(n-1)$. As $\p(U^{\geq 1}_A|V) + \p(U^{=0}_A|V)=1$, by elementary calculations we get
\[
	\p(U^{\geq 1}_A|V)\leq 1- e^{-|A| \bar{\lambda}_n}.
\]
Thus we obtain the upper bound in \eqref{eq:stochastic_domination_bounds}. The lower bound can be proven analogously. Now from \eqref{eq:stochastic_domination_bounds} and \eqref{eq:tmp_poisson} we infer \eqref{eq:stochastic_domination_bounds2}.
\end{proof}

The second lemma gives tail bounds on the number of bridges in terms of tails of Poisson variables with parameters depending on $\theta$ and $\beta$.

\begin{lemma}\label{le:auxilliary-poisson}
Let $X \in \mathfrak{X}$ be distributed according to $\mu_{\beta, \theta, \mathcal{C}}$. For any $k \in \N$, $t \in [0,1)$ and $s \in (0,1]$ such that $t + s \leq 1$ we have
\begin{align*}
	& \pr{|X \cap (E \times [t, t + s])| \geq k} \leq e^{s\beta \left( \Theta - \frac{1}{\Theta} \right) n^2} \pr{Y \geq k}, \\
	& \pr{|X \cap (E \times [t, t + s])| \leq k} \leq e^{s\beta \left( \Theta - \frac{1}{\Theta} \right) n^2} \pr{Y \leq k},
\end{align*}
where $Y$ is a Poisson variable with parameter $s \Theta \beta n^2$.

In particular for $k_n = o(n^2)$ we have for sufficiently large $n$
\[
\pr{|X \cap (E \times [t, t + s])| \leq k_n} \leq Ce^{- c n^2}
\]
for some $C,c > 0$ depending on $\beta$, $\theta$ and $s$.
\end{lemma}

\begin{proof}
Let $X_{t,t+s} = X \cap (E \times [t, t + s])$. Observe that by the Lipschitz condition \eqref{eq:Lipschitz} we have
\begin{equation}\label{eq:Lipschitz-slice}
\theta^{\mathcal{C}(X \setminus X_{t,t+s})} \Theta^{-|X_{t,t+s}|} \leq \theta^{\mathcal{C}(X)} \leq \theta^{\mathcal{C}(X \setminus X_{t,t+s})} \Theta^{|X_{t,t+s}|}.
\end{equation}
To prove the first estimate we write
\[
\Pp\left( |X_{t,t+s}| \geq k \right) =
\frac{\int\limits_{\mathfrak{X}} \ind{|X_{t,t+s}| \geq k} \theta^{\mathcal{C}(X)} \mathcal{B}(dX)}{\int\limits_{\mathfrak{X}} \theta^{\mathcal{C}(X)} \mathcal{B}(dX)}.
\]
By employing \eqref{eq:Lipschitz-slice} we can bound the right hand side from above by
\begin{align*}
&\frac{\int\limits_{\mathfrak{X}} \ind{|X_{t,t+s}| \geq k} \theta^{\mathcal{C}(X \setminus X_{t,t+s})} \Theta^{|X_{t,t+s}|} \mathcal{B}(dX)}{\int\limits_{\mathfrak{X}} \theta^{\mathcal{C}(X \setminus X_{t,t+s})} \Theta^{-|X_{t,t+s}|} \mathcal{B}(dX)} = \frac{\int\limits_{\mathfrak{X}} \ind{|X_{t,t+s}| \geq k}  \Theta^{|X_{t,t+s}|} \mathcal{B}(dX)}{\int\limits_{\mathfrak{X}} \Theta^{-|X_{t,t+s}|} \mathcal{B}(dX)},
\end{align*}
where in the equality we used the fact that the integrals factorize due to the independence property of the Poisson point process $\mathcal{B}$ for disjoint time intervals.

To estimate the integrals, we observe that under $\mathcal{B}$ the variable $|X_{t,t+s}|$ has Poisson distribution with parameter $\lambda = s \beta n^2 $. Thus we can write
\[
\int\limits_{\mathfrak{X}} \Theta^{-|X_{t,t+s}|} \mathcal{B}(dX) = \E \Theta^{- |X_{t,t+s}|} = e^{\lambda \left( \frac{1}{\Theta} - 1 \right)}.
\]
and
\begin{equation}\label{eq:poisson-upper-bound}
\int\limits_{\mathfrak{X}} \ind{|X_{t,t+s}| \geq k}  \Theta^{|X_{t,t+s}|} \mathcal{B}(dX) = \sum\limits_{i = k}^{\infty} e^{-\lambda} \frac{(\Theta\lambda)^{i}}{i!} = e^{\lambda (\Theta - 1)} \pr{Y \geq k}.
\end{equation}
where $Y$ is a Poisson variable with parameter $\Theta \lambda = s \Theta \beta n^2$. Thus we obtain
\[
\Pp\left( |X_{t,t+s}| \geq k \right) \leq e^{\lambda \left( \frac{1}{\Theta} - 1 \right)} e^{\lambda (\Theta - 1)} \pr{Y \geq k} = e^{\beta \left( \Theta - \frac{1}{\Theta} \right) s n^2} \pr{Y \geq k}.
\]

The proof of the second estimate is analogous.

For the case of $k_n = o(n^2)$ we use Bennett's inequality -- if $Y$ is a Poisson variable with parameter $\lambda$, then for any $0 \leq x \leq \lambda$ we have
\[
\pr{X \leq \lambda - x} \leq \exp\left\{ - \frac{x^2}{2 \lambda} \psi\left( - \frac{x}{\lambda} \right)\right\},
\]
where
\[
\psi(t) = \frac{(1+t)\log(1+t) - t}{t^2 / 2}
\]
for $t \geq -1$.

Writing $k_{n} = \varepsilon_n \lambda$, with $\lambda = s \beta \Theta n^2$ and $\varepsilon_{n} \to 0$ as $n \to \infty$, we have
\[
\pr{Y \leq \varepsilon_{n} \lambda} \leq e^{ - \lambda \left( \varepsilon_{n} \log \varepsilon_{n} + 1 - \varepsilon_{n} \right) }
\]
and thus
\[
\pr{|X \cap (E \times [t, t + s])| \leq k_n} \leq e^{\lambda - \frac{\lambda}{\Theta^2}} e^{ - \lambda \left( \varepsilon_{n} \log \varepsilon_{n} + 1 - \varepsilon_{n} \right) } = e^{- \lambda \left( \frac{1}{\Theta^2} - \varepsilon_n + \varepsilon_{n} \log \varepsilon_{n}  \right)}.
\]
Since $\varepsilon_n - \varepsilon_{n} \log \varepsilon_{n} \to 0$ as $n \to \infty$, the right hand side is at most $C e^{- c n^2}$ for some $C, c > 0$.
\end{proof}

The following lemma will be useful in the proof of Lemma \ref{le:intensity}.

\begin{lemma}\label{le:auxiliary-integrability}
As in the previous lemma, let $X_{t,t+h} = X \cap (E \times [t,t+h])$. For any $k \in \N$ and $t \in [0,1)$ we have
\begin{displaymath}
\int_\mathfrak{X} \ind{|X_{t,t+h}| \ge k} |X_{t,t+h}| \Theta^{\left|X_{t,t+h}\right|} \mathcal{B}(dX) = O(h^k),
\end{displaymath}
where the implicit constant may depend on $\Theta$, $\beta, n$ and $k$.
\end{lemma}

\begin{proof}
As under $\mathcal{B}$ the variable $|X_{t,t+h}|$ has Poisson distribution with parameter $h \beta n^2$, we have for small enough $h > 0$
\[
\int_\mathfrak{X} \ind{|X_{t,t+h}| \ge k} |X_{t,t+h}| \Theta^{\left|X_{t,t+h}\right|} \mathcal{B}(dX) = \sum\limits_{i = k}^{\infty} e^{- h \beta n^2} i \frac{(\Theta h \beta n^2)^{i}}{i!} \leq h^k e^{- h \beta n^2} \sum\limits_{i=k}^{\infty} \frac{(\Theta \beta n^2)^{i}}{(i-1)!} = O(h^k)
\]
\end{proof}

\end{section}

\begin{section}{Appendix -- Proof of Lemma \ref{le:intensity}}\label{ap:intensity}

\begin{proof}[Proof of Lemma \ref{le:intensity}]
Recall from the introduction the definition of the canonical probability space $(\mathfrak{X},\mathcal{S},\mathcal{B})$. Fix $A \in \Ff_s$ and define for $t > s$
\begin{displaymath}
f(t) = \E J_t \Ind{A}.
\end{displaymath}
Note that by Lebesgue's dominated convergence theorem and the fact that with probability one there are no jumps at a prescribed deterministic moment in time,  $f$ is a continuous function.

In what follows we will denote $t' = t \mod 1$.
Consider any $t_1,t_2$ such that $t_1 < t_2 < \lceil t_1 \rceil$ . Note that $f(t_2) - f(t_1)$ is bounded from above by the mean number of bridges in $E\times (t_1',t_2']$. Using \eqref{eq:Lipschitz} and independence properties of the Poisson process we thus get

\begin{align*}
  &|f(t_2) - f(t_1)| \le \frac{1}{Z_{\beta,\theta,\mathcal{C}}}\int_\mathfrak{X} |X \cap (E\times(t_1',t_2'])| \theta^{\mathcal{C}(X)} \mathcal{B}(dX)\\
& \le \frac{1}{Z_{\beta,\theta,\mathcal{C}}}\int_\mathfrak{X} |X \cap (E\times(t_1',t_2'])|\Theta^{|X \cap (E\times(t_1',t_2'])|}\Theta^{\mathcal{C}(\emptyset) + |X\cap (E\times(t_1',t_2']^c)|}\mathcal{B}(dX)\\
& = \frac{1}{Z_{\beta,\theta,\mathcal{C}}}\int_\mathfrak{X} \Theta^{\mathcal{C}(\emptyset) + |X\cap (E\times(t_1',t_2']^c)|}\mathcal{B}(dX) \int_\mathfrak{X} |X \cap (E\times(t_1',t_2'])\Theta^{|X \cap (E\times(t_1',t_2'])|} \mathcal{B}(dX)\\
&\le K \int_\mathfrak{X} |X \cap (E\times(t_1',t_2'])|\Theta^{|X \cap (E\times(t_1',t_2'])|} \mathcal{B}(dX)\\
&= K \Theta \beta n^2|t_2-t_1|e^{(\Theta-1)\beta n^2 |t_2-t_1|},
\end{align*}
where $K$ is some constant (depending on $n$ and the parameters of the process but not on $t_i$).
Thus $f$ is locally Lipschitz which implies that $f'$ exists almost everywhere and $f$ satisfies the fundamental theorem of calculus.

Consider any differentiability point $t > s$ of $f$  and small $h > 0$, in particular small enough so that $t'+h<1$. As with probability one there is no jump at time $t$, by using Lemma \ref{le:auxiliary-integrability}  and exploiting the properties of the process $Q$ (specifically the fact that it may jump only when $\crw$ jumps and that it is c\`adl\`ag) we can write for $h \searrow 0$
 \begin{align*}
f(t+h) - f(t)&=\E (J_{t+h} - J_t)\Ind{A} = \E\Ind{A}\ind{J_{t+h} - J_t = 1}\ind{|X\cap (E\times (t',t'+h])| =1} + o(h)\\
&= \sum_{v,w\in V} \p(A \cap \{\X_{t} = v\}  \cap B_{v,w,h} \cap \{w \in Q_t\}) + o(h),
\end{align*}
where
\begin{align*}
B_{v,w,h} = &\{ \textrm{there is a unique bridge in $E\times (t',t'+h]$, it is unexplored at time $t$}\\
& \textrm{ and joins $v$ with $w$}\}.
\end{align*}
Consider an additional event
\begin{displaymath}
C_{v,w,h} =  \{\{v,w\} \times [t',t'+h]\; \textrm{has not been visited by $\crw$ before time t}\}.
\end{displaymath}

Note that $B_{v,w,h} \cap \{\X_t = v\} \cap C_{v,w,h}^c = \emptyset$. Moreover $B_{v,w,h} \cap \{\X_t = v\} \subset \{w \in A_t\}$ (recall that $A_t$ is the set of vertices which at time $t$ are available to the CRW by a fresh jump). Thus, we have
\begin{align*}
\E (J_{t+h} &- J_t)\Ind{A} = \sum_{v,w \in V} \E \ind{\X_t = v,  w \in A_t\cap Q_t}  \Ind{A\cap B_{v,w,h}\cap C_{v,w,h}} +o(h)\\
                          &= Z_{\beta,\theta,\mathcal{C}}^{-1}\sum_{v,w \in V} \int_\mathfrak{X} \ind{\X_t = v,  w \in A_t\cap Q_t}  \Ind{A\cap B_{v,w,h}\cap C_{v,w,h}} \theta^{\mathcal{C}(X)}\mathcal{B}(dX) +o(h).
\end{align*}
Denote a summand above by $I_{v,w,h}$. Since for a while we will be working with fixed $v,w$ denote for simplicity $e = \{v,w\}$.

Note that for any $U \in \mathcal{F}_{t}$, the event $C_{v,w,h}\cap U$ is measurable with respect to the restricted process $X'_h := X\setminus (\{e\}\times (t',t'+h])$.

Denote also
\begin{displaymath}
D_{v,w,h} = \{ X \cap ((E\setminus \{e\})\times (t',t'+h]) = \emptyset\}.
\end{displaymath}

Recall that conditionally on having just one point of a Poisson process in an interval, its position is distributed uniformly. Combining this with the independence properties of Poisson processes we get
\begin{align*}
	& I_{v,w,h} \\
	& = \int_\mathfrak{X} \ind{\X_t = v,  w \in A_t\cap Q_t}  \Ind{A\cap D_{v,w,h}\cap C_{v,w,h}} \kappa e^{-\kappa h}  \int_{t'}^{t'+h} \theta^{\mathcal{C}(X'_h\cup \{(e,u)\})}du \, \mathcal{B}(dX),
\end{align*}
where $\kappa = \frac{\beta}{n-1}$.
Note that on $D_{v,w,h}\cap\{X\cap (E\times \{t'\}) = \emptyset\}$, the function $[t',t'+h] \ni u \mapsto \mathcal{C}((X'_h\cup \{(e,u)\})$ is constant, so using the fact that almost surely there are no bridges at height $t'$ we can further write
\begin{displaymath}
I_{v,w,h} = \int_\mathfrak{X} \ind{\X_t = v,  w \in A_t\cap Q_t}  \Ind{A\cap D_{v,w,h}\cap C_{v,w,h}} \kappa e^{-\kappa h}  h \theta^{\mathcal{C}(X'_h\cup \{(e,t')\})}\mathcal{B}(dX).
\end{displaymath}
By Lemma \ref{le:auxiliary-integrability} we also have
\begin{align*}
  \int_\mathfrak{X} \Ind{D_{v,w,h}^c} \theta^{\mathcal{C}(X'_h\cup \{(e,t')\})}\mathcal{B}(dX)
  \le \int_\mathfrak{X} \Ind{D_{v,w,h}^c} \Theta^{\mathcal{C}(\emptyset) + |X| + 1}\mathcal{B}(dX) = O(h),
\end{align*}
so we get
\begin{displaymath}
  I_{v,w,h} = \int_\mathfrak{X} \ind{\X_t = v,  w \in A_t\cap Q_t}  \Ind{A\cap C_{v,w,h}} \kappa e^{-\kappa h}  h \theta^{\mathcal{C}(X'_h\cup \{(e,t')\})}\mathcal{B}(dX) + o(h).
\end{displaymath}

Similarly, up to an error of order $o(h)$ we can restrict the integration to the set $\{X'_h = X\}$, replace $\theta^{\mathcal{C}(X'_h\cup \{(e,t')\})}$ by $\theta^{\mathcal{C}(X\cup\{(e,t')\})}$ and then again return to integration over the whole space $\mathfrak{X}$, obtaining
\begin{displaymath}
  I_{v,w,h} = \int_\mathfrak{X} \ind{\X_t = v,  w \in A_t\cap Q_t}  \Ind{A\cap C_{v,w,h}} \kappa e^{-\kappa h}  h \theta^{\mathcal{C}(X\cup\{(e,t')\})}\mathcal{B}(dX) + o(h).
\end{displaymath}

Note also that on the event $\{w \in A_t\}\cap\{\X_t = v\}$ neither $(v,t')$ nor $(w,t')$ could have been visited by $\crw$ before time $t$. Since with probability one there are only finitely many bridges, this implies that
up to a set of probability zero $\{w \in A_t\}\cap\{\X_t = v\} \cap C_{v,w,h} \nearrow \{w \in A_t\}\cap\{\X_t = v\}$ as $h \searrow 0$. Thus we get
\begin{displaymath}
  \lim_{h \to 0+} \frac{I_{v,w,h}}{h} =\kappa \int_\mathfrak{X} \ind{\X_t = v,  w \in A_t\cap Q_t}  \Ind{A} \theta^{\mathcal{C}(X\cup\{(e,t')\})} \mathcal{B}(dX),
\end{displaymath}
which implies that
\begin{displaymath}
  f'(t) = \E \Ind{A}\kappa \sum_{w \in A_t\cap Q_t} Y_t^w = \E \Ind{A} \kappa\sum_{w \in A_t\cap Q_t} \E(Y_t^w|\mathcal{F}_t),
\end{displaymath}
where $Y_t^w = \ind{\{\X_t, w\}\in E}\theta^{\mathcal{C}(X\cup \{(\{\X_t,w\},t')\}) - \mathcal{C}(X)}$.

Now for fixed $w$ we have that $Y_t^w \colon [0,\infty) \times \mathfrak{X} \to \R$ is measurable with respect to $Bor([0,\infty)) \otimes \mathcal{S}$, so by Corollary 2 in \cite{MR3037218} we obtain that there is a choice of
$\E(Y_t^w|\mathcal{F}_t)$ which as a stochastic process is $\mathcal{F}_t$-progressively measurable.
Set $S_t = \kappa\sum_{w \in A_t\cap Q_t} \E(Y_t^w|\mathcal{F}_t)$ and define the progressively measurable process $\lambda_t = (\kappa |A_t\cap Q_t|\Theta^{-1}) \vee ( S_t \wedge(\kappa |A_t\cap Q_t|\Theta))$. Note that by the Lipschitz condition \eqref{eq:Lipschitz} on $\mathcal{C}$, for every $t$ we have $S_t = \lambda_t$ almost surely.

Thus, by Fubini's theorem,  we have for $t > s$,
\begin{displaymath}
  \E(J_t - J_s)\Ind{A} = \int_s^t f'(u)du = \int_s^t \E \Ind{A} S_u du = \E \Ind{A} \int_0^t \lambda_u du   - \E \Ind{A} \int_0^s \lambda_u du,
\end{displaymath}
which proves that $J_t - \int_0^t \lambda_u du$ is indeed a martingale with respect to $(\mathcal{F}_t)_{t\ge 0}$.
\end{proof}

\end{section}

\begin{section}{Appendix -- Proof of Lemma \ref{lemma:sprinkle}}\label{ap:missing-proofs}

\begin{proof}[Proof of Lemma \ref{lemma:sprinkle}]

By Lemma \ref{lem:almost_uniform} $p_{u,e}$, the conditional probability of an edge $e$ being added to the graph process in the transition from  $u$ to $u+1$ belongs to $[\Theta^{-2}/|E|,\Theta^{2}/|E|]$. Let $\{U_{u,e}\}_{u\in \mathbb{N}, e\in E}$ be i.i.d. random variables uniformly distributed on $[0,1]$ which are also independent of $G^s$. We define a coupled random graph process $\tilde{G}^s$. First, we set $\tilde{G}^s_0 := {G}^s_0$, then in the transition from $u$ to $u+1$ an edge $e$ is added to the edge set of $\tilde{G}^s_u$ if and only if it is added to ${G}^s_{u}$ and $U_{u,e}\leq (\Theta^2|E| p_{u,e})^{-1}$. Note that in the new process at each step there is probability $1 - \Theta^{-2}$ of no new edge being added, and if a new edge is added, each one is chosen with probability $1/|E|$, independently of the previous steps.

\newcommand{\tvg}[2]{\tilde{\mathcal{G}}_{#1}(#2)}
Clearly, for any $s,u,\ell$ we have $\tvg{s,u}{\ell}\subset \vg{s, u}{\ell}$. As the process $\tilde{G}^s$ is monotonic it is enough to prove the statement for the process $\tilde{G}^s$ and $u = (n^2/ \sqrt{\ell})\log n$.

The proof is an implementation of the classical sprinkling argument, introduced in \cite{ajtai-komlos-szemeredi}. We will work conditionally on $\tilde{G}^s_0$. To shorten the notation we denote $\mathbb{Q}(\cdot) = \p\left(\cdot \big| \tilde{G}^s_0 , |\vg{s,0}{\ell}|\geq \delta n^2, |X|=h\right)$. Also, for any two sets $A,B \subset V$ by $E(A,B)$ we will denote the number of edges $\{v,w\} \in E$ such that $v \in A, w \in B$.

The event $\{ | \tvg{s,u}{\delta n^2/8} | < \delta n^2/8\}$ (i.e., there is no component of size at least $\delta n^2/8$ in $\tilde{G}^{s}_{u}$) implies that $\tvg{s,0}{\ell}$ can be partitioned into two sets $A$ and $B$ such that each of them has size at least $\delta n^2/4$, each of them is a union of some connected components of $\tilde{G}^{s}_{0}$, and there are no paths joining $A$ and $B$ in $\tilde{G}^s_u$. We will show that such a partition is unlikely to exist in $\tvg{s,0}{\ell}$.

Fix two sets $A$ and $B$ which partition $\tvg{s,0}{\ell}$ as above, each of size at least $\delta n^2 / 4$, and let $\mathcal{C}_{A,B}(G)$ be the event that no path in a graph $G \subset H_n$ has one endpoint in $A$ and the other endpoint in $B$. We write simply $\mathcal{C}_{A,B}$ for $\mathcal{C}_{A,B}(\tilde{G}^s_u)$. Let
\[
D_{A, B} := \{v \in V \colon E(\{v\},A) \geq \delta^2 n / 64 \text{ and } E(\{v\},B) \geq \delta^2 n / 64 \}
\]
be the set of vertices that in the Hamming graph $H_{n}$ have at least $\delta^2 n / 64$ neighbors both in $A$ and in $B$. Let $D_{1} := \{v \in D_{A,B} \colon v \in A \cup B\}$ and $D_{2} := D_{A,B} \setminus D_{1}$.

First we bound the size of $D_{A,B}$. Note that there are at least $\delta^2 n^4 /16$ paths of length $2$ in $H_{n}$ between $A$ and $B$, since in $H_{n}$ all vertices are connected by a path of length at most $2$ and we assumed $|A|, |B| \geq \delta n^2 / 4$. On the other hand, for every $v \notin D_{A,B}$, there are at most $\delta^2 n / 64 \cdot 2 (n-1) \leq \delta^2 n^2 / 32$ paths of length $2$ between $A$ and $B$ with $v$ as the midpoint. Every $v \in D_{A,B}$ can be a midpoint in at most $ 4 n^2$ paths of length $2$ between $A$ and $B$.

Hence the total number of paths of length $2$ between $A$ and $B$ is bounded from above by $(\delta^2 n^2/32)|D_{A,B}^c|+4n^2|D_{A,B}|$.
Combining this with the lower bound we get that
\[
	\frac{\delta^2 n^2}{32}(n^2 - |D_{A,B}|) + 4n^2 |D_{A,B}| \geq \frac{\delta^2 n^4}{16}
\]
and thus there exists a constant $\rho > 0$ (depending only on $\delta$) such that $|D_{A,B}|\geq \rho n^2$.

Throughout the rest of the proof it will be convenient to work with a random graph which has edges chosen independently. Since in the graph process $\tilde{G}^{s}$ at each step with probability $\Theta^{-2}$ a uniformly random edge is added, the graph $\tilde{G}^{s}_{u}$ is obtained by adding $k$ uniformly random edges (with multiple edges allowed) to $\tilde{G}^{s}_{0}$, where $k$ has binomial distribution corresponding to $u$ trials with success probability $\Theta^{-2}$. Now let $\bar{G}_{s,u}$ be a random graph obtained by adding each edge $e \in E$ to $\tilde{G}^{s}_{0}$ independently with probability $p = \frac{u \Theta^{-2}}{2|E|}$.

Let $\bar{\mathbb{Q}}(P)$ denote the probability that the graph $\bar{G}_{s,u}$ satisfies property $P$. By using the second moment method one can see that with high probability after removing multiple edges from $\tilde{G}^{s}_{u}$ we will still have (for $n$ large enough) at least $\frac{u \Theta^{-2}}{2}$ edges in the graph, distributed uniformly. Indeed, as $u$ is small compared to $|E|$, the expected number of edges chosen at least once in $\tilde{G}^{s}_{u}$ is at least, say, $\frac{3}{4} u \Theta^{-2}$. As the edges present in $\tilde{G}^{s}_{u}$ are negatively correlated, the variance can be bounded from above by $u \Theta^{-2}$, which implies that with high probability we have at least $\frac{u \Theta^{-2}}{2}$ distinct edges. Therefore by the equivalence of $G(n,p)$ and $G(n,M)$ random graph models with respect to monotone properties (see e.g., \cite[Section 1.4]{janson2011random}) $\bar{\mathbb{Q}}(P) \to 0$ as $n \to \infty$ will imply $\mathbb{Q}(P) \to 0$ for any decreasing graph property $P$.

From now on we will work with the graph $\bar{G}_{s,u}$. Let $\bar{\mathcal{G}}_{s,u}(k)$ denote the set of vertices of $\bar{G}_{s,u}$ contained in connected components of size at least $k$. Let
\[
E_{1} = \{\{v,w\} \in E \colon v \in D_{1}, \, \{v,w\} \in E(A,B)\}
\]
and let $\mathcal{C}(E_1)$ denote the event that none of the edges from $E_1$ are in $\bar{G}_{s,u}$. Let $\mathcal{C}(D_2)$ denote the event that none of the vertices in $D_2$ have neighbors both in $A$ and in $B$ in $\bar{G}_{s,u}$. Clearly we have
\begin{equation}\label{eq:minimum-prob}
\bar{\mathbb{Q}}(\mathcal{C}_{A,B}(\bar{G}_{s,u})) \leq \bar{\mathbb{Q}}(\mathcal{C}(E_1) \cap \mathcal{C}(D_2) ) = \bar{\mathbb{Q}}(\mathcal{C}(E_1)) \bar{\mathbb{Q}}(\mathcal{C}(D_2)).
\end{equation}
We first estimate $\bar{\mathbb{Q}}(\mathcal{C}(E_1))$. Because of independence of the edges in $\bar{G}_{s,u}$ we have
\[
\bar{\mathbb{Q}}(\mathcal{C}(E_1)) \leq (1-p)^{|E_1|}.
\]
Since $|E(\{v\}, B)| \geq \delta^2 n /64$ for each $v \in D_1 \cap A$ and likewise $|E(\{v\}, A)| \geq \delta^2 n /64$ for each $v \in D_1 \cap B$, we easily get $|E_1| \geq \frac{1}{2} |D_1| \delta^2 n / 64$, so
\[
\bar{\mathbb{Q}}(\mathcal{C}(E_1)) \leq (1-p)^{|D_1| \delta^2 n / 128} \leq e^{- \frac{p |D_1| \delta^2 n}{128} }.
\]
For the upper bound on $\bar{\mathbb{Q}}(\mathcal{C}(D_2))$, we note that by independence of the edges in $\bar{G}_{s,u}$
\begin{align*}
    \bar{\mathbb{Q}}(\mathcal{C}(D_2)) &= \prod_{v \in D_2}\left(1 - \bar{\mathbb{Q}}(v\text{ has neighbors both in }A\text{ and }B)\right)\nonumber\\
    &=  \left(1 - \bar{\mathbb{Q}}(v\text{ has a neighbor in }A) \bar{\mathbb{Q}}(v\text{ has a neighbor in }B)\right)^{|D_2|}.\nonumber
  \end{align*}
Since $|E(\{v\}, A)| \geq \delta^2 n /64$ for $v \in D_2$, we have
\[
	\bar{\mathbb{Q}}(v\text{ has a neighbor in }A) \geq 1 - (1 - p)^{\delta^2 n /64} \geq 1 - e^{- \frac{p \delta^2 n}{64}}.
\]
An analogous estimate holds for $B$, which gives
\[
\bar{\mathbb{Q}}(\mathcal{C}(D_2)) \leq \left(1 - \left(1 - e^{- \frac{p \delta^2 n}{64}} \right)^2 \right)^{|D_2|}.
\]
Recall that $p = \frac{u}{2\Theta^2 |E|}$ and $|E| = n^2(n-1)$. Since $|D_{A,B}| \geq \rho n^2$, we have $|D_1| \geq \frac{\rho}{2} n^2$ or $|D_2| \geq \frac{\rho}{2} n^2$. In the first case we get
\[
\bar{\mathbb{Q}}(\mathcal{C}(E_1)) \leq e^{- \frac{\delta^2}{128}  \frac{u}{2\Theta^2 |E|} \frac{\rho n^2}{2} n } \leq e^{- cu}
\]
for some $c > 0$ depending on $\delta$ and $\Theta$. In the second case we have (exploiting $u \leq n^2$)
\[
\bar{\mathbb{Q}}(\mathcal{C}(D_2)) \leq \left(1 - \left(1 - e^{- \frac{\delta^2}{64} n \frac{u}{2\Theta^2 |E|}} \right)^2 \right)^{\frac{\rho}{2} n^2} \leq C e^{-c' \frac{u^2}{n^2}}
\]
for some $C,c' >0$ depending on $\delta$ and $\Theta$.

As $u = (n^2 / \sqrt{\ell}) \log n$ and $\log n \leq \sqrt{\ell}$, we have $u^2 / n^2 = (n^2 / \ell) \log^2 n \leq u$. Coming back to \eqref{eq:minimum-prob}, we obtain for some $C, c >0$
\[
	\bar{\mathbb{Q}}(\mathcal{C}_{A,B}(\bar{G}_{s,u})) \leq C e^{-c \frac{n^2}{\ell} \log^2 n}.
\]
Let $\mathcal{C}(G)$ denote the event that $\mathcal{C}_{A,B}(G)$ holds for some partition $A,B$ of the set $\vg{s,0}{\ell}$. Notice that there are at most $2^{n^2/\ell}$ such partitions, so by performing a union bound we obtain
\[
	\bar{\mathbb{Q}}(\mathcal{C}(\bar{G}_{s,u})) \leq 2^{\frac{n^2}{\ell}} \cdot C e^{-c \frac{n^2}{\ell} \log^2 n}.
\]
Recalling that $\{ |\bar{\mathcal{G}}_{s,u}(\delta n^2/8)| < \delta n^2/8\} \subset \mathcal{C}(\bar{G}_{s,u})$, we have that there exist constants $C_1, c_1 > 0$ (depending only on $\delta$, $\Theta$ and $u$) such that
\[
\bar{\mathbb{Q}}(|\bar{\mathcal{G}}_{s,u}(\delta n^2/8)| < \delta n^2/8) \leq C_1 \exp\left\{\frac{n^2}{\ell}\log 2 - c_1\frac{n^2 \log^2 n}{\ell}\right\}.
\]
As $\ell\leq n^2$ and $n \to\infty$, we have that the probability above converges to zero. Since the property of having a component of size at least $\delta n^2 / 8$ is increasing, the same holds with $\bar{\mathbb{Q}}(\cdot)$ and $\bar{\mathcal{G}}_{s,u}$ replaced by $\mathbb{Q}(\cdot)$ and $\tilde{\mathcal{G}}_{s,u}$. Thus we have
\[
\mathbb{Q}(|\vg{s, u}{\delta n^2/8}| < \delta n^2/8) = \p\left(|\vg{s, u}{\delta n^2/8}| < \delta n^2/8 \big| G^s_0 , |\vg{s,0}{\ell}|\geq \delta n^2, |X|=h\right)  \leq 1 - a_n
\]
for some $a_n \nearrow 1$. By integrating this bound over all $G^{s}_{0}$ satisfying $|\vg{s,0}{\ell}|\geq \delta n^2$ we obtain the statement of the lemma.
\end{proof}

\end{section}

\bibliography{bibliography}{}
\bibliographystyle{amsalpha}

\paragraph{E-mails:}
R. Adamczak: \href{mailto:r.adamczak@mimuw.edu.pl}{r.adamczak@mimuw.edu.pl}, M. Kotowski: \href{mailto:michal.kotowski@mimuw.edu.pl}{michal.kotowski@mimuw.edu.pl}, P. Miłoś: \href{mailto:pmilos@mimuw.edu.pl}{pmilos@mimuw.edu.pl}

\end{document}